\documentclass[a4paper]{amsart}
\usepackage{amsmath} 
\usepackage{appendix}
\usepackage{mathrsfs}
\usepackage{booktabs}
\usepackage{enumitem}
\usepackage{calligra}
\usepackage{stmaryrd}
\usepackage{verbatim}
\usepackage{amscd,latexsym,amsthm,amsfonts,amssymb,amsxtra}
\usepackage{tikz-cd}
\usepackage{changepage}
\usepackage[mathmode,boxsize=2em]{ytableau}
\usepackage{scalerel}
\usepackage{indentfirst}
\usepackage[T1]{fontenc}
\usepackage[utf8]{inputenc}
\usepackage{microtype}
\usepackage{lmodern}

\usepackage[a4paper,left=2cm, right=2cm,bottom=2.5cm, top=3cm]{geometry}

\usepackage{hyperref}
\usepackage[capitalize]{cleveref}
\hypersetup{
  linktoc=page,
	colorlinks = true,
	linkcolor = [rgb]{0,0.6,0}, 
	filecolor = cyan,
	urlcolor = cyan, 
	citecolor = [rgb]{0.58,0,0.82}, 
}

\theoremstyle{plain}
\newtheorem{thm}{Theorem}[section] 
\newtheorem*{thm*}{Theorem}
\newtheorem{cor}[thm]{Corollary}
\newtheorem{lem}[thm]{Lemma}  
\newtheorem{prop}[thm]{Proposition}
\newtheorem{conj}[thm]{Conjecture} 
\theoremstyle{definition}
\newtheorem{defn}[thm]{Definition}
\newtheorem{rek}[thm]{Remark}
\newtheorem{exe}[thm]{Example}

\makeatletter
    \let\c@equation\c@thm
\makeatother
\numberwithin{equation}{section}

\newcommand{\q}[1]{\left( #1 \right)}

\newcommand{\scrf}{\mathscr{F}}
\newcommand{\scr}[1]{\mathscr{#1}}
\newcommand{\bb}[1]{\mathbb{#1}}
\newcommand{\rr}[1]{\mathrm{#1}}
\newcommand{\cc}[1]{\mathcal{#1}}

\newcommand{\Het}[1]{\mathrm{H}^{#1}_{\acute{e}t}}
\newcommand{\Hetc}[1]{\mathrm{H}^{#1}_{\acute{e}t,c}}
\newcommand{\inv}{^{-1}}
\newcommand{\cros}{^{\times}}

\let\emptyset\varnothing
\let\bar\overline
\let\tilde\widetilde
\let\hat\widehat
\let\oplus\bigoplus

\newcommand{\dR}{\mathrm{dR}}
\newcommand{\frob}{\mathrm{Frob}}
\DeclareMathOperator{\FT}{FT}

\newcommand{\gal}{\mathrm{Gal}}

\newcommand{\GL}{\mathrm{GL}}

\newcommand{\im}{\mathrm{im}}
\newcommand{\id}{\mathrm{id}}

\newcommand{\Kl}{\mathrm{Kl}}

\newcommand{\pr}{\mathrm{pr}}
\newcommand{\rk}{\mathrm{rk}}

\DeclareMathOperator{\Spec}{Spec}
\newcommand{\Sym}{\mathrm{Sym}}
\newcommand{\SL}{\mathrm{SL}}
\newcommand{\sign}{\mathrm{sign}}
\newcommand{\sw}{\mathrm{Sw}}

\newcommand{\tr}{\mathrm{Tr}}

\newcommand{\bck}{\bar{\mathcal{K}}}
\newcommand{\WD}{\mathrm{WD}}

\newcommand{\bql}{\bar{\mathbb{Q}}_\ell}
\newcommand{\bq}{\bar{\mathbb{Q}}}
\newcommand{\Qp}{\mathbb{Q}_p}

\newcommand{\bfp}{\bar{\mathbb{F}}_p}
\newcommand{\fpp}{\mathbb{F}_p}
\newcommand{\fqq}{\mathbb{F}_q}

\newcommand{\piet}{\pi_1^{\acute{e}t}}
\newcommand{\et}{\acute{e}t}

\newcommand{\gr}{\mathrm{gr}}

\begin{document}
\title{\texorpdfstring{$\mathrm{L}$}{L}-functions of Kloosterman sheaves}

\author{Yichen Qin}
\date{}

\begin{abstract}
	In this article, we study a family of motives $\mathrm{M}_{n+1}^k$ associated with the symmetric power of Kloosterman sheaves constructed by Fresán, Sabbah, and Yu. They demonstrated that for $n=1$ the $L$-functions of $\mathrm{M}_{2}^k$ extend meromorphically to $\bb{C}$ and satisfy the functional equations conjectured by Broadhurst and Roberts. Our work aims to extend these results to the $L$-functions of some of the motives $\mathrm{M}_{n+1}^k$, with $n>1$, as well as other related two-dimensional motives. In particular, we prove several conjectures of Evans type, which relate moments of Kloosterman sheaves and Fourier coefficients of modular forms. 
\end{abstract}

\maketitle

\tableofcontents

\section{Introduction}

The \textit{Kloosterman sums} are exponential sums over finite fields, defined for each power of prime numbers $q=p^r$ and each $a\in \fqq$, by 
	\begin{equation*}
		\begin{split}
			\Kl_2(a;q):=\sum_{x\in \fqq\cros}\exp\Bigl(2\pi i/p\cdot   \tr_{\fqq/\fpp}\left(x+\frac{a}{x}\right)\Bigr),
		\end{split}
	\end{equation*}
where  $\tr_{\fqq/\fpp}$ is the trace from $\fqq$ to $\fpp$. These sums can be regarded as finite field versions of Bessel functions,  
	$$\rr{Be}(z):=\oint_{S^1}\exp\left(x+\frac{z}{x}\right)\frac{\mathrm{d}x}{x},$$
which satisfy the Bessel differential equations $(z\partial_z)^{2}-z=0$. 

When $a\neq 0$, Weil showed in \cite{weil1948some} that $\Kl_2(a;q)=-(\alpha_a+\beta_a)$ for some algebraic numbers $\alpha_a,\beta_a$ of complex norm $q^{1/2}$. For $k\geq 1$, the \textit{$k$-th symmetric power moments of Kloosterman sums} are integers $m_2^k(q)$ defined by 
	$$m_2^k(q)=\sum_{a\in \fqq} \sum_{i=0}^k \alpha_a^i\beta_a^{k-i}.$$
To package the information of these moments as $q$ varies across all powers of $p$, we consider the generating series	$$\exp\Biggl(\sum_{r\geq 1} \frac{m_{2}^k(p^r)}{r}T^r\Biggr),$$
which serves as the analog of the Hasse--Weil zeta function for varieties over finite fields. 

We define the \textit{(partial) $L$-function attached to $k$-th symmetric power moments of Kloosterman sheaves}, denoted by $L_k^S(s)$, by considering the Euler product, where the local factors at $p$ are made from the aforementioned generating series. These $L$-functions are a priori defined on the domain $\{s\in \bb{C}\mid \mathrm{Re}(s)> 1+\frac{k+1}{2}\}$ by construction and the work of Fu--Wan \cite{fu2005functions}. Hence, it is natural to question whether this $L$-function can be extended meromorphically to the complex plane and whether it satisfies a functional equation.

\begin{exe}\label{exe:evans-conj}
	The cases for $k \leq 8$ have been proven indirectly by demonstrating that the expressions of moments of Kloosterman sums consist of polynomials in $p$, Dirichlet characters, and Fourier coefficients of modular forms (holomorphic cuspidal Hecke eigenforms). 
\begin{itemize}
	\item When $k\leq 4$, the moments $m_2^k(p)$ can be computed explicitly. We find that the $L$-function is trivial if $k=1,2,$ or $4$, and is the Dirichlet $L$-function $L\left(s,(\frac{\bullet}{3})\right)$ if $k=3$.
	\item When $k=5$, there exists a holomorphic cuspidal Hecke eigenform $f\in S_3\bigl(\Gamma_0(15),\left(\frac{\cdot}{15}\right) \bigr)$ such that 
		$$a_f(p)=-\frac{1}{p^2}(m_2^5(p)+1)$$ 
   if $p\nmid 15$, proved by Peters et al. \cite{peters1992hasse} and Livné \cite{livne1995motivic}. 
   \item  When $k=6$, there exists a holomorphic cuspidal Heck eigenform $f\in S_4(\Gamma_0(6))$ such that 
	   $$a_f(p)=-\frac{1}{p^2}(m_2^6(p)+1)$$ 
   if $p\nmid 6$, proved by Hulek et al. \cite{Hulek2001a}.
	 \item When $k=7$, there exists a holomorphic cuspidal Hecke eigenform $f\in S_3(\Gamma_0(525),\epsilon_f)$, where $\epsilon_f=\left(\frac{\cdot}{21}\right)\cdot \epsilon_5$ and $\epsilon_5$ is a quartic character with conductor $5$, such that
	   $$a_f(p)^2\epsilon_f(p)\inv -p^2=-\frac{1}{p^2}\left(\frac{p}{105}\right)(m_2^7(p)+1)$$
   for $p>7$, conjectured by Evans \cite{Evans2010a} and proved by Yun \cite{yun2015galois}.
	  \item When $k=8$, there exists a holomorphic cuspidal Hecke eigenform $f\in S_6(\Gamma_0(6))$, such that
	  $$a_f(p)=-\frac{1}{p^2}(m_2^8(p)+1)$$
	  for $p\nmid 6$, conjectured by Evans \cite{Evans2010} and proved by Yun and Vincent \cite{yun2015galois}.
\end{itemize}
\end{exe}

From the examples discussed, we deduce that $L^S_k( s)$ can be extended meromorphically to $\mathbb{C}$ and satisfies a functional equation when $k \leq 8$. For general $k$, Broadhurst and Roberts predicted precise formulas for the functional equations of $L^S_{k}(s)$ in \cite{Broadhurst2016a, Broadhurst2017}. And then, Fresán--Sabbah--Yu established the following theorem:

\begin{thm}[Fresán--Sabbah--Yu]\label{thm::FSY}
	The partial $L$-function $L^S_k(s)$ can be extended meromorphically to the complex plane. Furthermore, we can complete $L_k^S(s)$ to a holomorphic function $\Lambda_k(s)$ such that 
		$$\Lambda_k(s)=\epsilon_k\Lambda_k(k+2-s),$$ 
	where $\epsilon_k\in \{\pm 1\}$ and $\epsilon_k$ is $1$ if $k$ is odd.
\end{thm}

The primary object of this article is to extend the theorem above to $L$-functions attached to moments (beyond symmetric power moments) of Kloosterman sums in multiple variables.

\subsection{Kloosterman sheaves}
The Kloosterman sums in $n$ variables are the exponential sums over finite fields, defined for each power of prime numbers $q=p^r$ and each $a\in \fqq\cros$, by
	$$\mathrm{Kl}_{n+1}(a;q):=\sum_{x_1,\ldots,x_n\in \fqq^{\times}}\exp\biggl(\frac{2\pi i}{p}\cdot \tr_{\fqq/\fpp}\biggl(x_1+\ldots+x_n+\frac{a}{x_1\cdots x_n}\biggr)\biggr).$$ 
By fixing a prime number $\ell\neq p$ and an embedding $\iota\colon \bql\to \bb{C}$, Deligne constructed lisse $\ell$-adic sheaves $\Kl_{n+1}$ over $\bb{G}_{m,\fqq}=\bb{A}^1_{\fqq}\backslash \{0\}$, which are pure of weight $n$ and of rank $n+1$ in \cite[Sommes.\,Trig.\,Thm.\,7.8]{SGA41/2}. Moreover, for every $a\in \fqq\cros=\bb{G}_{m,\fqq}(\fqq)$ and every geometric point $\bar{a}$ localized at $a$, we have
	$$\iota\circ \tr(\frob_q,(\Kl_{n+1})_{\bar a})=(-1)^n\Kl_{n+1}(a;q).$$
Hence, the $\ell$-adic sheaves $\Kl_{n+1}$ can be regarded as the sheaf version of the Kloosterman sums, and we call them \textit{Kloosterman sheaves}. 

In their work \cite{HNY-Kloosterman}, Heinloth, Ngô, and Yun constructed a larger class of $\ell$-adic sheaves, called \textit{Kloosterman sheaves for reductive groups}, using methods from the geometric Langlands program. For each split reductive group $G$, they construct a tensor functor
\begin{equation}\label{eq:HNY-functor}
	\Kl_{G}\colon \rr{Rep}(G)\to \rr{Loc}_{\bb{G}_{m,\fqq}}
\end{equation}
from the category of finite-dimensional representations of $G$ with coefficients in $\bb{Q}_{\ell}(\mu_p)$ to the category of lisse $\ell$-adic sheaves on $\bb{G}_{m,\fqq}$. Our primary interest lies in the case where $G = \mathrm{SL}_{n+1}$. In particular, by selecting $V$ as the standard representation $\rr{Std}$ of $\rr{SL}_{n+1}$ and $\Sym^k\rr{Std}$ respectively, we obtain the classical Kloosterman sheaf $\Kl_{n+1}(\frac{n}{2})$ and its symmetric power $\Sym^k\Kl_{n+1}(\frac{nk}{2})$.

Let $V = V_\lambda$ be the representation of the highest weight $\lambda = (\lambda_1, \ldots, \lambda_n)$ of $\mathrm{SL}_{n+1}$. We denote $|\lambda| = \sum_{i=1}^n  \lambda_i$ and $\Kl_{n+1}^\lambda$ as the sheaf $\Kl_{\mathrm{SL}_{n+1}}(V_\lambda)(-\frac{n|\lambda|}{2})$. We have an explicit description of $\Kl_{n+1}^\lambda$ using Weyl's construction, detained in Section~\ref{sec::weyl-construction}. In what follows, we formulate the analogs of moments and $L$-functions for $\Kl_{n+1}^\lambda$.

\begin{defn}\label{defn::moments}
	For each $\lambda$, the \textit{moment of the Kloosterman sheaf $\Kl_{n+1}^\lambda$} is defined as the integer 
		$$m_{n+1}^\lambda(q):=-\sum_{a\in \fqq\cros}\tr(\frob_q,(\Kl_{n+1}^\lambda)_{\bar a}).$$
\end{defn}
By the Grothendieck trace formula \cite[Rapport.\,Thm.\,3.1]{SGA41/2} and Theorem~\ref{thm::unramified}, the generating series 
		$$Z(\lambda,n+1,p;T):=\exp\Biggl(\sum_{r\geq 1}\frac{m_{n+1}^\lambda(p^r)}{r}\cdot T^r\Biggr).$$ 
	is a rational function
		$$\prod_{i=0}^2\det(1-\frob_pT\mid \mathrm{H}^i_{\et,\rr{c}}(\bb{G}_{m,\bfp},\Kl_{n+1}^\lambda))^{(-1)^{i+1}}\in \bb{Q}(T).$$
	In order to define the partial $L$-function associated with $\Kl_{n+1}^\lambda$ as an Euler product, it is not advisable to directly use $Z(\lambda,n+1,p; T)$ as the local factor at $p$, because the complex norms of roots and poles of $Z(\lambda,n+1,p; T)$ lie within the set $\{p^{-i/2}\mid 0\leq i\leq n|\lambda|+1\}$. Motivated by the work of Fu--Wan \cite{fu2005functions,fu2006trivial} for the sheaves $\Sym^k\Kl_{n+1}$, we remove some "trivial factors" from $Z(\lambda,n+1,p;T)$. By the long exact sequence \eqref{eq::long-exact-sequence} and the main theorem of Weil II \cite[3.3.1]{deligne1980conjecture}, we need to discard the contributions from the invariants and coinvariants of the Kloosterman sheaves at $0$ and $\infty$. Hence, the ideal candidate for the local factors at $p$ is
		$$M(\lambda,n+1,p;T):=\det(1-\frob_pT\mid \mathrm{H}^1_{\et,\mathrm{mid}}(\bb{G}_{m,\bfp},\Kl_{n+1}^\lambda)),$$
	where
		$$\mathrm{H}^1_{\et,\mathrm{mid}}(\bb{G}_{m,\bfp},\Kl_{n+1}^\lambda)
		=\im  \Bigl(\mathrm{H}^1_{\et,\rr{c}}(\bb{G}_{m,\bfp},\Kl_{n+1}^\lambda)\xrightarrow{\text{forget support}} \mathrm{H}^1_{\et}(\bb{G}_{m,\bfp},\Kl_{n+1}^\lambda) \Bigr)$$
	is the middle $\ell$-adic cohomology of $\Kl_{n+1}^\lambda$.

	\begin{defn}\label{defn::l-function}
		The partial $L$-function $L^S(\lambda,n+1;s)$ attached to $\Kl_{n+1}^\lambda$ is defined as the Euler product
			$$L^S(\lambda,n+1;s):=\prod_{s\not\in S(\lambda,n+1)}M(\lambda,n+1,p;p^{-s})\inv.$$
		Here, the set $S(\lambda,n+1)$ is a finite set of primes, only depending on $\lambda$ and $n+1$ (see Theorem~\ref{thm::unramified}) such that the degree of $M(\lambda,n+1,p;T)$ remains constant for $p\not\in S(\lambda,n+1)$. 
	\end{defn}
	The $L$-function is a priori a holomorphic function on the domain $\{s\in \bb{C}\mid \mathrm{Re}(s)>1+\frac{n|\lambda|+1}{2}\}$, because the complex norms of the roots of $M(\lambda,n+1,p;T)$ are $p^{-\frac{n|\lambda|+1}{2}}$. However, the definition alone does not provide further information. We can ask, as before, whether the partial $L$-function $L^S(\lambda,n+1;s)$ can be meromorphically extended to the entire complex plane and satisfies a functional equation.

\subsection{Main results}

 We introduce our main results here. For simplicity, when $\lambda=(k,0,\dots,0)$, we write $L^S(k,n+1;s)$ instead of $L^S((k,0,\dots,0),n+1;s)$.
\begin{thm}\label{thm::L-function}
For the values of $(n+1, k)$ given in the table below,
	\begin{center}
		\begin{tabular}{@{\extracolsep{4pt}}ll cccccc c}
			\toprule
			$n+1$ & $k$\\ 
			\midrule
			$3$ & $1,2,3,4,5,6,7,8,9$\\
			$5$ & $1,2,3,4$\\
			$4,7,8,10,11,13$ & $1,2,3$\\
			\bottomrule
		\end{tabular}
	\end{center}
the partial $L$-function $L^S(k,n+1;s)$ extends meromorphically to the complex plane. Furthermore, it can be completed into a holomorphic function $\Lambda(k,n+1;s)$ satisfying a functional equation
	$$\Lambda(k,n+1;s)=\pm \Lambda(k,n+1;nk+2-s).$$
\end{thm}

In Example~\ref{exe:evans-conj}, we see that there are some relations between Fourier coefficients of certain explicitly determined modular forms and symmetric power moments of Kloosterman sums. Yun proposed \textit{Conjectures of Evans type} in \cite{yun2015galois}, predicting new relations for Kloosterman sheaves for reductive groups.  For simplicity, throughout this article, a modular form will refer to a normalized holomorphic cuspidal Hecke eigenform.

These relations imply that the $L$-functions of these sheaves are $L$-functions of the corresponding modular forms.
\begin{thm}\label{thm::modular form}
The $L$-functions of the Kloosterman sheaves $\Sym^4\Kl_3$, $\Sym^3\Kl_4$, $\Sym^4\Kl_4$, $\Sym^3\Kl_5$, $\Kl_{3}^{\q{2,1}}$, and $\Kl_{3}^{\q{2,2}}$ arise from modular forms. Moreover, we determine explicitly these modular forms and the relations between their Fourier coefficients and moments of Kloosterman sheaves.
 
The information of these modular forms $f\in S_k(\Gamma_0(N),\epsilon)$ are summarized in the following table. 
		\begin{center}
			\begin{tabular}{@{\extracolsep{4pt}}ll cccccc c}
				\toprule
				Sheaves & $N$ & $k$ & $\epsilon$ & labels in LMFDB \cite{lmfdb}\\ 
				\midrule
				$\Sym^4\Kl_3$ & $14$ & $4$ & $1$ & 14.4.a.b\\
				$\Sym^3\Kl_4$ & $15$ & $3$ &  $\q{\frac{\bullet}{15}}$ & 15.3.a.b.\\
				$\Sym^4\Kl_4$ & $10$ & $6$ &  $1$ & 10.6.a.a.\\
				$\Sym^3\Kl_5$ & $33$ & $4$ & $1$ & $33.4.a.b$\\
				$\Kl_{3}^{(2,1)}$ & $14$ & $2$  & $1$ & 14.2.a.a\\
				$\Kl_{3}^{(2,2)}$ & $6$ & $4$ & $1$ &6.4.a.a \\ 
				\bottomrule
				\end{tabular}
		\end{center}
\end{thm}

Alongside establishing the main theorems, we have also successfully proved several new results about the Kloosterman Sheaves. For example, we calculated the local monodromy group of $\Kl_3$ at $\infty$ when $p=3$ in \cref{thm::Kl_3_monodromy_infty_p=3}. When $n\geq 3$, the local monodromy group of $\Kl_{n+1}$ at $\infty$ when $p\mid n+1$ is still unknown.

Furthermore, we observe that the modular forms linked to the moments of the sheaves $\Sym^6\Kl_2$ and $\Kl_3^{(2,2)}$ are identical, with label $6.4.a.a$ in LMFDB, thanks to \cref{thm::modular form} and \cite{Hulek2001a}. In particular, we deduce an identity between moments of $\Sym^6\Kl_2$ and $\Kl_3^{(2,2)}$ in \eqref{eq:identical-moments}. This prompts us to ask whether there exists a geometric explanation for this phenomenon, as conjectured in \cref{conj:identical-motives}.

\subsection{Idea of the proof}
Our strategy in proving \cref{thm::L-function} and \cref{thm::modular form} is as follows. We begin with constructing families of Galois representations of geometric origin, whose $L$-functions precisely match $L(\lambda, n+1; s)$, extending the construction in \cite[(3.1)]{fresan2018hodge}. Then, we subtract geometric information from these families of Galois representations to be able to apply some theorems from the automorphic side. Once we establish that these Galois representations are potentially automorphic, the $L$-functions $L(\lambda, n+1; s)$ extend meromorphically to $\bb{C}$ and satisfy functional equations as a result. At last, for \cref{thm::modular form}, one needs extra numerical results to locate the modular forms in LMFDB.

\subsubsection{Galois representations arising from geometry}
Drawing inspiration from the analogy between Kloosterman sums and Bessel functions, Fresán, Sabbah, and Yu considered the \textit{Kloosterman connection}, which is the rank $n+1$ connection on $\bb{G}_{m,\bb{C}}$ corresponding to the Bessel differential equation $(z\partial_z)^{n+1}-z=0$. They interpret the middle de Rham cohomology of the connection $\Sym^k\Kl_{n+1}$, i.e., the image of the forget supports morphism from the cohomology with compact support to the usual cohomology, as the de Rham realization of an exponential motive over $\bb{Q}$ in the sense of \cite{fresanexponential}. This exponential motive is classical, meaning that it is isomorphic to a Nori motive $\mathrm{M}_{n+1}^k$ over $\bb{Q}$. 

This motive is isomorphic to a subquotient of $\mathrm{H}^{nk-1}_c(\mathcal{K})(-1)$, where $\mathcal{K}$ is the hypersurface defined by the $k$-th iterated Thom--Sebastiani sum of the Laurent polynomial $g_{n+1}=\sum_{j=1}^n y_{j}+ 1/\prod_{j=1}^n y_{j}$, which is the Laurent polynomial
	\begin{equation}\label{eq:gbox}
		g^{\boxplus k}_{n+1}=\sum_{i=1}^k\Biggl(\sum_{j=1}^n y_{i,j}+ \frac{1}{\prod_{j=1}^n y_{i,j}}\Biggr)
	\end{equation}
on the torus $\bb{G}_m^{nk}$. We extend their method to construct a motive $\mathrm{M}^\lambda_{n+1}$ for each $\lambda\in \bb{N}^n$ in Definition~\ref{def::Kloosterman-motive}, using the Weyl construction. When $\lambda=(k,0,\ldots,0)$, we recover the motive $\rr{M}_{n+1}^k$ constructed by Fresán--Sabbah--Yu.

For each motive $\mathrm{M}_{n+1}^\lambda$, its $\ell$-adic realizations $(\mathrm{M}_{n+1}^\lambda)_\ell$ are continuous $\ell$-adic representations of $\gal(\bq/\bb{Q})$ with coefficients in $\bb{Q}_\ell$, isomorphic to subquotients of $\mathrm{H}_{\et,\rr{c}}^{n|\lambda|-1}(\mathcal{K}_{\bq},\bb{Q}_\ell)(-1)$. By Theorem~\ref{thm::unramified}, we demonstrate that $\{(\mathrm{M}_{n+1}^\lambda)_\ell\}_\ell$ form a compatible family of Galois representations, with $(\mathrm{M}_{n+1}^\lambda)_\ell$ being unramified as a representation of $\gal(\bq_p/\bb{Q}_p)$ for primes $p$ outside a finite set of primes, $S(\lambda,n+1)$. Moreover, there exists an isomorphism of $\gal(\bq_p/\bb{Q}_p)$-representations
	\begin{equation}\label{eq::ell-adic-realization-char-p}
	    (\mathrm{M}_{n+1}^\lambda)_\ell[\zeta_{p}]\simeq \mathrm{H}^1_{\et,\mathrm{mid}}(\bb{G}_{m,\bfp},\Kl_{n+1}^\lambda).
	\end{equation}
Subsequently, we observe that the partial $L$-functions of this family of $\ell$-adic Galois representations coincide with the $L$-functions $L^S(\lambda,n+1;s)$ of $\Kl_{n+1}^\lambda$. We refer to $\mathrm{M}_{n+1}^\lambda$ as the \textit{motive attached to the sheaf $\Kl_{n+1}^\lambda$}. 

To investigate these compatible families of Galois representations, as indicated by \eqref{eq::ell-adic-realization-char-p}, it is necessary to study the cohomologies of Kloosterman sheaves. However, the challenges posed by $\Kl_{n+1}^\lambda$ are notably more intricate compared to the relatively straightforward scenarios encountered with $\Sym^k\Kl_2$ in \cite{yun2015galois,fresan2018hodge}. Notably, we employ complicated combinatorial formulas to describe $\Kl_{n+1}^\lambda$, which all become simple for $\Sym^k\Kl_2$ (see \cref{cor::dimension-mid-de-ell-adic} for example). Also, an annoying new feature of $\Kl_{n+1}^\lambda$ is that their $0$-th cohomology might be nonzero, contrary to the case of $\Sym^k\Kl_2$ where $0$-th cohomology always vanishes. This phenomenon makes the proof of \cref{thm::etale-realization-mod-p} and \cref{thm::ell-adic-galois-even} more technical, necessitating a degree of compromise by introducing certain technical restrictions.

\subsubsection{Potential automorphy}
We prove \cref{thm::L-function} by applying a theorem by Patrikis--Taylor \cite{patrikis_taylor_2015} to $\{(\mathrm{M}_{n+1}^\lambda)_\ell\}_{\ell}$. We must verify a critical condition known as \textit{regularity} to employ this theorem. Through the $p$-adic comparison theorem, this condition amounts to saying that the Hodge numbers of the de Rham realization of $\mathrm{M}_{n+1}^\lambda$ are either $0$ or $1$. Relying on the result in the author's previous paper \cite{Qin23Hodge} (see also \cref{cor::hodge}), the regularity holds for cases presented in Theorem~\ref{thm::L-function}.

Notice that the table of specific values of $(n+1,k)$ in Theorem~\ref{thm::L-function} is chosen so that the Hodge numbers of $\rr{M}_{n+1}^k$ are regular, see \cref{cor::hodge}. Recent developments in (potential) automorphy, such as the work of Boxer--Calegari--Gee--Pilloni \cite{Boxer18}, offer promising avenues for further exploration. These advancements may potentially extend the results of Theorem~\ref{thm::L-function} to cases beyond the current bounds on Hodge numbers.

\subsubsection{Conjectures of Evans type}
Let $\mathrm{M}$ be a motive attached to one of the sheaves in Theorem~\ref{thm::modular form}. To prove the claimed conjectures of Evans type, it suffices to show that the $\ell$-adic realization of $\mathrm{M}$ is modular, meaning it is isomorphic to $\rho_{f,\ell}(h)$ for some modular forms $f$ and some integer $h$. Here, $\rho_{f,\ell}$ represents the two-dimensional Galois representation of $\gal(\bq/\bb{Q})$ attached to $f$, constructed in \cite{deligne1971formes,deligne1971formes-b}. To prove this, we use an argument similar to that in \cite[Thm.\,4.6.1]{yun2015galois} to show the modularity, which is originally due to Serre \cite[\S4.8]{serre1987representations} and can also be found in \cite[Thm.\,1.4.3]{kisin2007modularity}. The key ingredient of this argument is Serre's modularity conjecture. 

After establishing modularity, the remaining task is to determine the modular forms' information as explicitly as possible. We can begin by extracting information from the geometric properties of $\rr{M}$. In Section~\ref{sec::Galois-repr}, we study the compatible family ${\rr{M}_\ell}$ of Galois representations and analyze its conductor $N$. This provides us with some information about the size and prime divisors of the modular form's level. Additionally, we use the calculation of Hodge numbers of the de Rham realization $\rr{M}_{\dR}$ from \cite{Qin23Hodge} to determine the weight of the modular form.

However, due to a lack of information at some "bad" primes or missing calculation of Hodge numbers, we only get partial information on weights and the levels of those modular forms. We turn to numerical results of traces of Frobenius for assistance in obtaining the Fourier coefficients of the corresponding modular form using Sagemath \cite{sagemath}. Then we are able to determine the actual levels of modular forms in \cref{prop::modularity-Sym3Kl4} and \cref{prop::modularity_sym3Kl5}, and the actual weights in \cref{prop::modularity-sym4kl4} and \cref{prop::modularity_Kl3V22}. In particular, we get some new results on Hodge numbers that cannot be obtained using methods from \cite{Qin23Hodge}.

At last, we utilize the information from both geometry and computation to pinpoint the modular form in the LMFDB database.

\subsection{Organization of the article} 
In Section~\ref{sec::Kloosterman-sheaf}, we investigate the properties of Kloosterman sheaves, primarily focusing on those appearing in Theorems~\ref{thm::L-function} and~\ref{thm::modular form}, including their local structures at $0$ and $\infty$, the dimension formulas for their $\ell$-adic cohomologies. In Section~\ref{sec::motives}, we construct the motives attached to Kloosterman sheaves and explore properties of their de Rham realizations, $\ell$-adic realizations, and other realizations in characteristic $p>0$. In Section~\ref{sec::functional-eq}, we first investigate the ramification properties of the Galois representations $(\mathrm{M}_{n+1}^\lambda)_\ell$ as detailed in Theorem~\ref{thm::unramified}, and Theorem~\ref{thm::ell-adic-galois-even}. Then, we prove Theorem~\ref{thm::L-function}. In Section~\ref{sec::evans}, we demonstrate Theorem~\ref{thm::modular form} by showing the modularity for each sheaf case by case in Propositions~\ref{prop::modularity_sym4Kl3}, \ref{prop::modularity-Sym3Kl4}, \ref{prop::modularity-sym4kl4}, \ref{prop::modularity_sym3Kl5}, \ref{prop::modularity_Kl3V21} and \ref{prop::modularity_Kl3V22}. In \cref{appx:computation}, we outline the process of calculating moments of Kloosterman sheaves.

\subsection*{Acknowledgement} This work is based on the author's Ph.D. thesis, prepared at Centre de Mathématiques Laurent Schwartz at École Polytechnique. He wants to thank his supervisors, Javier Fresán and Claude Sabbah, for proposing this question and for many helpful discussions and suggestions. He also thanks Lei Fu and Christian Sevenheck for their comments on a previous version of this article, as well as Jochen Heinloth, Gabriel Ribeiro, Bin Wang, Ping Xi, and Daxin Xu for valuable discussions. Finally, the author is grateful to the anonymous referee for numerous suggestions that helped correct inaccuracies and improve the presentation of this paper.

\section{Properties of Kloosterman sheaves}\label{sec::Kloosterman-sheaf}
In this section, we primarily focus on Kloosterman sheaves appearing in Theorems~\ref{thm::L-function} and~\ref{thm::modular form}. After recalling some preliminaries about Weyl's construction and $\ell$-adic sheaves, we give Kloosterman sheaves geometrical descriptions in \cref{prop:geometric-char}. Then we describe their local structures at $0$ and $\infty$ in \cref{sec::local-monodromy-0-ell} and \cref{sec::local-monodromy-infty-ell}. At last, we give dimension formulas of the $\ell$-adic cohomologies of Kloosterman sheaves in \cref{sec::dim-middle}.

\subsection{Weyl's construction}\label{sec::weyl-construction}

We recall some preliminaries from \cite[\S6, \S15 \& \S17]{W.Fulton2013}. A \textit{partition} of an integer $k$ is a sequence of nonnegative integers of the form $\mu:=(\mu_1,\mu_2, \ldots,\mu_m)$ such that $\mu_1\geq \mu_2\cdots\geq \mu_m$ and $\sum_i{\mu_i}=k$. For a partition of $k$, we can associate a Young diagram, such that $\mu_i$ are the lengths of the $i$-th rows. For example, the Young diagram of the partition $(3,2,1)$ is shown in the following diagram.
	\begin{center}
	\begin{ytableau}
	1 & 2 &3\\
	4 &5 & \none\\
	6&\none&\none
	\end{ytableau} 
	\end{center}
	
	For a partition $ \mu$ of $k$, we define two elements $a_{\mu}$ and $b_{\mu}$ in the group ring $\mathbb{Z}[S_k]$ as follows. First, we label each block in the Young diagram by indexes in $\{1,\ldots,k\}$. We take $P_\mu:=\{\sigma\in S_k\mid \sigma \text{ preserves each row}\}$ and $Q_{\mu}:=\{\tau\in S_k\mid \tau \text{ preserves each column}\}$. Let $\mathrm{sign}\colon S_k\to \{\pm 1\}$ be the sign character of $S_k$. Then we define 
		$$a_{\mu}:=\sum_{\sigma\in P_{\mu}}\sigma, \quad b_{\mu}:=\sum_{\tau \in Q_{\mu}}\mathrm{sign}(\tau)\tau$$
	and $c_\mu=a_\mu\cdot b_\mu$ in the group ring $\bb{Z}[S_k]$.
	
	Let $K$ be a field of characteristic $0$ and $V=K^{n+1}$ be the standard representation of $\rr{SL}_{n+1}$. The group $S_k$ acts on the tensor product $V^{\otimes k}$ by 
		$$\sigma \cdot v_1\otimes \cdots \otimes v_k:=v_{\sigma(1)}\otimes \cdots\otimes v_{\sigma(k)}.$$
	Then we have the endofunctor $\mathbb{S}_{\mu}$ of the category of finite-dimensional representations of $\mathrm{SL}_{n+1}$ defined by $\mathbb{S}_{\mu}V:=V^{\otimes k}\cdot c_{\mu}.$ For convenience, we also write 
 
    \begin{equation}\label{eq:isotypic-component-convention}
        (V^{\otimes k})^{P_\mu\times Q_\mu,1\times \rr{sign}}:=V^{\otimes k}\cdot c_{\mu}.
    \end{equation}
	
	Let $\lambda=(\lambda_1,\ldots,\lambda_n)$ be a sequence of nonnegative integers. Let $V$ be the standard representation $K^{n+1}$ equipped with the natural action of $\mathrm{SL}_{n+1}$ and $V_\lambda$ be the unique irreducible subrepresentation of the highest weight $\sum_{i} \lambda_i (L_1+\ldots+L_i)$ of 
		$$\Sym^{\lambda_1}V\otimes \Sym^{\lambda_2}\wedge^2V\otimes \cdots\otimes \Sym^{\lambda_n}\wedge^nV.$$

	In the case of $\mathrm{SL}_3$, the representation with the highest weight $\lambda_1L_1+ \lambda_2(L_1+L_2)$ can be described as 
		\begin{equation}\label{eq::irreducible-repr-sl3}
			\ker(\Sym^{\lambda_1}V\otimes \Sym^{\lambda_2}\wedge^2V\xrightarrow{\pi_{\lambda_1,\lambda_2}} \Sym^{\lambda_1-1}V\otimes \Sym^{\lambda_2-1}\wedge^2V),
		\end{equation}
	where $\pi_{\lambda_1,\lambda_2}$ sends $v_1\otimes \cdots \otimes v_{\lambda_1}\otimes w_1\otimes\dots\otimes w_{\lambda_2}$ to 
			$$\frac{1}{(\lambda_1)!(\lambda_2)!}\sum_{\sigma \in S_{\lambda_1},\tau\in S_{\lambda_2}}<v_{\sigma(1)},w_{\tau(1)}>\cdot v_{\sigma(2)}\cdots v_{\sigma(\lambda_1)}\otimes w_{\tau(2)}\cdots w_{\tau(\lambda_2)}$$ 
	where $<\cdot,\cdot>\colon V\times \wedge^2V\to K$ is the natural pairing.

	In general, we can construct the representation $V_{\lambda}$ using Schur functors as follows. Let 
		\begin{equation}\label{eq:G_lambda}
            \begin{split}
                \mu(\lambda)&:=\Bigl(\sum_{j=1}^n\lambda_j,\sum_{j=2}^n\lambda_j,\ldots,\lambda_n\Bigr),\\
			    G_\lambda&:=P_{\mu(\lambda)}\times Q_{\mu(\lambda)}.
            \end{split}
		\end{equation} 
	By applying $\mathbb{S}_{\mu(\lambda)}$ to $V^{\otimes |\lambda|}$, the resulting representation is nothing but $V_{\lambda}$. More precisely, we have $V_{\lambda}=(V^{\otimes |\lambda|})^{G_\lambda,1\times \rr{sign}}$. For example, if $\lambda=(k,0,\ldots,0)$, then $P_\lambda =S_k$ and $Q_\lambda$ is trivial. Hence, $\bb{S}_\lambda(V^{\otimes k})=(V^{\otimes k})^{S_k}=\Sym^kV$.

\subsection{Some generalities on \texorpdfstring{$\ell$}{l}-adic sheaves}

Let $p\neq \ell$ be two prime numbers, $q$ a power of $p$, $\iota\colon \bql\hookrightarrow \bb{C}$ an embedding. We denote by $E$ either the algebraic closure $\bql$ of $\bb{Q}_\ell$, or a finite extension of $\bb{Q}_\ell$ inside $\bql$. By an $\ell$-adic sheaf on a connected separated Noetherian scheme $X$ over $\fqq$, we mean a constructible $E$-sheaf on $X$.

\subsubsection{Cohomologies of \texorpdfstring{$\ell$}{ell}-adic sheaves on curves}

Let $C$ be a geometrically connected smooth projective curve over $\fqq$. The $\ell$-adic cohomologies $\rr{H}^i_{\et}(C_{\bar{\bb{F}}_q},\scr{F})$ of an $\ell$-adic sheaf $\scr{F}$ on $C$ are finite-dimensional $E$-vector spaces equipped with Frobenius actions. 

Suppose that $\scrf$ is a lisse $\ell$-adic sheaf on an affine open subset $U$ of $C$. We denote by $\rho_{\scr{F}}$ the corresponding continuous $\ell$-adic representation of $\piet(U,\bar{ \eta}_U)$, and by $G_{\mathrm{geom}}$ the geometric monodromy group of $\scrf$, i.e., the Zariski closure of the image of $\piet(U_{\bar{\bb{F}}_q},\bar {\eta}_U)$ in $\mathrm{GL}(\scrf_{\bar \eta})$ under $\rho_{\scrf}$. Then 
	$$\rr{H}_{\et}^2(U_{\bar{\bb{F}}_q},\scrf)=\rr{H}_{\et,\rr{c}}^0(U_{\bar{\bb{F}}_q},\scrf)=0,$$
	$$\rr{H}_{\et}^0(U_{\bar{\bb{F}}_q},\scrf)=(\scrf\mid_{\bar{\eta}_U})^{G_{\rr{geom}}}, 
	\quad \text{and} \quad
	\rr{H}^2_{\et,\rr{c}}(U_{\bar{\bb{F}}_q},\scrf)=(\scrf\mid_{\bar{\eta}_U})_{G_{\rr{geom}}}(-1),$$
where $(\scrf\mid_{\bar{\eta}_U})^{G_{\rr{geom}}} $ and $(\scrf\mid_{\bar{\eta}_U})_{G_{\rr{geom}}}$ are the invariants and the coinvariants of $\scrf$ under the action of $G_{\rr{geom}}$.

\subsubsection{The Grothendieck--Ogg--Shafarevich formula} For each closed point $x\in |C|$, we denote the localization (resp. strict localization) of $C$ at $x$ (resp. $\bar x$) by $C_{(x)}$ (resp. $C_{(\bar x)}$). The special points and generic points of $C_{(x)}$ and $C_{(\bar x)}$ are denoted by $s_x,\eta_x$ and $s_{\bar x}, \eta_{\bar x}$ respectively.

Let $\scrf$ be an $\ell$-adic sheaf on $C$ which is lisse on an open subset $U\subset C$. We denote $\rk(\scrf)=\rk(\scrf_{\bar\eta})$, $\rk_x(\scrf)=\rk(\scrf_{\bar s_x})$, and $\sw_x(\scrf)=\sw(\scrf_{\bar\eta_x})$. Then the Euler characteristic $\chi(U_{\bar{\bb{F}}_q},\scrf|_U)=\sum_{i=0}^2(-1)^{i+1}\dim \rr{H}_{\et}^i(U_{\bar{\bb{F}}_q},\scrf|_U)$ can be computed by the Grothendieck--Ogg--Shafarevich formula
    \begin{equation}\label{eq:GOS-formula}
        \chi(U_{\bar{\bb{F}}_q},\scrf|_U)=\Bigl(2-2g-\sum_{x\in|C\backslash U|}\deg(x)\Bigr)\cdot \rk(\scrf)-\sum_{x\in |C\backslash U|}\deg(x)\cdot \sw_x(\scrf).
    \end{equation}          
see \cite[X.\,Théorème\,7.1]{SGA5} or \cite[(2.2)]{Lau87}. The sum on the right-hand side is a finite sum because $\sw_x(\scrf)=0$ whenever $x\in U$.

\subsubsection{The middle \texorpdfstring{$\ell$}{ell}-adic cohomology}
Let $C$ be a curve as above, $j\colon U\hookrightarrow C$ an open immersion, and $\scrf$ an $\ell$-adic cohomology on $U$. The \textit{middle $\ell$-adic cohomology} of $\scrf$ is the image of the forgetting support morphism 
	$$\rr{H}^1_{\et,\rr{c}}(C_{\bar{\bb{F}}_q},\scrf)\to \rr{H}^1_{\et}(C_{\bar{\bb{F}}_q},\scrf),$$
denoted by $\rr{H}^1_{\et,\rr{mid}}(C_{\bar{\bb{F}}_q},\scrf)$, which is identified with the $\ell$-adic cohomology of the (non-derived) direct image $j_*\scrf$.
According to \cite[2.0,7]{katz1988gauss}, we have a long exact sequence
	\begin{equation}\label{eq::long-exact-sequence}
		\begin{split}
			0\to (\scrf\mid_{\bar{\eta}_U})^{G_{\rr{geom}}}\to &\oplus_{x\in |C\backslash U|,\ \bar x \text{ over }x } (\scrf|_{\bar\eta_x})^{I_{\bar x}} 
			\to\Hetc{1}(U_{\bfp},\scrf)\\
			\to&\Het{1}(U_{\bfp},\scrf) 
                    \to \oplus_{x\in |C\backslash U|,\ \bar x \text{ over }x } (\scrf|_{\bar \eta_x})_{I_{\bar x}}(-1)\to (\scrf\mid_{\bar{\eta}_U})_{G_{\rr{geom}}}(-1)\to 0
		\end{split}
	\end{equation}
where $\eta_{\bar x}$ are the generic point of the strict henselization of $C$ at $\bar x$, the groups $I_{\bar x}$ are the inertia groups at $\bar x$, $(\scrf|_{\bar\eta_x})^{I_{\bar x}}$ are the invariants of $I_{\bar x}$ and $(\scrf|_{\bar\eta_x})_{I_{\bar x}}$ are the coinvariants of $I_{\bar x}$. 

Assume that $\scrf$ is pure of weight $w$. By the main theorem of Weil II \cite[3.3.1]{deligne1980conjecture} and \eqref{eq::long-exact-sequence}, we conclude that 
	$$\mathrm{H}_{\et,\mathrm{mid}}^1(U_{\bar{\bb{F}}_q},\scrf)\simeq \gr^W_{w+1}\Hetc{1}(U_{\bar{\bb{F}}_q},\scrf)\simeq \gr^W_{w+1}\Het{1}(U_{\bar{\bb{F}}_q},\scrf).$$
 In particular, the dimension of the middle $\ell$-adic cohomology is given by
	\begin{equation}\label{eq::dim-middle-general-formula}
		\dim \Hetc{1}(U_{\bar{\bb{F}}_q},\scrf)-\sum_{x\in |C\backslash U|,\ \bar x \text{ over }x }\dim (\scrf|_{\bar\eta_x})^{I_{\bar x}}+\dim \scrf^{G_{\mathrm{geom}}}.
	\end{equation}

\subsection{Kloosterman sheaves}\label{subsec:Kloosterman-sheaves}
Let $p$ and $\ell$ be two distinct prime numbers and $\fqq$ be the finite field with $q=p^r$ elements. Let $\zeta_p$ be a primitive $p$-th root of unity $\zeta_p$ in $\bql$, and we denote by $E=\bb{Q}_\ell(\zeta_p)$. We fix a nontrivial additive character $\psi_p\colon \bb{F}_p\to E\cros$, and denote by $\psi_q$ the character $\psi_p\circ \tr_{\fqq/\fpp}$. The \textit{Artin-Schreier sheaf} $\scr{L}_{\psi_q}$ is a lisse $\ell$-adic sheaf with coefficients in $E$ on $\bb{A}^1_{\fqq}$, whose trace function is given by $\psi_q$. We denote by $\scr{L}_{\psi_q(f)}$ the inverse image $f^*\scr{L}_{\psi_q}$ of the Artin-Schreier sheaf along a regular function $f\colon X\to \bb{A}^1_{\fqq}$. 

Consider the following diagram
	\begin{equation}\label{eq:diagram-Kl}
		\begin{tikzcd}
		~ & \bb{G}_m^{n+1}\ar[ld, "\pi"'] \ar[rd,"\sigma"] & ~\\
		\bb{G}_m & ~ &\bb{A}^1
		\end{tikzcd}
	\end{equation}
where $\sigma$ denotes the sum of coordinates and $\pi$ denotes the product of coordinates. We define the \textit{Kloosterman sheaf} on $\bb{G}_{m,\bb{F}_q}$ by
	\begin{equation}\label{eq::defn-kln+1}
		\Kl_{n+1}:=\mathrm{R}^{n}\pi_!\sigma^*\scr{L}_{\psi_q}.
	\end{equation}

Deligne showed in \cite[Sommes.\,Trig.\,Thm.\,7.8]{SGA41/2} that $\Kl_{n+1}$ is a lisse $\ell$-adic sheaf of rank $n+1$, pure of weight $n$, tamely ramified at $0$ with a single Jordan block, and is totally wildly ramified at $\infty$ with Swan conductor $1$. Moreover, we have an isomorphism 
	$$\Kl_{n+1}^\vee \simeq \iota_{n+1}^*\Kl_{n+1}$$
where $\Kl_{n+1}^\vee$ is the dual of $\Kl_{n+1}$ and $\iota_{n+1}\colon \bb{G}_m\to \bb{G}_m$ is defined by the multiplication of $(-1)^{n+1}$. 

In the generality of \textit{Kloosterman sheaves for reductive groups} constructed in \cite{HNY-Kloosterman}, one gets a tensor functor \eqref{eq:HNY-functor} from the category of finite-dimensional representations of $\rr{SL}_{n+1}$ to the category of $\ell$-adic local systems on $\bb{G}_m$. If we take $V$ as the standard representation $\rr{Std}$ of $\rr{SL}_{n+1}$ and the symmetric power $\Sym^k\rr{Std}$, then $\Kl_{\rr{SL}_{n+1}}(V)$ are $\Kl_{n+1}(\frac{n}{2})$ and $\Kl_{\rr{SL}_{n+1}}(V)=\Sym^k\Kl_{n+1}(\frac{nk}{2})$ respectively.

If we take $V$ as the irreducible representation of the highest weight $\lambda$, we get $\bigl(\Kl_{n+1}^{\otimes |\lambda|}\bigr)^{G_\lambda,1\times \rr{sign}}\bigl(\frac{n|\lambda|}{2}\bigr)$. For simplicity, we write 
\begin{equation}\label{eq:Kloosterman-sheaves}
    \Kl_{n+1}^\lambda:=\Kl_{\rr{SL}_{n+1}}(V_\lambda)\bigl(-\tfrac{n|\lambda|}{2}\bigr) .
\end{equation}

Alternatively when $n=2$, we use \eqref{eq::irreducible-repr-sl3} to conclude that the sheaf $\Kl_{\rr{SL}_3}(V_{\lambda_1,\lambda_2})$ is the kernel of 
\begin{equation}\label{eq::sheaf-repr-sl3}
	\Sym^{\lambda_1}\Kl_3\otimes \Sym^{\lambda_2}(\Kl_3)^\vee (\lambda_1+\lambda_2)\to \Sym^{\lambda_1-1}\Kl_3\otimes \Sym^{\lambda_2-1}(\Kl_3)^\vee (\lambda_1+\lambda_2-2).
\end{equation}

\subsubsection{Geometric interpretations}
 Now, we describe Kloosterman sheaves \eqref{eq:Kloosterman-sheaves} geometrically. Let $g\colon \bb{G}_{m,\fpp}^n\to \bb{A}^1_{\fpp}$ be the Laurent polynomial $\sum_{i=1}^n y_i +\frac{1}{\prod _i y_i}$ and $[n+1]\colon \bb{G}_{m,\fpp}\to \bb{G}_{m,\fpp}$ the $(n+1)$-th power map.
 
\begin{lem}\label{lem::D-F-trasnform}
	 We have an isomorphism	of $\ell$-adic sheaves
		$$[n+1]^*\Kl_{n+1}\simeq \FT_{\psi_p}(\mathrm{R}^{n-1}g_!E)|_{\bb{G}_m},$$
	where $\FT_{\psi_p}$ is the Deligne-Fourier transform \cite{Lau87}.
\end{lem}
\begin{proof}
The proof is similar to that of \cite[Prop.\,2.10]{fresan2018hodge}. Let $x_1,\ldots, x_{n+1}$ be the coordinates of $\bb{G}_{m,\fpp}^{n+1}$ in the diagram \eqref{eq:diagram-Kl}. We perform a change of variable $z=\prod_{i=1}^n x_i$. Let $j\colon \bb{G}_{m,\fpp}\to \bb{A}^1_{\fpp}$. Then we can rewrite \eqref{eq::defn-kln+1} as 
	$$\Kl_{n+1}=\mathrm{R}(\pr_{z})_!\scr{L}_{\psi_p\bigl(\sum_{i=1}^n x_i+\frac{z}{\prod_{i=1}^n x_i}\bigr)}[n],$$
where $\pr_{z}$ is the projection from $\mathbb{G}_{m}^{n+1}$ to $\mathbb{G}_{m,z}$. Let $t$ be the coordinate of the source of the map $[n+1]$. Then 
	\begin{equation}
		\begin{split}
			[n+1]^*\Kl_{n+1}\simeq &\mathrm{R}(\pr_{z})_!\scr{L}_{\psi_p\bigl(\sum_{i=1}^n x_i+\frac{t^{n+1}}{\prod_{i=1}^n x_i}\bigr)}[n] 
							\simeq j^*\rr{R}(\pr_t)_!\scr{L}_{\psi_p(tg)}[n],
		\end{split}
	\end{equation}
where $tg$ is seen as a function on $\mathbb{G}_{m}^n\times \mathbb{A}^1_{t}$, $\mathrm{pr}_t$ is the projection from $\mathbb{G}_{m}^n\times \mathbb{A}^1_t$ to $\mathbb{A}^1_t$, and we performed a change of variable $y_i=x_i/t$ in the last isomorphism. Then, by a calculation of the Deligne--Fourier transform, we obtain
\begin{equation*}
	\begin{split}
		\FT_{\psi_p}(\rr{R}g_!E)\simeq & \rr{R}(\pr_2)_!(\pr_1^*\rr{R}g_!E\otimes \scr{L}_{\psi_p(xt)}[1])\\
		\simeq & \rr{R}(\pr_2)_!(\rr{R}(g\times \id)_!\pr_1^*E\otimes \scr{L}_{\psi_p(xt)}[1])\\
		\simeq & \rr{R}(\pr_2)_!\rr{R}(g\times \id)_!(\pr_1^*E\otimes \scr{L}_{\psi_p(tg)}[1])\\
		\simeq &\rr{R}(\pr_t)_!\scr{L}_{\psi_p(tg)}[1],
	\end{split}
\end{equation*}
where we used the base change theorem in the second isomorphism and the projection formula in the third isomorphism.
The morphisms in the above calculation are illustrated in the following diagram.
$$\begin{tikzcd}[column sep={5em,between origins} ]
	& \bb{G}_m^n\times \bb{A}^1_t \ar[ld,"\pr_1"']\ar[rd,"g\times \id"] & & \\
	\bb{G}_m^n\ar[rd,"g"] &  &\bb{A}^1_x \times \bb{A}^1_t \ar[ld,"\pr_1"']\ar[rd,"\pr_2"] & \\
	& \bb{A}^1_x & & \bb{A}^1_t &
\end{tikzcd}$$
We conclude from the above isomorphisms that $[n+1]^*\Kl_{n+1}\simeq j^*\FT_{\psi_p}(\mathrm{R}g_!E)[n-1].$
\end{proof}

Consider the torus $\bb{G}_{m,\fpp}^{n|\lambda|+1}$ with coordinates $\{x_{i,j}\mid 1\leq i\leq |\lambda|, 1\leq j\leq n\}$ and $z$. Let $f_{|\lambda|}\colon \bb{G}_{m,\fpp}^{n|\lambda|+1}\to \bb{A}^1_{\fpp}$ be the Laurent polynomial $\sum_{i=1}^{|\lambda|} \Bigl(\sum_{j=1}^n x_{i,j}+\frac{z}{\prod_j x_{i,j}} \Bigr)$ and $\pr_z$ be the projection from $\bb{G}_{m,\fpp}^{n|\lambda|+1}$ to $\mathbb{G}_{m,z}$.

Similarly, consider the torus $\bb{G}_{m,\fpp}^{n|\lambda|+1}$ with coordinates $\{x_{i,j}\mid 1\leq i\leq |\lambda|, 1\leq j\leq n\}$ and $t$. We let $\tilde f_{|\lambda|}\colon \bb{G}_{m,\fpp}^{n|\lambda|+1}\to \bb{A}^1_{\fpp}$ be the Laurent polynomial $\sum_{i=1}^{|\lambda|} \Bigl(\sum_{j=1}^n x_{i,j}+\frac{t^{n+1}}{\prod_j x_{i,j}} \Bigr)$ and $\pr_t$ be the projection from $\bb{G}_{m,\fpp}^{n|\lambda|+1}$ to its $\mathbb{G}_{m,t}$.

\begin{prop}\label{prop:geometric-char}
	We have the isomorphism	of $\ell$-adic sheaves
		$$\Kl_{n+1}^\lambda\simeq \Bigl(\rr{R}^{n|\lambda|}\pr_{z*}\scr{L}_{\psi_p(f_{|\lambda|})}\Bigr)^{G_\lambda,\rr{sign}^n\times \rr{sign}^{n+1}}$$
	and
		$$[n+1]^*\Kl_{n+1}^\lambda\simeq \Bigl(\rr{R}^{n|\lambda|}\pr_{t*}\scr{L}_{\psi_p(\tilde{f}_{|\lambda|})}\Bigr)^{G_\lambda,\rr{sign}^n\times \rr{sign}^{n+1}},$$
	where $G_\lambda=P_{\mu(\lambda)}\times Q_{\mu(\lambda)}$ is the group defined in \eqref{eq:G_lambda}, and the component $(G_\lambda,\rr{sign}^n\times \rr{sign}^{n+1})$ means taking the isotypic component with respect to 
		$\sum_{\sigma \in P_{\mu(\lambda)}}\sign^n(\sigma)\sigma\cdot \sum_{\tau \in Q_{\mu(\lambda)}}\sign^{n+1}(\tau)\tau.$
\end{prop}
\begin{proof}
	By \cite[(1.2.2.7)]{Lau87}, the Deligne--Fourier transform interchanges tensor product and the convolution. Using \cref{lem::D-F-trasnform}, we have
		\begin{equation*}
			\begin{split}
				([n+1]^*\Kl_{n+1})^{\otimes |\lambda|}\simeq 
				&j^*\rr{FT}_{\psi_p}\Bigl(((\mathrm{R}g_!E)[n-1])^{\boxtimes |\lambda|}\Bigr)[|\lambda|]\\
		 \simeq &j^*\rr{FT}_{\psi_p}\bigl((\mathrm{R}g^{\boxplus |\lambda|}_!E)\bigr)[n|\lambda|-1]\\
		 \simeq &\mathrm{R}(\pr_t)_!\scr{L}_{\psi_p(t\cdot g^{\boxplus |\lambda|})}[n|\lambda|-1]\\
		 \simeq &\mathrm{R}(\pr_t)_!\scr{L}_{\psi_p(\tilde f_{|\lambda|})}[n|\lambda|-1],
			\end{split}
		\end{equation*}
		where $g^{\boxplus|\lambda|}$ is the Laurent polynomial in \eqref{eq:gbox}, we used the Künneth formula in the second isomorphism, and we performed a change of variable $x_{i,j}=t\cdot y_{i,j}$ in the last isomorphism.

		Notice that the Deligne--Fourier transform preserves the action of the symmetric group $S_{|\lambda|}$. However, the Künneth formula yields an extra sign character $\rr{sign}^n$ on the right-hand side. By taking the corresponding isotypic component on both sides, we get the second isomorphism.

	As for the first isomorphism, similar to Remark~\ref{rek::Galois-descent-motive}, one has 
	\begin{equation*}
		\begin{split}
			\Kl_{n+1}^{\otimes |\lambda|}\simeq 
				   &\bigl([n+1]_*([n+1]^*\Kl_{n+1})^{\otimes |\lambda|}\bigr)^{\mu_{n+1}}\\
	 	    \simeq &\Bigl(\mathrm{R}(\pr_t)_!\scr{L}_{\psi_p(\tilde f_{|\lambda|})}\Bigr)^{\mu_{n+1}}[n|\lambda|-1]\\
			\simeq &\mathrm{R}(\pr_z)_!\scr{L}_{\psi_p(  f_{|\lambda|})}[n|\lambda|-1].
		\end{split}
	\end{equation*}
	At last, we add the corresponding isotypic components to both sides and get the first isomorphism.
\end{proof}

\subsection{The local structures of Kloosterman sheaves at \texorpdfstring{$0$}{0}}\label{sec::local-monodromy-0-ell}

Let $\bb{A}_{( 0)}^1=\{s_{ 0},\eta_{0}\}$ be the henselization of $\bb{A}^1_{\fqq}$ at $0$. The inertial group $I_{\bar 0}$ acts on the generic fiber $(\Kl_{n+1})_{\bar \eta_0}$. By a special case of \cite[7.0.7]{katz1988gauss}, the generic fiber $V=(\Kl_{n+1})_{\bar{\eta}_0}$ is a tamely ramified $\ell$-adic representation of $\gal(\bar {\eta}_0/\eta_0)$ with coefficients in $E=\bb{Q}_\ell(\zeta_p)$. The inertia group $I_{\bar 0}$ acts on $V$ unipotently by a single Jordan block. More precisely, there exists a basis $\{v_0,v_1,\ldots, v_n\}$ on which the nilpotent part of the monodromy operator $N\colon V\to V(-1)$ and $\frob_0$ act by
	$$\frob_0(v_i)=q^{n-i}v_i \ \text{ and }\ N(v_i)=v_{i+1}$$
for $i=0,\ldots,n$ (for convenience, we let $v_{n+1}=0$). 

\begin{rek}\label{req:inv-0-ind-p}
	The local monodromy of $\Kl_{n+1}|_{\eta_0}$ does not depend on the characteristic $p$ of the base field $\fqq$. Therefore, the local monodromy of $(\Kl_{n+1}^\lambda)_{\bar \eta_0}=V^{\otimes |\lambda|}\cdot c_\lambda$ is also independent of $p$. Consequently, the dimension of the $I_{\bar 0}$-invariants of $(\Kl_{n+1}^\lambda)_{\bar \eta_0}$ remains independent on $p$.
\end{rek}

The dimension of the $I_{\bar 0}$-invariants of $(\Sym^k\Kl_{n+1})_{\bar \eta_0}$ are computed in \cite[Thm.\,0.1]{fu2006trivial}. 
\begin{thm}[Fu--Wan]\label{thm::local-symk-at-0}
As a $\frob_q$-module, the $I_{\bar 0}$-invariants $(\Sym^k\Kl_{n+1})_{\bar \eta_0}^{I_{\bar 0}}$ is isomorphic to 
    $$\bigoplus_{u=0}^{\lfloor\frac{nk}{2}\rfloor}E(-u)^{\oplus m_k(u)},$$ 
where $m_k(u)$ are numbers characterized by the generating series,
\begin{equation}\label{eq::counting_inv_at_0}
\begin{split}
	\sum_{u=0} m_k(u)x^u=\prod_{i=n+1}^{n+k}(1-x^i)\cdot \prod_{i=2}^{k}(1-x^i)\inv.
\end{split}
\end{equation}
In particular, the dimension of $(\Sym^k\Kl_{n+1})_{\bar \eta_0}^{I_{\bar 0}}$ is $\sum_{u=0}^{\lfloor\frac{nk}{2}\rfloor}m_k(u).$
\end{thm}

To finish, we provide the formula of dimensions of $I_{\bar 0}$-invariants of $\Kl^{(2,1)}_{3}|_{\eta_0}$ and $\Kl^{(2,2)}_{3}|_{\eta_0}$.
\begin{prop}\label{prop::more-example-at-0}
\begin{enumerate}[label=(\theenumi)., ref= \ref{prop::more-example-at-0}.(\theenumi)]
	\item As a $\frob_q$-module $\Kl^{(2,1)}_{3}|_{\eta_0}$ is isomorphic to 
		$$\bigoplus_{i=1}^7E(-i)\bigoplus \bigoplus_{i=2}^6E(-i)\bigoplus \bigoplus_{i=3}^5 E(-i),$$
	and the $I_{\bar 0}$-invariants of $\Kl^{(2,1)}_{3}|_{\eta_0}$ is isomorphic to 
		$$E(-1)\bigoplus E(-2)\bigoplus E(-3).$$
	 
	\item As a $\frob_q$-module $\Kl^{(2,2)}_{3}|_{\eta_0}$ is isomorphic to 
		$$\bigoplus_{i=2}^{10}E(-i) \bigoplus \bigoplus _{i=3}^{9}E(-i)\bigoplus \bigoplus_{i=4}^{8}E(-i)^{\oplus 2}\bigoplus E(-6),$$
	and the $I_{\bar 0}$-invariants of $\Kl^{(2,2)}_{3}|_{\eta_0}$ is isomorphic to 
		$$E(-2) \bigoplus E(-3)\bigoplus E(-4)^{\oplus 2}\bigoplus E(-6).$$
\end{enumerate}
\end{prop}

\begin{proof}
The proof is similar to that of \cref{thm::local-symk-at-0}. We provide the proof for (1), while the proof for (2) is similar. 
 
The nilpotent part $N$ of the monodromy operator on $V=(\Kl_{3})_{\bar\eta_0}$ can be enhanced to a Lie algebraic representation $\rho$ of $\rr{sl}_2$, such that $\rho\Bigl(\begin{matrix} 0 & 0\\1 & 0\end{matrix}\Bigr)=N$. Similarly, the nilpotent part of the monodromy operator on $(\Sym^k\Kl_{3})_{\bar \eta_0}$ can be viewed as an $\rr{sl}_2$-representation $\rho_k$ with $\rho_k\Bigl(\begin{matrix} 0 & 0\\1 & 0\end{matrix}\Bigr)=\Sym^k N$. By the representation theory of $\rr{sl}_2$, we can decompose $(\Sym^k\Kl_{3})_{\bar \eta_0}$ into irreducible representations of $\rr{sl}_2$ as $\bigoplus_{i=0}^{{\lfloor\frac{k}{2}\rfloor}} \Sym^{2k-2i}E^2$. Moreover, each $\Sym^{2k-2i}E^2$ is isomorphic to $ \bigoplus_{j=i}^{2k-i}E(-j)$ as a $\frob_q$-module. As for the subspace of $I_{\bar 0}$-invariants of $(\Sym^k\Kl_{n+1})_{\bar \eta_0}$, it is identified with the kernel of $\Sym^kN$, which is $\bigoplus_{i=0}^{{\lfloor\frac{k}{2}\rfloor}} E(-i)$. 
 
Back to $\Kl_{3}^{(2,1)}$ and we omit the Tate twists for now. Using the alternative description \eqref{eq::sheaf-repr-sl3}, to determine the local structure of $\Kl_3^{(2,1)}$, it is sufficient to establish that of $\Sym^{2}\Kl_3\otimes \Kl_3^\vee$.  As $\rr{sl}_2$-representations, we have isomorphisms
	$$V=V^\vee=\Sym^2E^2 \text{ and } \Sym^2V= \Sym^4E^2 \oplus E.$$
By the formula
		$$\Sym^aE^2\otimes \Sym^b E^2=\Sym^{a+b}E^2\oplus \Sym^{a+b-2}E^2\oplus \cdots \oplus \Sym^{|a-b|}E^2$$
from \cite[Exe.\,11.11]{W.Fulton2013}, one concludes that 
	$$\Sym^2V\otimes V^{\vee}=\Sym^6E^2\oplus \Sym^4E^2\oplus (\Sym^2 E^2)^{\oplus 2}.$$
By removing one piece of $\Sym^2 E^2$ from $\Sym^2V\otimes V^{\vee}$ and adding back the Tate twists, we get the expression of $\Kl^{(2,1)}_{3}|_{\eta_0}$ as well as that of  $\Kl^{(2,1)}_{3}|_{\eta_0}^{I_{\bar 0}}$.
\end{proof}

\subsection{The local structures of Kloosterman sheaves at \texorpdfstring{$\infty$}{infinity}}\label{sec::local-monodromy-infty-ell}

\subsubsection{Notation}\label{nota::multi-indices}
Let $p$, $\ell$, $\fqq$, and $E$ be as in Section~\ref{subsec:Kloosterman-sheaves}. We fix a primitive $(n+1)$-th root of unity $\zeta=\zeta_{n+1}$ in $\bfp$. 

\begin{enumerate}[label=(\theenumi)., ref= \ref{nota::multi-indices}.(\theenumi)]
	\item  For multi-indices $\underline I\in \bb{N}^{n+1}$, we denote by $C_{\underline I}=\sum_{i=0}^n I_i\cdot \zeta^i$ and $m_{\underline I}=\sum_{i=0}^ni\cdot I_i$.
	\item  Let $v_0,\ldots,v_n$ be a basis of $E^{n+1}$. We denote by $\sigma=(01\cdots n)\in S_{n+1}$, acting on $v_i$ by $\sigma v_i:=v_{\sigma(i)}$. For a multi-index $\underline I\in \bb{N}^{n+1}$, we denote by $v^{\underline I}=v_0^{I_0}\cdots v_n^{I_n}$. 
	\begin{enumerate}
		\item Let $d(k,n+1,p)$ be the cardinality of the set $A_{k}^0:=\{\underline I\mid |\underline I|=k,\ C_{\underline I}=0\}$. 
		\item We denote by $a(k,n+1,p)$ the cardinality of the set of $\sigma$-orbits in $A_k^0$.
		\item We denote by $b(k,n+1,p)$ $\sigma$-orbits in $A_{k}^{0}$ such that the subspace spanned by the orbit is not zero.
		\item We denote by $d(k,n+1)$, $a(k,n+1)$ and $b(k,n+1)$ the generic values of $d(k,n+1,p)$, $a(k,n+1,p)$ and $b(k,n+1,p)$ as $p$ varies respectively.
	\end{enumerate}
\end{enumerate}

Fu and Wan partly determined the local structure of $\Sym^k\Kl_{n+1}$ in \cite[Thm.\,2.5\,\&\,Thm.\,3.1]{fu2005functions}.

\begin{thm}[Fu--Wan]\label{thm::local-inv-infty}
	\begin{enumerate}
		\item If $p\nmid n+1$ and $2n\mid q-1$, we have an isomorphism
	of $\frob_{q}$-modules
		\begin{equation*}
		 	\begin{split}
				(\Sym^k\Kl_{n+1}\mid_{\eta_\infty\otimes \fqq})^{I_{\overline \infty}}\q{\tfrac{nk}{2}}\simeq
			 	\begin{cases}
			 	E^{\oplus a(k,n+1,p)} & 2\mid n,\\
			 	0 & 2\nmid nk,\\
			 	E^{\bigoplus b(k,n+1,p)} & 2\nmid n \text{ and } 2\mid k.
			 	\end{cases}
		 	\end{split}
		 \end{equation*} 
		\item The Swan conductor of $\Sym^k\Kl_{n+1}$ at $\infty$ is $\frac{1}{n+1}\bigl(\binom{n+k}{n}-d(k,n+1,p)\bigr)$.
	\end{enumerate}
\end{thm}

Similar to \cref{prop::more-example-at-0}, we study the local structures $\Kl^{(2,1)}_{3}$ and $\Kl^{(2,2)}_{3}$ at $\infty$.

\begin{prop}\label{prop::more-example-at-infty}
	\begin{enumerate}[label=\theenumi., ref= \ref{prop::more-example-at-infty}.(\theenumi)]
		\item The Swan conductor of $\Kl_3^{(2,1)}$ at $\infty$ is $5$ if $p\neq 2,3,7$, and is $4$ if $p=2,7$. The dimension of the invariants $(\Kl_3^{(2,1)}\mid_{\bar\eta_\infty})^{I_{\overline \infty}}$ is $0$ if $p\neq 2,7$, and is $1$ if $p=2,7$.
		\item The Swan conductor of $\Kl_3^{(2,2)}$ at $\infty$ is $8$ if $p\neq 2,3$, and is $6$ if $p=2$. The dimension of the invariants $(\Kl_3^{(2,2)}\mid_{\bar\eta_\infty})^{I_{\overline \infty}}$ is $1$ if $p\neq 2,3$, and is $3$ if $p=2$.
	\end{enumerate}
\end{prop}
\begin{proof}
	The proof is similar to that of \cite[Thm.\,3.1]{fu2005functions}. We provide proof for the first statement and omit the proof of the second one.

	\noindent\textbf{Swan conductors}: According to the alternative description \eqref{eq::sheaf-repr-sl3}, it suffices to compute the Swan conductors of $\Sym^{2}\Kl_3 \otimes \Kl_3^\vee$ and $\Kl_3$. When $3\neq p$, after passing to a finite extension $k$ of $\fqq$, by Lemma 1.5 in \textit{loc. cit.}, we have
		$$[3]^*\Kl_3|_{\eta_\infty\otimes k }(1)=\scr{L}_{\psi_k(3t)} \oplus \scr{L}_{\psi_k(3\zeta_3 t)}\oplus \scr{L}_{\psi_k(3\zeta _3^2t)}, $$
	where $[3]\colon \bb{G}_m\to \bb{G}_m$ is the cubic map and $\zeta_3$ is a primitive third root of unity in $\bfp$.

	Then we can get the local structure of $[3]^*(\Sym^{\lambda_1}\Kl_3 \otimes \Sym^{\lambda_2}\Kl_3^\vee)$ as $\bigoplus_{i=1}^N \scr{L}_{\psi(C_it)}$ for some $N\in \bb{N}$ and some $C_i\in \bfp$. Since each $\scr{L}_{\psi(C_it)}$ has Swan conductor $1$ if $C_i\neq 0$, and has Swan conductor $0$ if $C_i=0$, we conclude that 
		$$\rr{Sw}_\infty([3]^*(\Sym^{\lambda_1}\Kl_3 \otimes \Sym^{\lambda_2}\Kl_3^\vee))=\{i\mid C_i\neq 0\}.$$ 
	By \cite[1.13.1]{katz1988gauss}, the Swan conductor of $\Sym^{\lambda_1}\Kl_3 \otimes \Sym^{\lambda_2}\Kl_3^\vee$ is thus $\{i\mid C_i\neq 0\}/3$. 

	By direct computation, we obtain 
		\begin{equation}\label{eq:Kl3-21-infty}
			\begin{split}
				[3]^*(\Sym^{2}\Kl_3 \otimes\Kl_3^\vee|_{\eta_\infty\otimes k})(3)
				&= \scr{L}_{\psi(3t)} ^{\oplus 3} \oplus \scr{L}_{\psi(3\zeta_3 t)}^{\oplus 3} \oplus \scr{L}_{\psi(3\zeta_3^2 t)}^{\oplus 3}\\
				&\oplus \scr{L}_{\psi(-6t)} \oplus \scr{L}_{\psi(-6\zeta_3 t)} \oplus \scr{L}_{\psi(-6\zeta_3^2 t)}\\
				&\oplus \scr{L}_{\psi(3(2-\zeta_3)t)} \oplus \scr{L}_{\psi(3\zeta_3(2-\zeta_3) t)} \oplus \scr{L}_{\psi(3\zeta_3^2(2-\zeta_3) t)}\\
				&\oplus \scr{L}_{\psi(3(2\zeta_3-1)t)} \oplus \scr{L}_{\psi(3\zeta_3(2\zeta_3-1) t)} \oplus \scr{L}_{\psi(3\zeta_3^2(2\zeta_3-1) t)}.
			\end{split}
		\end{equation}
		Depending on the value of $p$, the Swan conductor of $\Kl_3^{(2,1)}$ can be computed as follows.
	\begin{itemize}
		\item If $p\neq 2,7$, the numbers $C$ appearing in components $\scr{L}_{\psi(Ct)}$ in \eqref{eq:Kl3-21-infty} are all nonzero. So $\sw_\infty(\Kl_3^{(2,1)})= \rk(\Sym^2\Kl_3\otimes \Kl_3^\vee)/3-\sw_\infty(\Kl_3)=6-1=5$.
		\item If $p=2$, then only $6,6\zeta_3$ and $6\zeta_3^2$ are $0$ in $\bfp$. So $\sw_\infty(\Kl_3^{(2,1)})=4$.
		\item If $p=7$, we can take $\zeta_3=2$. So only $2-\zeta_3,\zeta_3(2-\zeta_3)$ and $\zeta_3^2(2-\zeta_3)$ are $0$ in $\bfp$. Hence, $\sw_\infty(\Kl_3^{(2,1)})= 4$.
	\end{itemize}

\noindent\textbf{Dimension of the invariants}:
Let $k$ be a finite extension of $\fqq$ containing $\zeta_3$. Consider the extension
		$k(t)=k(z)[t]/(t^3-z)$
of $k(z)$, and the extension
		$k(y)=k(t)[y]/(y^q-y-t)$
of $k(t)$. The Galois group $H=\gal(k(y)/k(t))$ is isomorphic to $\fqq$, and is a normal subgroup of $G=\gal(k(y)/k(z))$. The quotient $G/H$ is 
	$\gal(k(t)/k(z))=\bb{Z}/2\bb{Z}.$
	
For each $a\in \fqq$, we denote by $g_{a}$ the element in $H$, such that $g_{a}\cdot y=y-a .$ We choose an element $g\in G$ such that $g\cdot y=\zeta_3 y$. It follows that $g^3=g_{0}=\id$ and $g\not \in H$.

Let $W$ be a one-dimensional $E$-vector space and choose $v_0$ as a basis. We define an action of $H$ on $W$ by 
	$$g_{a}\cdot v_0=\psi_k(-3a)v_0.$$
By the construction, as an $H$-representation, $W$ is isomorphic to $\scr{L}_{\psi_k(3t)}$. Then the induced $G$-representation
		$$V:=\mathrm{Ind}_H^GW=\oplus_{i=0}^2 g^iW,$$ 
is identified with $[3]_*\big(\scr{L}_{\psi_k(3t)}\big|_{\eta_\infty\otimes k}\big)$. Let $v_i:=g^i v_0$. The set $\{v_0,v_1,v_2\}$ form a basis of $V$, and the action of $H$ on $v_i$ is given by
		$$\begin{aligned}
			g_{a}\cdot v_i
			=g^i\cdot g^{-i} g_{a,\mu}g^i\cdot v_0
			=g^i\cdot g_{\zeta_3^ia}\cdot v_0
			=\psi_k(-3\zeta_3^ia)v_i,
		\end{aligned}$$
	and the action of $g$ on $V$ is given by $gv_i=v_{i+1}$ where $v_3=v_0$.
	
	It follows that $\{v_av_b\otimes v_c^\vee\mid a\leq b,\, 0\leq a,b,c\leq 2\}$ form a basis of $\Sym^2V\otimes V^\vee=\Sym^2\Kl_3\otimes \Kl_3^\vee|_{\eta_\infty\otimes k}(3)$. To calculate the dimension of $(\Kl_3^{(2,1)}\mid_{\bar\eta_\infty})^{I_{\overline \infty}}$, it suffices to calculate the dimension of the $G$-invariant subspace 
		$$(\Sym^2V\otimes V^\vee)^G.$$
	Let $w=\sum_{a,b,c} \alpha_{a,b,c}v_av_b\otimes v_c^\vee$, then 
		$$g_a\cdot w=\sum_{a,b,c} \psi_k(-3(\zeta_3^a+\zeta_3^b-\zeta_3^c)a)\alpha_{a,b,c}v_av_b\otimes v_c^\vee,$$
	and 
		$$g\cdot w=\sum_{a,b,c} \alpha_{a+1,b+1,c+1}v_av_b\otimes v_c^\vee.$$
	\begin{itemize}
		\item If $p\neq 2,3,7$, then $(\zeta_3^a+\zeta_3^b-\zeta_3^c)$ is never $0$ in $\bfp$. So there are no fixed vectors in $\Sym^2V\otimes V^\vee$. 
		\item If $p=2$, then $w=\sum_{i=0}^2  v_iv_{i+1}\otimes v_{i+2}^{\vee}$ spans $(\Sym^2V\otimes V^\vee)^G$. 
		\item If $p=7$, then $\zeta_3=2$ in this case, and $w=\sum_{i=0}^2   v_iv_{i}\otimes v_{i+2}^{\vee}$ spans $(\Sym^2V\otimes V^\vee)^G$. 
	\end{itemize}
	In conclusion, the dimension of $(\Kl_3^{(2,1)}\mid_{\bar\eta_\infty})^{I_{\overline \infty}}$ is $0$ if $p\neq 2,3,7$ and is $1$ if $p=2,7$.
\end{proof}
\begin{rek}
	In Section~\ref{sec::dim_mom_p=3}, we will determine the local monodromy group of $\Kl_3$ at $p=3$. As consequence, we can prove that when $p=3$ the Swan conductor of $\Kl_3^{(2,1)}$ at $\infty$ is $5$ and the dimension of the invariants $(\Kl_3^{(2,1)}\mid_{\bar\eta_\infty})^{I_{\overline \infty}}$ is $0$. The argument is similar to those of \cref{cor::dimention_p=3} and \cref{eq::local-inv-infinity-p=3}.
\end{rek}

\subsection{The dimensions of the middle \texorpdfstring{$\ell$}{l}-adic cohomology}\label{sec::dim-middle}
In this subsection, our objective is to calculate the dimension of the middle $\ell$-adic cohomology  of $\Sym^k\Kl_{n+1}.$ Proposition~\ref{prop::moments-p-large} provides the dimensions when $p$ is coprime to $n+1$. 

However, the case that $p\mid n+1$ remains mysterious because the local monodromy group of $\Kl_{n+1}$ is still unknown. When $n=1$, the dimension in the case of $p=2$ was computed in \cite[Cor.\,4.3.5]{yun2015galois}. Following his method, we give a dimension formula when $n=2$ and $p=3$ in Section~\ref{sec::dim_mom_p=3}. The key idea is to use the complete classification of finite subgroups of $\SL_3$ to find the local monodromy group at $\infty$ of $\Kl_3$.

\subsubsection{When \texorpdfstring{$\gcd(p,n+1)=1$}{gcd(p, n+1)=1}}\label{sec::dim_mom_p-neq-3}

\begin{prop}\label{prop::moments-p-large}
	When $p$ is coprime to $n+1$, the formula of the dimension of the middle $\ell$-adic cohomology $\mathrm{H}^1_{\acute{e}t,\mathrm{mid}}(\bb{G}_{m,\bar{\bb{F}}_p},\Sym^k\Kl_{n+1})$ is 
	\begin{equation*}
		\begin{split}
			\frac{1}{n+1}\q{\binom{k+n}{n}-d(k,n+1,p)}-\sum_{u=0}^{\lfloor\frac{nk}{2}\rfloor}m_k(u)+ 2\delta\q{k,p}-
			\begin{cases}
				a(k,n+1,p) & 2\mid n,\\
				0 & 2\nmid n k,\\
				 b(k,n+1,p) 
				& \text{else},
			\end{cases} 
		\end{split}
	\end{equation*}
where the number $\delta\q{k,p}$ is 
	$\begin{cases} 
		1 &p=2 \text{ and } 2\mid k,\\ 
		0 & \text{ else },
	\end{cases}$ 
the numbers $m_k(u)$ are defined in \eqref{eq::counting_inv_at_0}, and the numbers $d(k,n+1,p)$, $a(k,n+1,p)$, and $b(k,n+1,p)$ are defined in Section~\ref{nota::multi-indices}.
\end{prop}
\begin{proof}
By the long exact sequence (\ref{eq::long-exact-sequence}), the dimension of $\mathrm{H}^1_{\acute{e}t,\mathrm{mid}}(\bb{G}_{m,\bar{\bb{F}}_p},\Sym^k\Kl_{n+1})$ is given by
	\begin{equation*}
		\begin{split}
			\mathrm{H}^1_{\acute{e}t}(\bb{G}_{m,\bar{\bb{F}}_p},\Sym^k\Kl_{n+1})+\dim \Sym^k\Kl_{n+1}^{G_{\mathrm{geom}}}-\dim \Sym^k\Kl_{n+1}^{I_{\bar 0}}-\dim \Sym^k\Kl_{n+1}^{I_{\bar \infty}}.
		\end{split}
	\end{equation*}

By \eqref{eq:GOS-formula}, $\dim \mathrm{H}^1_{\acute{e}t}(\bb{G}_{m,\bar{\bb{F}}_p},\Sym^k\Kl_{n+1})-\mathrm{H}^0_{\acute{e}t}(\bb{G}_{m,\bar{\bb{F}}_p},\Sym^k\Kl_{n+1})=\rr{Sw}(\Sym^k\Kl_{n+1})$, which is calculated in \cref{thm::local-inv-infty}. As for the invariants of the global monodromy group, it is $E(nk/2)$ if $p=2$ and $k$ is even, and $0$ otherwise by combining \cite[Thm.\ 11.1]{katz1988gauss} and \cite[Lem.\ 0.2]{fu2006trivial}. Next, the dimensions of the invariants of the inertia groups at $0$ and $\infty$ are summarized in \cref{thm::local-symk-at-0} and \cref{thm::local-inv-infty} if $p\nmid n+1$. At last, combining everything together, we get the dimension of the middle cohomology.
\end{proof}

\subsubsection{{When \texorpdfstring{$n=2$ and $p=3$}{n=2 and p=3}}}\label{sec::dim_mom_p=3}

\paragraph{The classification of finite subgroups of \texorpdfstring{$\mathrm{SL_3}$}{SL3}}
Let $\zeta_9$ be a primitive ninth root of unity, and we put $\omega=\zeta_9^6$ and $
\varepsilon=\zeta_9^4$. We define the following matrices in $\SL_3(\bb{C})$
	$$S=\left(\begin{matrix}1&0&0\\0&\omega&0\\0&0&\omega^2\end{matrix}\right), \ T=\left(\begin{matrix}0&1&0\\0&0&1\\1&0&0\end{matrix}\right),$$
	$$U=\left(\begin{matrix}\varepsilon&0&0\\0&\varepsilon&0\\0&0&\varepsilon\omega\end{matrix}\right),\ V=\frac{1}{\omega-\omega^2}\left(\begin{matrix}1&1&1\\1&\omega&\omega^2\\1&\omega^2&\omega\end{matrix}\right).$$
Let 
	$$G_{108}=<S,T,V>\subset \mathrm{SL}_3,$$
	$$G_{216}=<S,T,V,UVU\inv>\subset \mathrm{SL}_3,$$
and
	$$G_{648}=<S,T,V,U>\subset \mathrm{PGL}_3.$$
We summarize the complete classification of solvable finite subgroups of $\SL_3(\bb{C})$ from \cite[Ch. XII]{millerBlichfeldtDickson} in the following Theorem.
\begin{thm}\label{thm::classification-finite-subgroups}
If $G$ is a finite solvable subgroup of $\mathrm{SL}_3(\bb{C})$, it is isomorphic to one of the following groups:
\begin{itemize}
	\item[(A).] Diagonal abelian groups,
	\item[(B).] Groups arising from finite subgroups of $\mathrm{GL}_2$,
	\item[(C).] Groups generated by groups of type (A) and the element $T$,
	\item[(D).] Groups generated by groups of type (C) and a matrix of the form 
				$$Q_{a,b,c}:=\left(\begin{matrix}
					a &0&0\\
					0&0&b\\
					0&c&0
				\end{matrix}\right)$$
			for some roots of unity satisfying $abc=-1$,
	\item[(E).] The group $G_{108}$,
	\item[(F).] The group $G_{216}$,
	\item[(G).] The group $G_{648}$.
\end{itemize}
\end{thm}

\paragraph{The local monodromy at \texorpdfstring{$\infty$}{infinity} when \texorpdfstring{$p=3$}{p=3}}
Let $j\colon \bb{G}_{m,\bb{F}_3}\hookrightarrow \bb{P}^1_{\bb{F}_3}$ be the inclusion. Restricting the $\ell$-adic sheaf $j_*\Kl_3$ to $\eta_\infty$, we have a representation $\rho\colon I_{\bar \infty}\to \SL_3(\bql)$ of the inertia group at $\infty$. Recall that $\Kl_3$ is totally wild at $\infty$ with Swan conductor $1$. We want to determine the local monodromy group of $\Kl_3$ at $\infty$, namely the finite solvable subgroup $D_0 = \rho(I_{\bar \infty})$ of $\SL_3$. The group admits a lower numbering filtration $\{D_i\}$ terminating at $D_N$, such that $\#D_0/D_1$ is coprime to $3$, $D_1$ is the $3$-Sylow subgroup of $D_0$, and $D_i/D_{i+1}$ are cyclic abelian of order $3$ for $i\geq 1$. 

\begin{thm}\label{thm::Kl_3_monodromy_infty_p=3}
The image of $I_{\bar \infty}$ under $\rho$ is isomorphic $G_{108}$, whose lower numbering filtration is given by
$$D_0\rhd D_1=<S,T>\rhd D_{2}=\cdots =D_4=<\omega \mathrm{I}_3>\rhd \{1\}. $$
\end{thm}

\begin{proof} 

By \cite[Lem.\,1.19]{katz1988gauss}, the local monodromy group $D_0=\rho(\infty)$ satisfies the following conditions: 
\begin{itemize}
	\item[(a).] $D_0$ acts on $V=\Kl_3\mid_{\bar\eta_\infty}$ irreducibly,
	\item[(b).] $D_0$ admits no faithful $\bql$-linear representation of dimension smaller than $3$.
\end{itemize}

The groups of type (A) are abelian groups. As irreducible representations of abelian groups are all one-dimensional, the group $D_0$ cannot be isomorphic to the groups of type (A) due to condition (a). The groups of type (B) are groups induced from subgroups of $\GL_2$, which admit faithful $\bql$-linear representations of dimension $2$, which violates condition (b).

We establish the following lemma to eliminate more possibilities.

\begin{lem} \label{lem::claim1}
Let $V=\Kl_3\mid_{\bar\eta_\infty}$. Then the Swan conductor of $\Sym^3V$ is $2+\dim (\Sym^3V)^{I_{\bar \infty}}$.
\end{lem}
\begin{proof}
As the symmetric power of the standard representation of $\SL_3$ is irreducible, the invariants $(\Sym^3\Kl_3)^{\SL_3}$ (isomorphic to $\mathrm{H}^0_{\et}(\bb{G}_{m,\bar{\bb{F}}_3},\Sym^k\Kl_3)$) is $0$. By the Grothendieck--Ogg--Shafarevich formula, the dimension of $\mathrm{H}^1_{\et}(\bb{G}_{m,\bar{\bb{F}}_3},\Sym^3\Kl_3)$ is equal to the Swan conductor of $\Sym^3V=(\Sym^3\Kl_3)_{\bar \eta_\infty}$, which is smaller or equal to $\left\lfloor\frac{1}{3}\cdot \rk \,\Sym^3\Kl_3\right\rfloor=3$ because the breaks of $\Sym^3\Kl_3$ are at most $\frac{1}{3}$. 

Considering the long exact sequence \eqref{eq::long-exact-sequence}, we have
\begin{equation*}
\begin{split}
	0\to (\Sym^3\Kl_3)^{\SL_3}&\to (\Sym^3\Kl_3\mid_{\eta_0})^{I_{\bar 0}}\oplus (\Sym^3\Kl_3\mid_{\eta_{\infty}})^{I_{\bar\infty}}\\
	&\to \mathrm{H}^1_{\et,\rr{c}}(\bb{G}_{m,\bar{\bb{F}}_3},\Sym^3\Kl_3)\to \mathrm{H}^1_{\et,\mathrm{mid}}(\bb{G}_{m,\bar{\bb{F}}_3},\Sym^3\Kl_3)\to 0.
\end{split}
\end{equation*}

Recalling that the dimension of $(\Sym^3\Kl_3\mid_{\eta_0})^{I_{\bar 0}}$ is $2$ by Theorem~\ref{thm::local-symk-at-0}, we deduce from the exact sequence that  
	$$3\geq \dim \mathrm{H}^1_{\et}(\bb{G}_{m,\bar{\bb{F}}_3},\Sym^3\Kl_3)=2+\dim \mathrm{H}^1_{\et,\mathrm{mid}}(\bb{G}_{m,\bar{\bb{F}}_3},\Sym^3\Kl_3)+\dim (\Sym^3\Kl_3\mid_{\eta_\infty})^{I_{\bar \infty}}.$$
If the middle cohomology is nonzero, it is one-dimensional.
By the computations in \cref{app::m_3^4(p)}, we obtain
	$$\tr(\frob\mid \mathrm{H}^1_{\et,\mathrm{mid}}(\bb{G}_{m,\bar{\bb{F}}_3},\Sym^3\Kl_3))=-(m_3^3(3)+1+p^2)=0.$$
We arrive at a contradiction, as $\mathrm{H}^1_{\et,\mathrm{mid}}(\bb{G}_{m,\bar{\bb{F}}_3},\Sym^3\Kl_3)$ is pure of weight $10$.  Consequently, the Swan conductor is $2+\dim (\Sym^3V)^{I_{\bar \infty}}$.
\end{proof}

Now assume that $D_0$ is of type (C) or type (D). The representation $\Sym^3V$ is the direct sum of three subrepresentations $V_1=\mathrm{span}\{v_0^3,v_1^3,v_2^3\}$, $V_2=\mathrm{span}\{v_i^2v_j\}_{i\neq j}$ and $V_3=\mathrm{span}\{v_0v_1v_2\}$. 
\begin{itemize}
	\item If $D_0$ is of type (C), the action of $D_0$ has fixed vectors in each $V_i$. So $\dim(\Sym^3V)^{I_{\bar \infty}}\geq 3$. Applying Lemma~\ref{lem::claim1}, we find that
		$$5\leq 2+\dim(\Sym^3V)^{I_{\bar \infty}}=\sw(\Sym^3V)\leq 3.$$
	\item If $D_0$ is of type (D), the operators $T$ and $Q_{a,b,c}$ have no fixed vectors in each $V_i$. As a result, the subspace of invariants $(\Sym^3V)^{I_{\bar \infty}}$ has dimension $0$ and $V_i$ are all totally wild. Using Lemma~\ref{lem::claim1}, we deduce that
		$$2=\sw(\Sym^3V)=\sum_{i=1}^3\sw(V_i)\geq 3,$$
	which is again not possible.
\end{itemize}

The group $D_0$ also cannot be isomorphic to groups of type (G) because $G_{648}$ has no normal subgroup of order $81$, i.e. a normal $3$-Sylow subgroup. The possible orders of normal subgroups of $G_{648}$ are $1,3,27,54,216$ and $648$ as determined by a group theoretic computation.

Now the remaining cases are the groups of type (E) and (F). 

\begin{lem}\label{lem::claim2}
If $D_0$ is of type (E) or (F), the Swan conductor of $\Sym^6V$ is $6$.
\end{lem}

\begin{proof}
From the above discussion, the group $D_0$ is  either the group $G_{108}$ or $G_{216}$. In both cases, the $3$-Sylow subgroup $D_1$ of $D_0$ is generated by matrices $S$ and $T$, of order $27$. The group $D_1$ has only $3$ subgroups of order $9$. They are  
	$$H_1=<S,\omega I_3>,\quad H_2=<T, \omega I_3>\quad \text{and}\quad H_3=<ST,\omega I_3>.$$
Since $V$ is totally wild, the last nontrivial group $D_N$ in the ramification filtration has no invariant vectors, i.e. $V^{D_N}=0$. Since $S,ST$ and $T$ have nonzero fixed vectors $v_1, v_1+w^2v_2+v_3$ and $v_1+v_2+v_3$ respectively, the group $D_N$ is either $H_i$ or $<\omega I_3>$.

There exist nonnegative integers $a,b,c$ such that the lower numbering filtration is of the form
	$$D_0\rhd D_1=\cdots =D_a\rhd\cdots\rhd D_{a+b}\rhd\cdots=D_{a+b+c}\rhd \{1\}, $$
where $D_{a+1}=\cdots D_{a+b}$ is $H_1$, $H_2$ or $H_3$ if $b\neq 0$, and $D_{a+b+1}=\cdots=D_{a+b+c}=<\omega I_3>$ if $c\neq 0$. 

We know that $\sw(V)=1$ and $\sw(\Sym^3V)=2$ or $3$ according to (\ref{lem::claim1}). Then 
\begin{equation*}
 \begin{split}
 	1=\sw(V)=\sum_{i=1}^{\infty}\frac{\dim V-\dim V^{D_i}}{[D_0:D_i]}=\frac{1}{[D_0:D_1]}\q{3*a+3*\frac{b}{3}+3*\frac{c}{9}}
 \end{split}
 \end{equation*}
 and
 \begin{equation*}
 \begin{split}
 	2\text{ or }3=\sw(\Sym^3V)=\sum_{i=1}^{\infty}\frac{\dim \Sym^3V-\dim \Sym^3V^{D_i}}{[D_0:D_i]}=\frac{1}{[D_0:D_1]}\q{8*a+6*\frac{b}{3}+0*\frac{c}{9}}.
 \end{split}
 \end{equation*}
If $D_0=G_{108}$, the only possibility is $(a,b,c)=(1,0,3)$. If $D_0=G_{216}$, we have two possibilities $(a,b,c)=(2,0,6)$ or $(1,4,3)$. In all cases, the number $c$ is nonzero, and the last nontrivial group $D_N=D_{a+b+c}$ is $<\omega I_3>$ of order $3$. Also, we obtain in all cases that $\sw(\Sym^3V)=2$.

Now consider the Swan conductor of $\Sym^6V=\Sym^6(\Kl_3)_{\bar \eta_\infty}$. In this case $D_{N}$ acts trivially on $\Sym^6V$, so it suffices to compute $\dim\Sym^6V^{D_1}$ and $\dim\Sym^6V^{H_i}$ (if $b\neq 0$, $D_{a+1}=H_i$ for some $i\in \{1,2,3\}$). Let $\{v_i\}_{i=0,1,2}$ be the canonical basis of $V$ and $f_i=v_0+\omega^i v_1+\omega^{2i}v_2$ for $i=0,1,2$. Then the actions of $S$ and $T$ on the basis $\{f_i\}$ are 
	$Sf_i=f_{i+1}$ and $Tf_i=\omega^{-i}f_i$,
where $f_3:=f_0$. 

Consider the set of multi-indexes 	
	$$A:=\{\underline I=(I_0,I_1,I_2)\in \bb{Z}_{\geq 0}^3\mid |\underline I|:=I_0+I_1+I_2=6\},$$ 
on which $\sigma=(123)\in S_3$ acts. For any vector $f=\sum_{\underline I\in A}a_{\underline I}f^{\underline I}$ in $\Sym^kV$, we have
	$$Sf=\sum_{\underline I\in A}a_{\sigma\inv \underline I}f^{\underline I}\quad \text{and}\quad Tf=\sum_{\underline I\in A}a_{\underline I}\omega^{I_2-I_1}f^{\underline I}.$$
So if $f\in \Sym^6V^{D_1}$, i.e., $Sf=Tf=f$, the vector $f$ is contained in the span of $\{\sum_{i=0}^2f^{\sigma\underline I}\mid I_0\equiv I_1\equiv I_2\mod 3\}$. The dimension of the subspace of invariants $(\Sym^3V)^{D_{1}}$ is $4$. Similarly, we can compute that $\dim\Sym^6V^{H_i}=10$ for $i\in \{1,2,3\}$.

In conclusion, for $(a,b,c)=(1,0,3)$, $(2,0,6)$, and $(1,4,3)$, the Swan conductors are
	\begin{equation*}
	\begin{split}
		\frac{1}{4}\q{24*1+18*0+0*\frac{3}{9}}=\frac{1}{8}\q{24*2+18*0+0*\frac{6}{9}}=\frac{1}{8}\q{24*1+18*\frac{4}{3}+0*\frac{3}{9}}=6.
	\end{split}
	\end{equation*}
\end{proof}

\begin{lem}\label{lem::claim3}
The dimension of $(\Sym^6V)^{I_{\bar \infty}}$ is $2$ if $D_0=G_{108}$ and is $1$ if $D_0=G_{216}$.
\end{lem}
\begin{proof}

Let $D_0=\rho(I_{\infty})$ be either $G_{108}$ or $G_{216}$, which is a normal subgroup of $G=G_{216}$ or $G=G_{648}$ respectively. Let $\bar{a}\in G/D_0$ be the class of an element $a\in G$, then
\begin{equation}\label{eq::mean_trace}
\begin{split}
	\tr(\bar a \mid (\Sym^kV)^{I_{\bar \infty}})=\frac{1}{\#D_0}\sum_{g\in a\cdot D_0}\tr(g\mid \Sym^k V).
\end{split}
\end{equation}
In particular, if we let $a=1$, we get
\begin{equation}\label{eq::mean_trace_a=1}
\begin{split}
	\sum \dim (\Sym^kV)^{I_{\bar \infty}} x^k=\frac{1}{\#D_0}\sum_{g\in D_0}\frac{-1}{P_{g,V}(x)}
\end{split}
\end{equation} 
where $P_{g,V}(x)=\det(x\cdot g-1\mid V)$ is the characteristic polynomial of $g$. This can be easily computed by Sagemath\footnote{The code can be found on \href{https://yichenqin.net}{my web page}.} \cite{sagemath}. Therefore, we deduce
\begin{equation}\label{eq::sagemath-polynomial}
\begin{split}
	\begin{cases}
		P(x)=-\dfrac{1 - x^3 + x^6 + x^{12} - x^{15} + x^{18}}{(-1 + x^3)^3 (1 + 
    x^3)^2 (1 + x^6)} & \text{if } D_0=G_{108};\\
    	\tilde P(x)=-\dfrac{1 - x^3 + x^9 - x^{15} + x^{18}}{(-1 + x^3)^3 (1 + x^3)^2 (1 + x^6)} & \text{ if } D_0=G_{216}.
	\end{cases}
\end{split}
\end{equation}
In particular, their coefficients of $t^6$ are $2$ and $1$, respectively. 
\end{proof}

By \cref{app::m_3^4(p)}, we have 
\begin{equation*}
\begin{split}
	\tr(\frob\mid \mathrm{H}^1_{\et}(\mathbb{G}_{m,\bar{\bb{F}}_3},\Sym^6\Kl_3))=-820.
\end{split}
\end{equation*}
Combining Theorem~\ref{thm::local-symk-at-0} and \cite[Thm.\,7.0.7]{katz1988gauss}, we deduce that
\begin{equation*}
\begin{split}
	\tr(\frob\mid (\Sym^6\Kl_3)^{I_{\bar 0}})=1+p^2+p^4+p^6=820.
\end{split}
\end{equation*}
Using the long exact sequence (\ref{eq::long-exact-sequence}), we conclude that 
\begin{equation}\label{eq::trace-sym6-inv}
\begin{split}
	\tr(\frob\mid \mathrm{H}^1_{\et,\mathrm{mid}}(\mathbb{G}_{m,\bar{\bb{F}}_3},\Sym^6\Kl_3))=-\tr(\frob\mid (\Sym^6\Kl_3)^{I_{\bar \infty}}),
\end{split}
\end{equation}
and 
\begin{equation*}
\begin{split}
	\dim\mathrm{H}^1_{\et,\mathrm{mid}}(\mathbb{G}_{m,\bar{\bb{F}}_3},\Sym^6\Kl_3)=2-\dim (\Sym^6\Kl_3)^{I_{\bar \infty}}.
\end{split}
\end{equation*}

If $D_0=G_{216}$, then both $\dim (\Sym^6\Kl_3)^{I_{\bar \infty}}$ and the middle cohomology are one-dimensional. However, by \eqref{eq::trace-sym6-inv}, since $(\Sym^6\Kl_3)^{I_{\bar \infty}}$ is pure of weight $12$ and $\mathrm{H}^1_{\et,\rr{mid}}(\mathbb{G}_{m,\bar{\bb{F}}_3},\Sym^6\Kl_3)$ is pure of weight $13$, we get a contradiction.

In conclusion, the only possibility is $D_0=G_{108}$. The ramification filtration of $D_0$ is given in terms of the triple $(1,0,3)$ in the proof of Lemma~\ref{lem::claim2}.
\end{proof}

\paragraph{The dimension of the middle cohomology}

\begin{prop}\label{cor::dimention_p=3}
	When $p=3$, the Swan conductor of the action of $I_{\bar \infty}$ on $(\Sym^k\Kl_3)\mid_{\eta_\infty}$ is given by
		$$\mathrm{Swan}_\infty(\Sym^k\Kl_3)=
			\begin{cases}
				\frac{1}{3}\binom{k+2}{2} &3\nmid k,\\
				\frac{1}{4}\q{\binom{k+2}{2}-\frac{d(k,3,3)+2}{3}} &3\mid k.
			\end{cases}$$
\end{prop}
\begin{proof}
Recall that we denote $V=\Kl_3\mid_{\eta_\infty}$. If $3\nmid k$, then there is no fixed vector of $\Sym^kV$ under the action of the group $<\omega I_3>$. So the Swan conductor can be expressed as
		$$\sum_{i=1}^\infty \frac{\dim \Sym^kV-0}{[D_0:D_i]}=\frac{\dim \Sym^kV}{3}\cdot \sum \frac{3}{[D_0:D_i]}=\frac{1}{3}\cdot \binom{k+2}{2}.$$

If $3\mid k$, the situation is similar to the case where $k=6$. In this case $D_{4}=<\omega I_3>$ acts trivially on $\Sym^kV$. The dimension of $\Sym^kV^{D_1}$ is computed in terms of invariant vectors under the action of $S$ and $T$. We again let $\{v_i\}_{i=0,1,2}$ be the canonical basis of $V$ and $f_i=v_0+\omega^i v_1+\omega^{2i}v_2$ for $i=0,1,2$. If $Sf=Tf=f$, the vector $f$ is contained in the span of the set $\{\sum_{i=0}^2f^{\sigma\underline I}\mid I_0\equiv I_1\equiv I_2\mod 3\}$. The dimension of the invariants of $S$ and $T$ is exactly the number $\frac{d(k,3,3)-2}{3}+1=\frac{d(k,3,3)+2}{3}$, where $d(k,3,3)$ is introduced in Section~\ref{nota::multi-indices}. In conclusion, the Swan conductor is given by
	\begin{equation*}
		\begin{split}
			\sum_{i=1}^\infty\frac{\dim \Sym^kV-\dim\Sym^kV^{D_i}}{[D_0:D_i]}=\frac{1}{4}\q{\binom{k+2}{2}-\frac{d(k,3,3)+2}{3}}.
		\end{split}
	\end{equation*}
\end{proof}

\begin{prop}\label{eq::local-inv-infinity-p=3}
	The invariants of  the inertia group are given by
	\begin{equation*}
		\begin{split}
			(\Sym^kV)^{I_{\bar \infty}}=\bql(-k)^{\oplus \tilde p_k}\oplus \scr{L}_{\theta}(-k)^{\oplus p_k-\tilde p_k},
		\end{split}
	\end{equation*}
	where $\theta$ is an unramified character which sends Frobenius to $-1$, and $p_k$ and $\tilde{p}_k$ are the $k$-th coefficients of the generating series $P(x)$ and $\tilde P(x)$ from \eqref{eq::sagemath-polynomial}. In particular, the dimension of $(\Sym^kV)^{I_{\bar \infty}}$ is $p_k$.
\end{prop}
\begin{proof}
Let $\phi$ be a lifting of the image of $\frob_\infty$ in $\GL_3$ and $\phi_1=\frac{1}{3}\phi$ in $\SL_3$. Since $\phi$ normalizes $D_0=G_{108}$, it is in the normalizer of $G_{108}$ in $\SL_3$, i.e. $G_{216}$. By direct computation, we find that $G_{216}/D_1$ is the quaternion group $Q_8$ and $D_0/D_1$ is a cyclic group. Notice that $\bar{\phi}_1\inv g\bar{\phi}_1=g^3$ for $g\in Q_8$, which implies that $\phi_1 \not \in D_0$. In \eqref{eq::mean_trace} we let  $a=\phi_1$. Then we obtain
\begin{equation*}
\begin{split}
	Q(x):=\sum_{s=0}^\infty \tr(\phi_1\mid (\Sym^sV)^{I_{\bar \infty}})x^s=\frac{1}{108}\sum_{g\in \phi_1D_0}\frac{-1}{P_{g,V}(x)}.
\end{split}
\end{equation*}
As $\phi_1\not \in G_{108}$ and $G_{216}=G_{108}\cup \phi_1G_{108}$, the series $Q(x)$ is nothing but 
	$$2\tilde P(x)-P(x)=\frac{-1 + x^3 - x^6}{(-1 + x^3) (1 + x^6)}.$$
Let $p_k$ and $\tilde{p}_k$ be the $k$-th coefficient of $P(x)$ and $\tilde P(x)$ respectively. 

Notice that $\phi_1^2\in G_{108}$, because $[G_{216}:G_{108}]=2$. Thus, the eigenvalues of $\phi_1$ acting on $(\Sym^kV)^{I_{\bar \infty}}$ are $\pm 1$. Assume that the dimensions of eigenspaces of $1$ and $-1$ are $\lambda_1$ and $\lambda_{-1}$ respectively. Then $\lambda_1+\lambda_{-1}=\dim (\Sym^kV)^{I_{\bar \infty}}$ and $\lambda_1-\lambda_{-1}=2\tilde p_k-p_k$. Therefore, we deduce the desired decomposition in \cref{eq::local-inv-infinity-p=3}.
\end{proof}

\begin{cor}\label{cor::dimension_mid_p=3}
	When $p=3$, the dimension of the moments are given by
		$$\dim \mathrm{H}^1_{\acute{e}t,\mathrm{mid}}(\bb{G}_{m,\bar{\bb{F}}_p},\Sym^k\Kl_3)
		=\begin{cases}
			\frac{1}{3}\binom{k+2}{2}-\lfloor\frac{k+2}{2}\rfloor &3\nmid k;\\[5pt]
			\frac{1}{4}\bigl(\binom{k+1}{2}-\frac{d(k,3,3)+2}{3}\bigr)-\lfloor\frac{k+2}{2}\rfloor-p_k  &3\mid k.
		\end{cases}$$
\end{cor}

\section{Motives attached to Kloosterman moments}\label{sec::motives}

In this section, we aim to construct motives attached to moments of Kloosterman sheaves. Our approach generalizes the construction presented in \cite{fresan2018hodge} by the Weyl construction. Next, we investigate their de Rham realizations, $\ell$-adic realizations, and other realizations in characteristic $p > 0$.

\subsection{The construction of motives}\label{sec:motive-construction}
Let $n$ be an integer, $V_{ \lambda}$ the irreducible representation of the highest weight $\sum_i \lambda_i(L_1+\ldots+L_i)$, and $\mathcal{K}\subset \bb{G}_m^{n|\lambda|}$ the hypersurface defined by the equation 
	\begin{equation}\label{eq:hypersurface}
	    \sum_{i=1}^{|\lambda|}\biggl(\sum_{j=1}^nx_{i,j}+\frac{1}{\prod_{j=1}^nx_{i,j}}\biggr)=0.
	\end{equation}
The group $S_{|\lambda|}\times \mu_{n+1}$ acts on $\mathcal{K}$ by $(\sigma\times \mu)\cdot x_{i,j}:=\mu\cdot x_{\sigma(i),j}$. By a slight abuse of notation, we denote $P_{ \lambda}$ and $Q_{\lambda}$ as the groups $P_{\mu(\lambda)}$ and $Q_{\mu(\lambda)}$ from Section~\ref{sec::weyl-construction}, and put $G_\lambda=P_\lambda\times Q_\lambda$. Let $\chi_n\colon \rr{sign}^n\times \rr{sign}^{n+1}$ be the character of $G_\lambda$ and for each representation $V$ of $S_{|\lambda|}$, we denote the isotypic component with respect to 
	\begin{equation}\label{eq:isotypic-component}
		\frac{1}{\#G_\lambda}\sum_{\sigma\in P_\lambda} \rr{sign}(\sigma)^n\sigma\cdot \sum_{\tau\in Q_\lambda} \rr{sign}(\tau)^{n+1}\tau
	\end{equation}
by $V^{G_\lambda,\chi_n}$. Moreover, if a finite group $H$ acts on $V$ and commutes with $S_{|\lambda|}$, then we denote the isotypic component $(V^{G_\lambda,\chi_n})^H$ as $V^{G_\lambda\times H,\chi_n}$.

\begin{defn}\label{def::Kloosterman-motive}
	The \textit{motives attached to moments of $\Kl_{n+1}^\lambda$} are the Nori motives over $\bb{Q}$ with rational coefficients, of the form
		$$\mathrm{M}_{n+1}^\lambda:=\mathrm{gr}^W_{n|\lambda|+1}\mathrm{H}^{n|\lambda|-1}_{c}(\mathcal{K})^{G_\lambda\times \mu_{n+1},\chi_n}(-1),$$
	where $W_\bullet$ is the (motivic) weight filtration \cite[Thm.\,10.2.5]{Huber2017}, and the exponent $(G_\lambda\times \mu_{n+1},\chi_n)$ means taking the isotypic component with respect to \eqref{eq:isotypic-component} and the action of $\mu_{n+1}$ described above.
\end{defn}
\begin{rek}\label{rek::Galois-descent-motive}
The action of $\zeta_{n+1}\in \mu_{n+1}$ on $\cc{K}$ is not an automorphism defined over $\bb{Q}$ (only defined over $K=\bb{Q}(\zeta_{n+1})$). But taking the invariants of $\mu_{n+1}$ on $\rr{N}:=\mathrm{gr}^W_{n|\lambda|+1}\mathrm{H}^{n|\lambda|-1}_{c}(\mathcal{K})^{G_\lambda,\chi_n}(-1)$ still gives rise to a Nori motive over $\bb{Q}$. In fact, one can see the Nori motive $\rm N$ as a $\bb{Q}$-vector space together with an action of the motivic Galois group $G_{\rr{mot}}(\bb{Q})$. We restrict $\rr{N}$ to a Nori motive $\rr{N}_K$ over $K$, i.e., a $\bb{Q}$-vector space with an action of the motivic Galois group $G_{\rr{mot}}(K)$. Then one can consider a Nori motive over $K$
	$$\rr{N}^{\mu_{n+1}}:=\im(\rr{N}_K\xrightarrow{\varphi} \rr{N}_K),$$
where $\varphi=\frac{1}{\#\mu_{n+1}}\sum_{\zeta\in \mu_{n+1}}\zeta$. One can check that $\rr{N}^{\mu_{n+1}}$ is stable under the action of $\gal(K/\bb{Q})$. By \cite[Thm.\,9.1.16]{Huber2017}, the motive $\rr{N}^{\mu_{n+1}}$ comes from a Nori motive over $\bb{Q}$.
\end{rek}
When the representation $V_\lambda$ is the $k$-th symmetric power of the standard representation of $\mathrm{SL}_{n+1}$, i.e., $V_{(k,0,\ldots,0)}$, we recover the motive $\rr{M}_{n+1}^k$ constructed in \cite[(3.1)]{fresan2018hodge}. For simplicity, we use $\mathrm{M}_{n+1}^k$ instead of $\mathrm{M}_{n+1}^{(k,0,\ldots,0)}$ in this situation.

\begin{prop}\label{prop::perfect-pairing}
The motives $\mathrm{M}^{\lambda}_{n+1}$ are pure of weight $n|\lambda|+1$. Moreover, they are equipped with $(-1)^{n|\lambda|+1}$-symmetric perfect pairings
	\begin{equation*}
		\begin{split}
			\mathrm{M}^{\lambda}_{n+1}\times \mathrm{M}^{\lambda}_{n+1}\to \bb{Q}(-n|\lambda|-1).
		\end{split}
	\end{equation*}
\end{prop}
\begin{proof}
	The motives $\mathrm{gr}^W_{n|\lambda|+1}(\mathrm{H}^{n|\lambda|-1}_{c}(\mathcal{K})(-1))$ are pure of weight $n|\lambda|+1$ by construction. Additionally, they are equipped with $(-1)^{n|\lambda|+1}$-symmetric perfect pairings, using a similar proof \cite[Thm.\,3.2]{fresan2018hodge} for exponential mixed Hodge structures. Taking into account the isotypic components, the motives $\mathrm{M}^{\lambda}_{n+1}$ are also pure of weight $n|\lambda|+1$, and possess the induced $(-1)^{n|\lambda|+1}$-symmetric pairings.
\end{proof}

\subsection{Realizations in characteristic \texorpdfstring{$0$}{0}}\label{sec::realization}

\subsubsection{The de Rham realizations}\label{subsec::derham-realization}

The de Rham realizations of $\rr{M}_{n+1}^\lambda$ underlies a pure Hodge structure of weight $nk+1$. When $n=1$ and $\lambda=(k)$, the Hodge numbers of $\rr{M}_{2}^k$ are computed in \cite[Thm.\,1.8]{fresan2018hodge}, which are either $0$ or $1$. In \cite[Thm.\,1.2 \& Thm.\,5.23]{Qin23Hodge}, we computed the Hodge numbers for more motives and expressed them using generating series. By a direct computation on generating series in \textit{loc. cit.}, we deduce the following corollary.

\begin{cor}\label{cor::hodge}
	For pairs $(n+1,k)$ listed in the table in \cref{thm::L-function}, the Hodge numbers of $\rr{M}_{n+1,\rr{dR}}^k$ are either $0$ or $1$.	
\end{cor}

For $\rr{M}_{4,\rr{dR}}^4$ and $\rr{M}_{3,\rr{dR}}^{(2,2)}$, although we cannot compute their Hodge numbers directly, they still have Hodge numbers either $0$ or $1$, see \cref{rek::hodge-m44} and \cref{rek::hodge-m322}.

\subsubsection{The \texorpdfstring{$\ell$}{l}-adic realizations}\label{sec::realization-ell-adic}

For a prime $\ell$, the $\ell$-adic realization 
	\begin{equation}\label{eq::ell-adic-realization-char-0}
		(\mathrm{M}_{n+1}^{\lambda})_\ell:=\gr^W_{n|\lambda|+1}\Hetc{n|\lambda|-1}(\mathcal{K}_{\bq},\bb{Q}_\ell)^{G_\lambda\times \mu_{n+1},\chi_n}(-1)
	\end{equation}
of $\rr{M}_{n+1}^\lambda$ is a continuous $\ell$-adic representation of the absolute Galois group $\gal(\bq/\bb{Q})$, which is pure of weight $n|\lambda|+1$ and is equipped with a $(-1)^{n|\lambda|+1}$-symmetric pairing by Proposition~\ref{prop::perfect-pairing}. Similar to the situation for motives, we indeed obtain a representation of $\gal(\bq/\bb{Q})$. As explained in \cref{rek::Galois-descent-motive}, although the action of $\mu_{n+1}$ does not commute with $\gal(\bq/\bb{Q})$, the invariants of $\mu_{n+1}$ are stable under the action of $\gal(\bq/\bb{Q})$.

For the case of symmetric power moments of Kloosterman sums, we computed the dimensions of $(\mathrm{M}_{n+1}^k)_{\dR}$ in \cite[Cor.\,2.19]{Qin23Hodge}. By the comparison theorem, we have the following proposition.

\begin{prop}\label{cor::dimension-mid-de-ell-adic}
	The dimension of $(\mathrm{M}_{n+1}^k)_\ell$ is 
	\begin{equation*}
		\begin{split}
			\frac{1}{n+1}\q{\binom{k+n}{n}-d(k,n+1)}-\sum_{u=0}^{\lfloor\frac{n k}{2}\rfloor}m_k(u) 
			-\begin{cases}
				a(k,n+1) & 2\mid n,\\
				0 & 2\nmid n k,\\
				 b(k,n+1) 
				& \text{else}.
			\end{cases}
		\end{split}
	\end{equation*}
where the numbers $a(k,n+1), b(k,n+1)$ and $d(k,n+1)$ are defined in Section~\ref{nota::multi-indices}, the numbers $m_k(u)$ are defined in \eqref{eq::counting_inv_at_0}.
\end{prop}

We will study the ramification properties of these Galois representations in Section~\ref{sec::Galois-repr}.

\subsection{Other realizations in characteristic \texorpdfstring{$p>0$}{p>0}}\label{sec::analogy-pos-char}

\subsubsection{The \texorpdfstring{$\ell$}{l}-adic case}\label{subsec::ell-adic-analogy}

\begin{prop}\label{prop::ell-adic-galois-descent}
	We have
		$$\mathrm{H}^i_{\et,?}(\bb{G}_{m,\bfp},\Kl_{n+1}^\lambda)
		\simeq 
		\mathrm{H}^{n|\lambda|+i}_{\et,?}\Bigl(\bb{G}_{m,\bfp}^{n|\lambda|+1},\scr{L}_{\psi(\tilde f_{|\lambda|})}\Bigr)^{ G_\lambda\times \mu_{n+1},\chi_n}\footnote{Here the action of $\mu_{n+1}$ is induced by that on $\bb{G}_{m,\fpp(\zeta_{n+1})}$, and we can understand the $\mu_{n+1}$-invariants similarly as in \cref{rek::Galois-descent-motive}.},$$
	for $i\in \{0,1,2\}$. 
	\end{prop}
	\begin{proof}
		We provide the proof for the usual cohomology here, and the properties of the cohomology with compact support and the middle cohomology can be proved similarly.
	
		Let $\pr_z$ be the projection from $\bb{G}_{m}^{n|\lambda|}\times \bb{G}_{m,z}$ to the last factor $\bb{G}_{m,z}$. The projection $\pr_t$ is defined in a parallel way to $\pr_z$. 
		By the isomorphism 
		$ \Bigl([n+1]_*\scr{L}_{\psi_p(\tilde f_{|\lambda|})} \Bigr)^{\mu_{n+1}}\simeq \scr{L}_{\psi_p( f_{|\lambda|})},$ we have
			$\Kl_{n+1}^\lambda\simeq ([n+1]_*[n+1]^*\Kl_{n+1}^\lambda)^{\mu_{n+1}}.$
	Then
	\begin{equation*}
		\begin{split}
			\mathrm{H}^i_{\et}(\bb{G}_{m,\bfp},\Kl_{n+1}^\lambda)
			\simeq & \mathrm{H}^i_{\et}(\bb{G}_{m,\bfp},([n+1]_*[n+1]^*\Kl_{n+1}^\lambda)^{\mu_{n+1}}) \\[5pt]
			\simeq & \mathrm{H}^i_{\et}(\bb{G}_{m,\bfp}, [n+1]^*\Kl_{n+1}^\lambda )^{\mu_{n+1}} \\[5pt]
			\simeq & \mathrm{H}^i_{\et}\Bigl(\bb{G}_{m,\bfp}, [n+1]^*\bigl(\Kl_{n+1}^{\otimes |\lambda|}\bigr)^{G_\lambda,\chi_n} \Bigr)^{\mu_{n+1}} \\[5pt]
			\simeq & \Bigl(\mathrm{H}^{n|\lambda|+i}_{\et,?}\Bigl(\bb{G}_{m,\bfp}^{n|\lambda|+1},\scr{L}_{\psi(\tilde f_{|\lambda|})}\Bigr)^{G_\lambda,\chi_n}\Bigr)^{ \mu_{n+1}},
		\end{split}
	\end{equation*}
	where in the last isomorphism we used the geometric description of $\Kl_{n+1}^\lambda$ from Proposition~\ref{prop:geometric-char}.
	\end{proof}
Similar to the construction for relevant de Rham cohomologies in \cite[(2.12)]{fresan2018hodge}, we have the following corollary.
	\begin{cor}\label{cor::perfect-pariring-ell-adic}
		There is a $(-1)^{n|\lambda|+1}$-symmetric perfect self-pairing on $\mathrm{H}^1_{\mathrm{\et, \rr{mid}}}(\bb{G}_{m,\bfp},\Kl_{n+1}^\lambda)$.
	\end{cor}

\begin{thm}\label{thm::etale-realization-mod-p}
Assume that $n|\lambda|\geq 3$. We have isomorphisms of $\ell$-adic cohomologies
	\begin{equation*}
		\begin{split}
			\gr^W_{n|\lambda|+i}\mathrm{H}^i_{\mathrm{\et, c}}(\bb{G}_{m,\bfp},\Kl_{n+1}^\lambda)
			\simeq &\gr^W_{n|\lambda|+i}\mathrm{H}^{n|\lambda|-2+i}_{\mathrm{\et,c}}(\mathcal{K}_{\bfp},\bb{Q}_{\ell}(\zeta_p))^{G_\lambda \times \mu_{n+1}, \chi_n}(-1)
		\end{split}
	\end{equation*}
for $i\in \{0,1,2\}$, and 
\begin{equation*}
	\begin{split}
		\mathrm{H}^1_{\mathrm{\et, mid}}(\bb{G}_{m,\bfp},\Kl_{n+1}^\lambda)
		\simeq &\gr^W_{n|\lambda|+1} \mathrm{H}^{n|\lambda|-1}_{\mathrm{\et,c}}\Bigl(\mathcal{K}_{\bfp},\bb{Q}_{\ell}(\zeta_p)\Bigr) ^{G_\lambda \times \mu_{n+1}, \chi_n}(-1)\\
		\simeq &\gr^W_{n|\lambda|+1} \mathrm{H}^{n|\lambda|+1}_{\mathrm{\et,\cc{K}_{\bfp}}}\Bigl(\bb{G}_{m,\bfp}^{n|\lambda|},\bb{Q}_{\ell}(\zeta_p)\Bigr) ^{G_\lambda \times \mu_{n+1}, \chi_n},
	\end{split}
\end{equation*}
which is also isomorphic to $\gr^W_{n|\lambda|+1} \mathrm{H}^{n|\lambda|-1}_{\mathrm{\et}}\Bigl(\mathcal{K}_{\bfp},\bb{Q}_{\ell}(\zeta_p)\Bigr) ^{G_\lambda \times \mu_{n+1}, \chi_n}(-1)$ when $\mathcal{K}$ is smooth.
\end{thm}

\begin{proof}
By performing a change of variables $(t,x_{i,j})\mapsto (t,x_{i,j}/t)$, for $i\in \{0,1\}$, we obtain 
\begin{equation*}
	\begin{split}
		\mathrm{H}^{n|\lambda|+i}_{\et,\rr{c}}\Bigl(\bb{G}^{n|\lambda|+1}_{m,\bfp},\scr{L}_{\psi(\tilde f_{|\lambda|})}\Bigr)\simeq 
		\Hetc{n|\lambda|+i}\Bigl(\bb{G}_{m,\bfp}^{n|\lambda|+1},\scr{L}_{\psi(t\cdot g^{\boxplus |\lambda|})}\Bigr).
	\end{split}
\end{equation*}
Then, considering the localization sequence for the triple
	\begin{equation*}
		\left((\bb{A}^1\times \bb{G}_m^{n|\lambda|},t\cdot g^{\boxplus |\lambda|}),(\bb{G}_m^{n|\lambda|+1},t\cdot g^{\boxplus |\lambda|}),(0\times \bb{G}_m^{n|\lambda|},0)\right),
	\end{equation*}
we have exact sequences 
\begin{equation}\label{eq::ell-1}
	\begin{split}
		\Hetc{n|\lambda|-1+i}\Bigl(\bb{G}_{m,\bfp}^{n|\lambda|},\bb{Q}(\zeta_p)\Bigr)\to 
		&\Hetc{n|\lambda|+i}\Bigl( \bb{G}_{m,\bfp}^{n|\lambda|+1},\scr{L}_{\psi_p(t\cdot g^{\boxplus |\lambda|})}\Bigr) \\ 
		\to&\Hetc{n|\lambda|+i}\Bigl(\bb{A}^1_{\bfp} \times \bb{G}_{m,\bfp}^{n|\lambda|},\scr{L}_{\psi_p(t\cdot g^{\boxplus |\lambda|})}\Bigr)
		 \to 
		\Hetc{n|\lambda|+i}\Bigl(\bb{G}_{m,\bfp}^{n|\lambda|},\bb{Q}_\ell(\zeta_p)\Bigr).
	\end{split}
\end{equation}
for $i \in \{0,1,2\}$. Next, we consider another triple
	\begin{equation}\label{eq::triples-2}	\left((\bb{A}^1\times \bb{G}_m^{n|\lambda|},t\cdot g^{\boxplus k}),(\bb{A}^1\times (\bb{G}_m^{n|\lambda|}\backslash \mathcal{K}),t\cdot g^{\boxplus |\lambda|}),(\bb{A}^1\times \mathcal{K},0)\right).
	\end{equation}
Observe that for any $r\geq 0$, we have 
	\begin{equation*}
		\begin{split}
			\Hetc{r}\Bigl(\bb{A}^1_{\bfp}\times (\bb{G}_m^{n|\lambda|}\backslash \mathcal{K})_{\bfp},\scr{L}_{\psi_p(t\cdot g^{\boxplus |\lambda|})}\Bigr)=&
			\Hetc{r}\Bigl(\bb{A}^1_{\bfp}\times (\bb{G}_m^{n|\lambda|}\backslash \mathcal{K})_{\bfp},\scr{L}_{\psi_p(t)}\Bigr)\\
			=&\oplus_{a+b=r}\Hetc{a}\Bigl(\bb{A}^1_{\bfp},\scr{L}_{\psi_p}\Bigr)\otimes \Hetc{b}\Bigl((\bb{G}_m^{n|\lambda|}\backslash \mathcal{K})_{\bfp},\bb{Q}_\ell(\zeta_p)\Bigr)=0,
		\end{split}
	\end{equation*}
where we performed a change of variables in the first identity by $(t,x_{i,j})\mapsto (t\cdot (g^{\boxplus |\lambda|})\inv,x_{i,j})$. So, by the long exact sequences associated with the triple \eqref{eq::triples-2}, we deduce
\begin{equation}\label{eq::ell-2}
	\begin{split}
		\Hetc{n|\lambda|+i}\Bigl(\bb{A}^1_{\bfp}\times \bb{G}_{m,\bfp}^{n|\lambda|},\scr{L}_{\psi_p(t\cdot g^{\boxplus |\lambda|})}\Bigr)\simeq \Hetc{n|\lambda|-2+i}\Bigl( \mathcal{K}_{\bfp},\bb{Q}_\ell(\zeta_p)\Bigr)(-1).
	\end{split}
\end{equation}

Now, we combine \eqref{eq::ell-1} and \eqref{eq::ell-2} to get exact sequences for $i\in\{0,1,2\}$. Then taking the isotypic component of these sequences, we conclude 
\begin{equation}\label{eq::ell-3}
	\begin{split}
		\Hetc{n|\lambda|-1+i}\Bigl(\bb{G}_{m,\bfp}^{n|\lambda|},&\bb{Q}_\ell(\zeta_p)\Bigr)^{G_\lambda\times \mu_{n+1},\chi_n}
		\to 
		\Hetc{i} ( \bb{G}_{m,\bfp},\Kl_{n+1}^\lambda )\\  
		\to		& \Hetc{n|\lambda|-2+i} ( \mathcal{K}_{\bfp},\bb{Q}_\ell(\zeta_p) )  ^{G_\lambda\times \mu_{n+1},\chi_n}(-1)
		 \to\Hetc{n|\lambda|+i}\Bigl(\bb{G}_{m,\bfp}^{n|\lambda|},\bb{Q}_\ell(\zeta_p)\Bigr)^{G_\lambda\times \mu_{n+1},\chi_n}
	\end{split}
\end{equation}
by Proposition~\ref{prop::ell-adic-galois-descent}. By taking the graded quotient $\gr^W_{n|\lambda|+i}$ on the sequence \eqref{eq::ell-3}, we obtain by analyzing the Frobenius weights that
\begin{equation*}
	\begin{split}
		\gr^W_{n|\lambda|+i}\rr{H}^{i}_{\et,\rr{c}}( \bb{G}_{m,\bfp},\Kl_{n+1}^\lambda)
		\simeq 
		\mathrm{gr}^W_{n|\lambda|+i} \Hetc{n|\lambda|-2+i}( \mathcal{K}_{\bfp},\bb{Q}_\ell(\zeta_p)) ^{G_\lambda\times \mu_{n+1},\chi_n}(-1).
	\end{split}
\end{equation*}
In particular, by putting $i=1$, we deduce 
\begin{equation*}
	\begin{split}
		\rr{H}^{1}_{\et,\rr{mid}}( \bb{G}_{m,\bfp},\Kl_{n+1}^\lambda)=\gr^W_{n|\lambda|+1}\rr{H}^{1}_{\et,\rr{c}}( \bb{G}_{m,\bfp},\Kl_{n+1}^\lambda)
		\simeq 
		\mathrm{gr}^W_{n|\lambda|-1} \Hetc{n|\lambda|-1}( \mathcal{K}_{\bfp},\bb{Q}_\ell(\zeta_p)) ^{G_\lambda\times \mu_{n+1},\chi_n}(-1).
	\end{split}
\end{equation*}

For the usual cohomology, we use similar localization sequences to get 
	$$\gr^W_{n|\lambda|+1}\mathrm{H}^1_{\mathrm{\et}}(\bb{G}_{m,\bfp},\Kl_{n+1}^\lambda)
	\simeq 
	\gr^W_{n|\lambda|+1} \mathrm{H}^{n|\lambda|+1}_{\mathrm{\et,\cc{K}_{\bfp}}}\Bigl(\bb{G}_{m,\bfp}^{n|\lambda|},\bb{Q}_{\ell}(\zeta_p)\Bigr) ^{G_\lambda\times \mu_{n+1},\chi_n},$$
which is also isomorphic to $\gr^W_{n|\lambda|+1} \mathrm{H}^{n|\lambda|-1}_{\mathrm{\et}}(\mathcal{K}_{\bfp},\bb{Q}_{\ell}(\zeta_p)) ^{G_\lambda\times \mu_{n+1},\chi_n}(-1)$ when $\mathcal{K}$ is smooth.
\end{proof}

From Theorem~\ref{thm::etale-realization-mod-p}, the name of $\rr{M}_{n+1}^\lambda$ is justified, because the $L$-functions of $\rr{M}_{n+1}^\lambda$ coincide with the $L$-functions attached to Kloosterman sheaves $\Kl_{n+1}^\lambda$.

\subsubsection{The \texorpdfstring{$p$}{p}-adic case.}\label{sec::p-adic-crystal}
\paragraph{Bessel \texorpdfstring{$F$}{F}-isocrystal}
Let $\bar{\bb{Q}}_p$ be the algebraic closure of $\bb{Q}_p$, and we choose an element $\varpi$ such that $\varpi^{p-1}=-p$. This gives rise to a unique nontrivial additive character $\psi\colon\fpp\to \bar{\bb{Q}}_p\cros$, satisfying $\psi(1)\equiv 1+\varpi \mod \varpi^2$. The Dwork's $F$-isocrystal $\scr{L}_{\varpi}$ is a rank $1$ connection $\mathrm{d}+\varpi\mathrm{d}z$ with Frobenius structure $\exp(\varpi(z^p-z))$ on the overconvergent structure sheaf of $\bb{A}^1$ over $K=\bb{Q}_p(\varpi)$. We denote $\scr{L}_{\varpi h}$ as the inverse image of $\scr{L}_{\varpi}$ along a regular function $h\colon X\to \bb{A}^1$. 

The \textit{Kloosterman crystal} is an overconvergent $F$-isocrystal also defined using the diagram \eqref{eq:diagram-Kl} by
\begin{equation*}
 \begin{split}
 	\Kl_{n+1}:=\mathrm{R}\pi_{\mathrm{rig}*}\scr{L}_{\varpi\sigma}[n].
 \end{split}
 \end{equation*} 
Similar to the Kloosterman sheaves for reductive groups, there are Bessel $F$-crystals for reductive groups from \cite{xu2022bessel}. The connection associated with $G=\mathrm{SL}_{n+1}$ and $V=V_{\lambda}$ is $\Bigl(\Kl_{n+1}^{\otimes |\lambda|}\Bigr)^{G_\lambda,1\times \rr{sign}}(\frac{n|\lambda|}{2})$. By abuse of notation, we denote by $\Kl_{n+1}^\lambda$ the $F$-isocrystal $\Bigl(\Kl_{n+1}^{\otimes k}\Bigr)^{G_\lambda,1\times \rr{sign}} $.

\paragraph{Rigid cohomologies}
Similar to the $\ell$-adic case, we have for $?\in \{\emptyset,\rm{c,mid}\}$ 
	$$\mathrm{H}^1_{\mathrm{rig},?}(\bb{G}_m/K,\Kl_{n+1}^\lambda)=\mathrm{H}^{n|\lambda|+1}_{\mathrm{rig},?}\Bigl(\bb{G}_m^{n|\lambda|+1},\scr{L}_{\varpi\tilde f_{|\lambda|}}\Bigr)^{G_\lambda\times \mu_{n+1},\chi_n}[\varpi].$$
Using the argument in \cite[\S3.2.2]{fresan2018hodge} by changing the isotypic component from $(S_k\times \mu_{n+1},\chi_n)$ to $(G_\lambda\times \mu_{n+1},\chi_n)$, we obtain
	\begin{equation}\label{eq::rigid-realization}
	\begin{split}
		\mathrm{H}^{1}_{\mathrm{rig,mid}}(\bb{G}_m/K,\Kl_{n+1}^\lambda)\simeq \gr^W_{n|\lambda|+1} \mathrm{H}^{n|\lambda|-1}_{\mathrm{rig,c}}(\mathcal{K}/K) ^{G_\lambda\times \mu_{n+1},\chi_n}(-1)[\varpi],
	\end{split}
	\end{equation}
which is also isomorphic to $\gr^W_{n|\lambda|+1} \mathrm{H}^{n|\lambda|-1}_{\mathrm{rig}}(\mathcal{K}/K) ^{G_\lambda\times \mu_{n+1},\chi_n}(-1)[\varpi]$ when $\mathcal{K}$ is smooth.

\section{\texorpdfstring{$L$}{L}-functions of Kloosterman sheaves}\label{sec::functional-eq}

In this section, the main goal is to prove Theorem~\ref{thm::L-function}. First, we study the Galois representations $(\mathrm{M}_{n+1}^\lambda)_\ell$ to provide the necessary properties needed in proving Theorems~\ref{thm::L-function} and~\ref{thm::modular form}. The general case is covered in Theorem~\ref{thm::unramified}, while a more detailed analysis of the case of $\Sym^k\Kl_{n+1}$ is provided in Theorem~\ref{thm::ell-adic-galois-even}. We also review essential properties of Deligne--Weil representations in Section~\ref{sec:W-D-repr}. Lastly, Theorem~\ref{thm::L-function} is proven in Section~\ref{subsec:potential-automorphy}.

\subsection{Galois representations attached to Kloosterman sheaves}\label{sec::Galois-repr}

\subsubsection{A compactification}\label{subsec::compactification}
Let $k$ be an integer and $p$ a prime number, not dividing $n+1$. The Laurent polynomial $g_{n+1}^{\boxplus k}=\sum_{i=1}^k\left(\sum_{j=1}^n y_{i,j}+\frac{1}{\prod_j y_{i,j}}\right)$ on the torus $\bb{G}_{m,\bb{Q}}^{nk}$ defines a hypersurface $\mathcal{K}$. We select a toric compactification $X_{\mathrm{tor}}$ of $\bb{G}_m^{nk}$ following the approach in \cite[\S4.3.2]{fresan2018hodge}, see also \cite[\S5.2.3]{Qin23Hodge}

We start with the pair $\bigl(\bb{G}_{m,\bb{Q}}^{nk},g_{n+1}^{\boxplus k}\bigr)$. Let $M=\oplus_{i,j}\bb{Z} y_{i,j}$ be the lattice of monomials on $\bb{G}_{m,\bb{Q}}^{nk}$ and $N=\oplus_{i,j} \bb{Z} e_{i,j}$ the dual lattice. We consider the toric compactification $X$ of $\bb{G}_{m,\bb{Q}}^{nk}$ attached to the simplicial fan $F$ in $N_{\bb{R}}$ generated by the rays
	$$\textstyle\bb{R}_{\geq 0}\cdot \sum_{i,j}\epsilon_{i,j}e_{i,j}$$
where $\epsilon_{i,j}\in \{0,\pm1\}$ and $(\epsilon_{i,j})_{i,j}\neq 0$. Each simplicial cone of maximal dimension $nk$ in $F$ provides an affine chart of $X$, which is isomorphic to $\bb{A}^{nk}$. On each chart, the function $g_{n+1}^{\boxplus k}$ has the same structure. For example, we can consider the maximal cone generated by
	$$\gamma_{i_0,j_0}:=\sum_{1\leq i\leq i_0-1,1\leq j\leq n }e_{i,j}+ \sum_{1\leq j\leq j_0}e_{i_0,j}$$
for $1\leq i_0\leq k$ and $1\leq j_0\leq n$, where the affine ring associated with the dual cone is the polynomial ring $\bb{Q}[u_{i,j}]$ such that $$u_{i,j}=
	\begin{cases}
		y_{i,j}/y_{i,j+1} & 1\leq j<n,\\
		y_{i,j}/y_{i+1,1} & i<k, j=n,\\
		y_{k,n} & i=k,j=n.
	\end{cases}$$

In this chart, we can rewrite $g_{n+1}^{\boxplus k}$ as $g_1/\bigl(\prod_{1\leq j\leq n}u_{1,j}^{j}\cdot \prod_{2\leq i\leq k,j}u_{i,j}^n\bigr)$, where 
	$$g_1=1+\sum_{e=1}^{k-1}\prod_{j=1}^nu_{1,j}^{j}
	\cdot \prod_{\substack{2\leq i\leq e\\1\leq j\leq n}}u_{i,j}^n 
	\cdot \prod_{j=1}^n u_{e+1,j}^{n-j}
	+\prod_{1\leq j\leq n}u_{1,j}^{j}\cdot \prod_{\substack{2\leq i\leq k\\1\leq j\leq n}}u_{i,j}^n\cdot h$$
for a polynomial $h$. The toric variety $X$ provides a compactification of $(\bb{G}^{nk}_{m},g_{n+1}^{\boxplus k})$, where the closure of the zero locus of $g_{n+1}^{\boxplus k}$, and $X\backslash \bb{G}_m^{nk}$ form a strict normal crossing divisor. 

We take the Zariski closure of the hypersurface $Z(g_{n+1}^{\boxplus k})$ inside $X$, denoted by $\bar{\mathcal{K}}$. One can check that
	$$ Z(g_1)\cap Z(u_{1,s })=\emptyset,\ Z(g_1)\cap Z(u_{r,s })=Z\biggl(1+\sum_{e=1}^{r-1}\prod_{j=1}^nu_{1,j}^{j}
	\cdot \prod_{\substack{2\leq i\leq e\\ 1\leq j\leq n}}u_{i,j}^n 
	\cdot \prod_{j=1}^nu_{e+1,j}^{n-j}\biggr)$$
and
	$$Z(\partial g_1/\partial u_{1,1})\cap Z(u_{r,s })=Z\biggl(\sum_{e=1}^{r-1}\prod_{j=1}^nu_{1,j}^{j}
	\cdot \prod_{\substack{2\leq i\leq e\\1\leq j\leq n}}u_{i,j}^n 
	\cdot \prod_{j=1}^nu_{e+1,j}^{n-j}/u_{1,1}\biggr)$$
for $1\leq s\leq n$ and $2\leq r\leq k$. It follows that for $1\leq s\leq n$ and $1\leq r\leq k$, we have $Z(g_1)\cap Z(\partial g_1/\partial u_{1,1 }) \cap Z(u_{r,s })=\emptyset$. We deduce that $\bar{\mathcal{K}}$ is smooth along the divisor $Z\bigl(\prod_{1\leq i\leq k,1\leq j\leq n} u_{i,j}\bigr)$. Moreover, one can check that $Z\bigl(\prod_{1\leq i\leq k,1\leq j\leq n} u_{i,j}\bigr)\cap \bar{\mathcal{K}}$ satisfies the strict normal crossing property.

As for $\mathcal{K}=\bck\cap \bb{G}_m^{nk}$, one can check that $\mathcal{K}$ is smooth if $\gcd(k,n+1)=1$ and has isolated singularities inside $\bb{G}_m^{nk}$ if $\gcd(n+1,k)>1$. In the latter case, the singular locus $\Sigma_0$ of $\mathcal{K}$ has only finitely many $\bq$-points or $\bfp$-points, all of  which are ordinary quadratic. We perform blow-ups of $\bb{G}_m^{nk}$ along the singular locus $\Sigma_0(\bq)$ and denote by $\mathcal{K}'$ the strict transform of $\mathcal{K}$. For convenience, we denote $\mathcal{K}'$ as $\mathcal{K}$ in the case $\gcd(k,n+1)=1$. We denote by $\bar{\mathcal{K}}'$ the closure of $\mathcal{K}'$ in $\rr{Bl}_{\Sigma_0}(X)$.

\begin{lem}\label{lem::blowup}
	Let $\bb{F}$ be either $\bq$ or $\bfp$. Suppose that $\gcd(k,n+1)>1$, $nk$ is even, and $nk\geq 4$. If $\bb{F}=\bfp$, we additionally assume $p\nmid n+1$. Then we have 
		$$\Hetc{nk-1}({\cc{K}}'_{\bb{F}})=
				\Hetc{nk-1}({\cc{K}}_{\bb{F}}).$$

\end{lem}
\begin{proof}
Let $T$ be the preimage of $\Sigma_0$ along the blow-up morphism $\bar{\cc{K}}'\to \bar{\cc{K}}$, which is a disjoint union of quadrics. Then consider the commutative diagram of exact sequences
\begin{equation}
	\begin{tikzcd}
		\Het{nk-2}(T_{\bb{F}})\ar[r] &\Hetc{nk-1}(({\cc{K}}'\backslash T)_{\bb{F}}) \ar[r,"\alpha"]\ar[d,"\simeq"] & \Hetc{nk-1}({\cc{K}}'_{\bb{F}})\ar[d,"\beta"]\ar[r,"\gamma"] &\Het{nk-1}(T_{\bb{F}})\\
		\Het{nk-2}((\Sigma_0)_{\bb{F}})\ar[r]& \Hetc{nk-1}(({\cc{K}}\backslash \Sigma_0)_{\bb{F}}) \ar[r] &\Hetc{nk-1}({\cc{K}}_{\bb{F}})\ar[r] & \Het{nk-1}((\Sigma_0)_{\bb{F}}).
	\end{tikzcd}
\end{equation}
Under the assumption that $nk\geq 4$, the cohomology $\Het{nk-2}((\Sigma_0)_{\bb{F}})$ and $\Het{nk-1}((\Sigma_0)_{\bb{F}})$ both vanish. In particular, we find that $\gamma$ is surjective if we extend the diagram by one more column to the right.

As $nk$ is even, the cohomology $\Het{nk-1}(T_{\bb{F}})=0$, because $T$ is disjoint union of quadrics. From this we conclude that $\Hetc{nk-1}({\cc{K}}'_{\bb{F}})=\Hetc{nk-1}({\cc{K}}_{\bb{F}}).$
\end{proof}

\subsubsection{The \texorpdfstring{$\ell$}{l}-adic case in general}\label{sec::ell-adic-in-general} 
 
Let $p\neq \ell$ be two different primes, $\lambda \in \bb{N}^n$ be a sequence, and $\zeta_{n+1}$ be either an $(n+1)$-th primitive root of unity in $\bfp$ or $\bq$. We adopt the notation from the previous section and replace $k$ with $|\lambda|$. We denote by $\Sigma'(p)=\Sigma'(|\lambda|,n+1,p)$ the singular set of $\bar{\mathcal{K}}'_{\bfp}$. Recall that each singular point $x$ of $\bar{\mathcal{K}}'_{\bfp}$ is of the form $x=(x_{i, j})_{1\leq i\leq k,\,1\leq j\leq n}=(\zeta_{n+1}^{a_i})_{i,j}$ for some $a_i\in \{0,1,\ldots,n\}$. The action of $S_{|\lambda|}\times \mu_{n+1}$ on $\Sigma'(p)$ is given by 
	$$(\sigma,\zeta_{n+1}^a)\cdot (x_{i,j})=(\zeta_{n+1}^a\cdot x_{\sigma(i),j}).$$
One can identify the $S_{|\lambda|}$-orbits in $\Sigma'(p)$ with the set of multi-indices 
	\begin{equation}\label{eq:number-singular-points}
	    \{\underline I\in \bb{N}^{n+1}\mid |\underline I|=|\lambda|,\ C_{\underline I}=0 \text{ in } \bfp,\,C_{\underline I}\neq 0 \text{ in } \bb{C}\}=d(k,n+1,p)-d(k,n+1)
	\end{equation}
by sending $x=(\zeta_{n+1}^{a_i})_{i,j}$ to $\underline I$ such that $I_j=\#\{i\mid  a_i=j\}$. On multi-indices, the actions of $\mu_{n+1}$ is given by $\zeta_{n+1}\cdot (I_0,I_1,\ldots,I_n)=(I_n,I_0,\ldots,I_{n-1})$.

Assume that $p\nmid n+1$. The singular points in $\Sigma'(p)$ are ordinary quadratic in the sense of \cite[XII 1.1]{SGA7-II}. Let $n|\lambda|=2m+1$ (resp. $n|\lambda|=2m+2$) and we apply the Picard--Lefschetz formula \cite[XV 3.4]{SGA7-II} to $\bar{\mathcal{K}}'_{\bb{Z}_p}\to \operatorname{Spec}(\bb{Z}_p)$. For each $x\in \Sigma'(p)$, there is a vanishing cycle class $\delta_x\in \mathrm{H}^{n|\lambda|-1}_{\acute{e}t}\bigl(\bar{\mathcal{K}}'_{\bq_p}\bigr)(m)$, well-defined up to a sign. These vanishing cycle classes are orthogonal to each other and satisfy 

	$$(\delta_x, \delta_x)=(-1)^m2\quad (\text{resp.}\quad  (\delta_x, \delta_x)=0).$$ 
We fix a place of $\bq$ over $p$ and denote by $I_p$ the corresponding inertia group. To each element $\sigma\in I_p$, the action on $\mathrm{H}^{n|\lambda|-1}_{\et}(\bar{\mathcal{K}}'_{\bar{\bb{Q}}})$ is given by
\begin{equation}\label{eq::P-L-formula}
	\sigma(v)=
	\begin{cases}
		v+(-1)^m\sum_{x\in \Sigma'(p)}\frac{\epsilon(\sigma)-1}{2}(v,\delta_x)\delta_x & 2\nmid n|\lambda|,\\
		v-(-1)^m\sum_{x\in \Sigma'(p)}\epsilon(\sigma)(v,\delta_x)\delta_x & 2\mid n|\lambda|,
	\end{cases}
\end{equation}
where $\epsilon$ is the character $I_p\twoheadrightarrow \{\pm 1\}$ of order $2$ if $n|\lambda|$ odd, and is the fundamental tame character $I_p\to \varprojlim_n \mu_{\ell^n}(\bql)$ if $n|\lambda|$ is even. Moreover, we have an exact sequence 
	$$0\to \mathrm{H}^{n|\lambda|-1}_{\et}(\bar{\mathcal{K}}'_{\bar{\bb{F}}_p})\to
	\mathrm{H}^{n|\lambda|-1}_{\et}(\bar{\mathcal{K}}'_{\bar{\bb{Q}}})
	\xrightarrow{\gamma} \sum_{x\in \Sigma'(p)}\mathbb{Q}_{\ell}(m-n|\lambda|+1),$$
where $\gamma$ is the sum of the intersections with the vanishing cycle classes $\delta_x$.

\begin{thm}\label{thm::unramified}
	Suppose that $\gcd(n+1,|\lambda|)=1$ when $n|\lambda|$ is odd. 
	\begin{enumerate}
		\item If $p\nmid n+1$ and $\cc{K}'_{\bb{F}_p}$ is smooth, the Galois representation $(\mathrm{M}^\lambda_{n+1})_\ell$ is unramified at primes $p$, and there is an isomorphism 
		of $\gal(\bq_p/\bb{Q}_p)$-representations
		$$(\mathrm{M}^\lambda_{n+1})_\ell[\zeta_{p}]\simeq \rr{H}^1_{\et,\rr{mid}}(\bb{G}_{m,\bfp},\Kl_{n+1}^\lambda).$$  
		\item If $p\nmid n+1$ and $p\neq 2$, the Galois representation $(\mathrm{M}^\lambda_{n+1})_\ell$ is at most tamely ramified.
	\end{enumerate}
\end{thm}
\begin{proof}
	For simplicity, we omit the coefficient $\bb{Q}_\ell$ in the cohomology. Let $\bar{\mathcal{K}}^{(0)}=\bar{\mathcal{K}}'$ and $\bar{\mathcal{K}}^{(i)}$ the disjoint union of all $i$-fold intersections of distinct irreducible components of $\bar{\mathcal{K}}'\backslash \mathcal{K}'$ for $i\geq 1$. Let $\bb{F}$ be either $\bq$ or $\bfp$. Consider the spectral sequence 
	\begin{equation}\label{eq:spectral-seq}
		\begin{split}
			(E_1^{p,q})_{\bb{F}}
			=\Het{q}(\bar{\mathcal{K}}^{(p)}_{{\bb{F}}})\Rightarrow \Hetc{p+q}({\mathcal{K}}'_{{\bb{F}}}).
		\end{split}
	\end{equation}
For the case $\bb{F}=\bq$, since $\bar{\mathcal{K}}^{(i)}$ are proper smooth for all $i$, all morphisms in the $E_2$-page are $0$ for the reason of weights. Therefore, the spectral sequence degenerates at the $E_2$-page. It follows from the spectral sequence that
	\begin{equation*}
		\begin{split}
			\mathrm{gr}^W_{n |\lambda|-1}\Hetc{n |\lambda|-1}(\mathcal{K}'_{\bq})=(E_\infty^{0,n |\lambda|-1})_{\bq}
			&=\ker(\Het{n |\lambda|-1}(\bar{\mathcal{K}}'_{\bar{\bb{Q}}})\to \Het{n |\lambda|-1}(\bar{\mathcal{K}}^{(1)}_{\bar{\bb{Q}}}) )\\
			&=\im(\Hetc{n |\lambda|-1}(\mathcal{K}'_{\bq})\xrightarrow{\alpha}\Het{n |\lambda|-1}(\bar{\mathcal{K}}'_{\bq})),
		\end{split}
	\end{equation*}
	where the map $\alpha$ is the surjective edge map from the abutment $\Hetc{n |\lambda|-1}(\mathcal{K}'_{\bq})$ to $E_2^{0,n |\lambda|-1}$. Notice that the above spectral sequence is equivariant with respect to the action of $S_{|\lambda|}\times \mu_{n+1}$. Using the isomorphism in Lemma~\ref{lem::blowup}, we conclude that
	\begin{equation}\label{eq::image-alpha}
	\begin{split}
		\im(\alpha) ^{G_\lambda\times \mu_{n+1},\chi_n}\simeq 
			(\mathrm{M}^\lambda_{n+1})_\ell(1).
	\end{split}
	\end{equation}

For the case that $\bb{F}=\bfp$, we conclude similarly $G$-equivariant isomorphisms
	\begin{equation*}
	\begin{split}
		\mathrm{gr}^W_{n |\lambda|-1}\Hetc{n |\lambda|-1}(\mathcal{K}'_{\bq})
		=\gr^{W'}_{n |\lambda|-1}(E_\infty^{0,n |\lambda|-1})_{\bfp}
		=\gr^{W'}_{n |\lambda|-1}\im\Bigl(\Hetc{n |\lambda|-1}(\mathcal{K}'_{\bfp})\xrightarrow{\beta}\Het{n |\lambda|-1}(\bar{\mathcal{K}}'_{\bfp})\Bigr),
	\end{split}
	\end{equation*}
where we denote by $W'$ the (Frobenius) weight filtration to distinguish it from the weight filtration $W$ in characteristic $0$. Recall that
\begin{equation*}
\begin{split}
	\mathrm{H}^1_{\mathrm{\et, mid}}(\bb{G}_m,\Kl_{n+1}^{\lambda})
	\simeq \gr^{W'}_{n |\lambda|+1} \mathrm{H}^{n |\lambda|-1}_{\mathrm{\et,c}}(\mathcal{K}_{\bfp},\bb{Q}_{\ell}(\zeta)) ^{G_\lambda\times \mu_{n+1},\chi_n}(-1)
\end{split}
\end{equation*}
from Theorem~\ref{thm::etale-realization-mod-p}. We obtain
	\begin{equation}\label{eq::image-beta}
	\begin{split} 
		\gr^{W'}_{n |\lambda|+1} \im(\beta) ^{G_\lambda\times \mu_{n+1},\chi_n}(-1)[\zeta_p]\simeq \mathrm{H}^1_{\et,\mathrm{mid}}(\bb{G}_{m,\bar{\bb{F}}_p},\Kl_{n+1}^{\lambda})  
	\end{split}
	\end{equation}

Now we consider the $G$-equivariant commutative diagram with exact rows and columns
	\begin{equation}\label{eq::ell-adic-p-l-formula}
		\begin{tikzcd}
		& \mathrm{H}_{\et,\rr{c}}^{n |\lambda|-1}({\mathcal{K}}'_{\bar{\bb{F}}_p})\ar[r,"\iota_c"]\ar[d,"\beta"] 
		&\mathrm{H}_{\et,\rr{c}}^{n |\lambda|-1}({\mathcal{K}}'_{\bar{\bb{Q}}})\ar[d,"\alpha"] &
		\\
		0\ar[r] & 	\mathrm{H}_{\et}^{n |\lambda|-1}(\bar{\mathcal{K}}'_{\bar{\bb{F}}_p})		\ar[r,"\iota"] \ar[d]	&	\mathrm{H}_{\et}^{n |\lambda|-1}(\bar{\mathcal{K}}'_{\bar{\bb{Q}}})		\ar[r,"\gamma"]\ar[d]	&	\oplus_{x\in \Sigma'(p)}\bb{Q}_\ell(m-n|\lambda|+1),\\
		 & 	\mathrm{H}_{\et}^{n |\lambda|-1}(\bar{\mathcal{K}}^{(1)}_{\bar{\bb{F}}_p})\ar[r,"\sim"]&	\mathrm{H}_{\et}^{n |\lambda|-1}(\bar{\mathcal{K}}^{(1)}_{\bar{\bb{Q}}}) &
		\end{tikzcd}
	\end{equation}
where the middle row is given by the Picard--Lefschetz formula (we assume $p\nmid n+1$). Moreover, taking into account \eqref{eq::P-L-formula}, the representation $\mathrm{H}_{\et}^{n |\lambda|-1}(\bar{\mathcal{K}}'_{\bar{\bb{Q}}})$ of $\gal(\bar{\bb{Q}}_p/\Qp)$ is at most tamely ramified when $p\neq 2$. We verified the second statement in the Theorem.

Notice that each class $\delta_x$ is a generator of $\rr{H}_{\{x\}}^{n|\lambda|-1}(\bar{\cc{K}}',\rr{R}\Psi(m))$, with support $\{x\}$, where $\rr{R}\Psi$ denotes the nearby cycle complex. So $\Delta=\oplus \bb{Q}_\ell(-m)\delta_x$ is contained in $\im(\alpha)$. If we take the isotypic component with respect to $(G_\lambda \times \mu_{n+1},\chi_n)$ on the second row, we have an exact sequence 
	\begin{equation}\label{eq::equivariant-part-p-l-formula}
		0\to \mathrm{H}_{\et}^{n |\lambda|-1}(\bar{\mathcal{K}}'_{\bar{\bb{F}}_p}) ^{G_\lambda\times \mu_{n+1},\chi_n}	\xrightarrow{\iota}	
		 \mathrm{H}_{\et}^{n |\lambda|-1}(\bar{\mathcal{K}}'_{\bar{\bb{Q}}}) ^{G_\lambda\times \mu_{n+1},\chi_n}	\xrightarrow{\gamma}	
		\Bigl(\oplus_{x\in \Sigma}\bb{Q}_\ell(m-n|\lambda|+1) \Bigr)^{G_\lambda\times \mu_{n+1},\chi_n}.
	\end{equation}
By a diagram-chasing argument, we get from \eqref{eq::equivariant-part-p-l-formula} an inclusion 
	\begin{equation}\label{eq:inclusion-ell}
		\begin{split}
			\gr^{W'}_{n |\lambda|+1} \im(\beta) ^{G_\lambda\times \mu_{n+1},\chi_n}
			\hookrightarrow 
			 \im(\alpha) ^{G_\lambda\times \mu_{n+1},\chi_n}.
		\end{split}
	\end{equation}

When $\mathcal{K}'$ has good reduction at $p$, the variety $\bar{\cc{K}}_{\bfp}'$ is smooth proper and the morphisms $\iota$ and $\iota_{c}$ in \eqref{eq::ell-adic-p-l-formula} are isomorphisms. So $\im(\alpha)\simeq \im(\beta)$ are pure of weight $nk-1$ ($W$ and $W'$ coincide). By \eqref{eq::image-alpha} and \eqref{eq::image-beta}, we get an isomorphism 
	$$(\mathrm{M}^\lambda_{n+1})_\ell[\zeta_{p}]\simeq \rr{H}^1_{\et,\rr{mid}}(\bb{G}_{m,\bfp},\Kl_{n+1}^\lambda)$$
of unramified $\gal(\bq_p/\bb{Q}_p)$-representations from \eqref{eq:inclusion-ell}. This verifies the first statement in the Theorem.
\end{proof}

\begin{rek}\label{rek::unramified-p=3}
	In the discussion above, we have omitted the case where $p \mid n+1$. In this situation, the singular points of $\bar{\mathcal{K}}_{\bfp}$ are isolated but not ordinary quadratic, rendering the Picard--Lefschetz formula inapplicable in this case. Nevertheless, the vanishing cycles with respect to $\bar{\mathcal{K}}'_{\bb{Z}_p}\to \Spec(\bb{Z}_p)$ remain $0$ if $i\neq n |\lambda|-1$ \cite[Cor.\,2.10]{illusie2003perversite}. 
 
    If $\gcd(k,n+1)=1$, then $\mathcal{K}=\mathcal{K}'$, and many of the above arguments, including $\eqref{eq::image-beta}$, remain valid even when $p\mid n+1$. Based on the long exact sequence associated with vanishing cycles  \cite[XIII\,(1.4.2.2)]{SGA7-II}, the cospecialization morphism
		$$\Het{n |\lambda|-1}(\bar{\mathcal{K}}_{\bfp})\to \Het{n |\lambda|-1}(\bar{\mathcal{K}}_{\bar{\bb{Q}}})$$
	is injective. Hence, the diagram
		$$\begin{tikzcd}
			&\Hetc{n |\lambda|-1}({\mathcal{K}}_{\bfp})\ar[d,"\beta"]\ar[r]	& \Hetc{n |\lambda|-1}({\mathcal{K}}_{\bar{\bb{Q}}})\ar[d,"\alpha"]\\
			0\ar[r]&\Het{n |\lambda|-1}(\bar{\mathcal{K}}_{\bfp})\ar[r] & \Het{n |\lambda|-1}(\bar{\mathcal{K}}_{\bar{\bb{Q}}})
		\end{tikzcd}$$
	induces an injective morphism
		\begin{equation}\label{eq::injection-isolated-sing}
			\gr^{W'}_{n|\lambda|+1} \im(\beta)^{G_\lambda\times \mu_{n+1},\chi_n}(-1)\hookrightarrow \mathrm{gr}^W_{n |\lambda|+1} \Hetc{n |\lambda|-1}(\mathcal{K}_{\bq})^{G_\lambda\times \mu_{n+1},\chi_n} (-1) =(\mathrm{M}^\lambda_{n+1})_\ell.
		\end{equation}
	As long as the dimensions of the source and the target of \eqref{eq::injection-isolated-sing} are the same, the inclusion becomes an isomorphism, implying that $(\mathrm{M}^\lambda_{n+1})_\ell$ is unramified at $p$. For instance, when $n\leq 2$, $p=n+1$, and $p\nmid k$, the Galois representations attached to $\Sym^k\Kl_{n+1}$ are unramified according to \cite[Cor.\,4.3.5]{yun2015galois} and Corollary~\ref{cor::dimension_mid_p=3}. 
\end{rek}

Let $N(V_\lambda)$ and $E(V_\lambda)$ be endomorphisms of $V_\lambda$, induced from $N$ and $E$ in $\mathfrak{sl}_{n+1}$, as defined in \cite[\S\,5]{Frenkel2009}. Inspired by the above examples, we conjecture that:

\begin{conj}
	The morphism \eqref{eq::injection-isolated-sing} is an isomorphism when the matrices $ N(V_\lambda)+E(V_\lambda)$ is invertible. 
\end{conj}

\subsubsection{The \texorpdfstring{$\ell$}{l}-adic case for \texorpdfstring{$\Sym^k\Kl_{n+1}$}{SymkKl_{n+1}}}
Now we give a description in detail of $\rr{M}^{\lambda}_{n+1,\ell}$ for $\lambda=(k,0,\ldots,0)$, i.e., the case for $\Sym^k\Kl_{n+1}$. Let $a(k,n+1,p)$, $a(k,n+1)$, and $\delta(k,p)$ be numbers defined in Section~\ref{nota::multi-indices} and Proposition~\ref{prop::moments-p-large}.

\begin{thm}\label{thm::ell-adic-galois-even}
	Let $p$ be a prime different from $\ell$ such that $p\nmid n+1$ and $\cc{K}'$ has bad reductions at $p$. Then  
	\begin{enumerate}
		\item If $nk$ is odd, $\gcd(k,n+1)=1$, and $p\neq 2$, the Galois representation $(\mathrm{M}^k_{n+1})_\ell$ is tamely ramified at $p$. For such primes, we have orthogonal decompositions $(\mathrm{M}^k_{n+1})_\ell=\rr{H}\oplus \rr{E}$ as $\gal(\bar{\bb{Q}}_p/\Qp)$-representations such that
				\begin{itemize}
					\item $\rr{H}[\zeta_p]=\rr{H}^1_{\et,\rr{mid}}(\bb{G}_{m,\bfp},\Sym^k\Kl_{n+1})$,
					\item $\rr{E}$ is generated by vanishing cycle classes.
				\end{itemize}
			\item If $n+1$ is a prime number, the Galois representation $(\mathrm{M}^k_{n+1})_\ell$ is tamely ramified at $p$. For such primes, the inertia groups $I_p\subset \gal(\bar{\bb{Q}}_p/\bb{Q}_p)$ act unipotently on $(\mathrm{M}^k_{n+1})_\ell$ such that $(\sigma-1)^2=0$ for any $\sigma \in I_p$. The image $U$ of the nilpotent part of the monodromy operator, denoted as $N$, is generated by vanishing cycle classes and has dimension $a(k,n+1,p)-a(k,n+1)-\delta(k,p)$. With respect to the intersection pairing, $U$ is totally isotropic with orthogonal complement $(\mathrm{M}^k_{n+1})_\ell^{I_p}$. Moreover, the induced map $\sigma-1\colon (\mathrm{M}^k_{n+1})_\ell\mapsto (\mathrm{M}^k_{n+1})_\ell/U$ is zero.
	\end{enumerate}
\end{thm}
\begin{proof}
For simplicity, we replace the exponent $({S_k\times \mu_{n+1},\chi_n})$ by $(G,\chi)$. Since the Galois representations are trivial or one-dimensional when $nk\leq 3$, we assume that $nk\geq 4$. When $p\nmid n+1$, all singularities of $\bar{\mathcal{K}}_{\bfp}$ are ordinary quadratic. Consider again the spectral sequence \eqref{eq:spectral-seq} and let $F^\bullet$ be the induced decreasing filtration on $\rr{H}^{nk-1}_{\et,\rr{c}}(\mathcal{K}'_{\bb{F}})$. Since $\bar{\mathcal{K}}^{(i)}$ are smooth proper over both $\bq$ and $\bar{\mathbb{F}}_p$ if $i\geq 1$, we have the isomorphisms
	$\Het{a}(\bar{\mathcal{K}}^{(i)}_{\bfp})\simeq \Het{a}(\bar{\mathcal{K}}^{(i)}_{\bq})$
for $i\geq 1$ and any $a\in \bb{Z}$. By the Picard--Lefschetz formula, we have isomorphisms
	$\Het{a}(\bar{\mathcal{K}}^{(i)}_{\bfp})\simeq \Het{a}(\bar{\mathcal{K}}^{(i)}_{\bq})$
for $0\leq a\leq nk-2$. So $(E_2^{i,nk-1-i})_{\bq}\simeq (E_2^{i,nk-i-1})_{\bfp}$ and
	\begin{equation}\label{eq::weight}
	(E_2^{i,nk-i-1})_{\bfp}=(E_\infty^{i,nk-i-1})_{\bfp}=(E_\infty^{i,nk-1-i})_{\bq}
	\end{equation}
for $i\geq 1$. In other words, the dimensions of the graded pieces
			$\gr_F^i\rr{H}^{nk-1}_{\et,\rr{c}}(\mathcal{K}'_{\bfp})=\gr_F^i\rr{H}^{nk-1}_{\et,\rr{c}}(\mathcal{K}'_{\bq})$
are independent of $p$ when $i\geq 1$. 
\begin{lem}\label{lem:gr-im-beta}
	The graded quotient $\gr_F^i\rr{H}^{nk-1}_{\et,\rr{c}}(\mathcal{K}'_{\bfp})^{G,\chi}$ is pure of Frobenius weight $nk-1-i$ if $1\leq i\leq nk-1$, and is mixed of weight $nk-1$ and $nk-2$ if $i=0$. Moreover, the dimension of $\gr^{W'}_{nk-2}\gr_F^0\rr{H}^{nk-1}_{\et,\rr{c}}(\mathcal{K}'_{\bfp})^{G,\chi}$ is
		$$
			-\dim \rr{H}^0(\bb{G}_{m,\bfp},\Sym^k\Kl_{n+1})+\begin{cases}
				a(k,n+1,p)-a(k,n+1) & 2\mid n,\\
				0		 & 2\nmid nk,\\
				b(k,n+1,p)-b(k,n+1) & 2\nmid n \text{ and } 2\mid k.
			\end{cases}
		$$
\end{lem}
\begin{proof}
	When $1\leq i\leq nk-1$, the graded quotient $\gr_F^i\rr{H}^{nk-1}_{\et,\rr{c}}(\mathcal{K}'_{\bfp})^{G,\chi}=(E_2^{i,nk-i-1})^{G,\chi}$ is pure of Frobenius weight $nk-1-i$, and its dimension is independent of $p$. By the exact sequence \eqref{eq::ell-3}, we deduce for $0\leq i$ that 
		\begin{equation}\label{eq:ell-4}
			\begin{split}
				\dim F^{1+i} \rr{H}^{nk-1}_{\et,\rr{c}}(\mathcal{K}'_{\bfp})^{G,\chi}
				&\leq \dim W'_{nk-2-i}\rr{H}^{nk-1}_{\et,\rr{c}}(\mathcal{K}'_{\bfp})^{G,\chi}\\
				&=\dim W'_{nk-i}\rr{H}^1_{\et,\rr{c}}(\bb{G}_{m,\bfp},\Sym^k\Kl_{n+1}).
			\end{split}
		\end{equation}

	By the long exact sequence \eqref{eq::long-exact-sequence}, the dimensions of the graded pieces of the Frobenius weight $W'$ filtration on $\rr{H}^1_{\et,\rr{c}}(\bb{G}_{m,\bfp},\Sym^k\Kl_{n+1})$ can be calculated in terms of those of $(\Sym^{k}\Kl_{n+1})^{I_{\bar 0}}_{\bar\eta_0}$, $(\Sym^{k}\Kl_{n+1})^{I_{\bar \infty}}_{\bar\eta_\infty}$, and $(\Sym^{k}\Kl_{n+1})^{G_{\rr{geom}}}$. According to Theorems~\ref{thm::local-symk-at-0} and~\ref{thm::local-inv-infty}, as $(\Sym^{k}\Kl_{n+1})^{I_{\bar \infty}}_{\bar\eta_\infty}$ and $(\Sym^{k}\Kl_{n+1})^{G_{\rr{geom}}}$ are pure of weight $nk$, we deduce that
		\begin{equation*}
			\begin{split}
				\dim W'_{nk-i}\rr{H}^1_{\et,\rr{c}}(\bb{G}_{m,\bfp},\Sym^k\Kl_{n+1})
				&=\dim {W}'_{nk-i}(\Sym^{k}\Kl_{n+1})^{I_{\bar 0}}_{\bar\eta_0}.
			\end{split}
		\end{equation*}
	for $1\leq i$. By Remark~\ref{req:inv-0-ind-p}, we deduce that $\dim W'_{nk-1-j}\rr{H}^1_{\et,\rr{c}}(\bb{G}_{m,\bfp},\Sym^k\Kl_{n+1})$ is independent of $p$ when $j\geq 0$.

	Now we replace $p$ by a prime $p'$ at which $\bar{\cc{K}}'$ has a good reduction. In this case
		$\rr{H}^{nk-1}_{\et,\rr{c}}(\mathcal{K}'_{\bar{\bb{F}}_{p'}})\simeq \rr{H}^{nk-1}_{\et,\rr{c}}(\mathcal{K}'_{\bq}),$
	and the Frobenius weight filtration $W'$ on the left-hand side coincides with the weight filtration $W$ on the right-hand side. In particular, we have
		\begin{equation}\label{eq:ell-6}
			\begin{split}
				\gr^{W'}_{nk-1-i}\rr{H}^{nk-1}_{\et,\rr{c}}(\mathcal{K}'_{\bar{\bb{F}}_{p'}})^{G,\chi}=\gr_F^i\rr{H}^{nk-1}_{\et,\rr{c}}(\mathcal{K}'_{\bar{\bb{F}}_{p'}})^{G,\chi}
			\end{split}
		\end{equation}
	for all $0\leq i\leq nk-1$. It follows that 
		\begin{equation}\label{eq:ell-8}
			\begin{split}
				&\dim F^{1+i} \rr{H}^{nk-1}_{\et,\rr{c}}(\mathcal{K}'_{\bar{\bb{F}}_{p'}})^{G,\chi}
				= \dim W'_{nk-2-i}\rr{H}^{nk-1}_{\et,\rr{c}}(\mathcal{K}'_{\bar{\bb{F}}_{p'}})^{G,\chi}\\
				=&\dim W'_{nk-i}\rr{H}^1_{\et,\rr{c}}(\bb{G}_{m,\bar{\bb{F}}_{p'}},\Sym^k\Kl_{n+1})
			\end{split}
		\end{equation}
    for $0\leq i$. Hence, we conclude that \eqref{eq:ell-4} is an equality. In particular, each $\gr_F^i\rr{H}^{nk-1}_{\et,\rr{c}}(\mathcal{K}'_{\bfp})^{G,\chi}$ is pure of Frobenius weight $nk-1-i$ if $1\leq i\leq nk-1$, and $\gr_F^0\rr{H}^{nk-1}_{\et,\rr{c}}(\mathcal{K}'_{\bfp})^{G,\chi}$ is mixed of Frobenius weight $nk-1$ and $nk-2$. 
    
    At last, using \eqref{eq:ell-8} we have
	\begin{equation*}
		\begin{split}
			\dim \gr^{W'}_{nk-2}\gr_F^0\rr{H}^{nk-1}_{\et,\rr{c}}(\mathcal{K}'_{\bfp})^{G,\chi}
                &=\dim \gr^{W'}_{nk-2}\rr{H}^{nk-1}_{\et,\rr{c}}(\mathcal{K}'_{\bfp})^{G,\chi}-\dim \gr^{W'}_{nk-2}F^1\rr{H}^{nk-1}_{\et,\rr{c}}(\mathcal{K}'_{\bfp})^{G,\chi}\\
			    &=\dim \gr^{W'}_{nk} \rr{H}^1_{\et,\rr{c}}(\bb{G}_{m,\bar{\bb{F}}_{p}},\Sym^k\Kl_{n+1})
			    - \dim \gr^{W'}_{nk} \rr{H}^1_{\et,\rr{c}}(\bb{G}_{m,\bar{\bb{F}}_{p'}},\Sym^k\Kl_{n+1}).
		\end{split}
	\end{equation*}
	By \eqref{eq::long-exact-sequence}, \cref{thm::local-symk-at-0} and~\cref{thm::local-inv-infty}, the above dimension coincides with the claimed number.
\end{proof}

	\noindent
	\textbf{(1)} Assume that $nk$ is odd and $p\nmid 2(n+1)$. By \eqref{eq::equivariant-part-p-l-formula}, the representation $(\mathrm{M}^\lambda_{n+1})_{\ell}$ is tamely ramified\footnote{$(\mathrm{M}^\lambda_{n+1})_{\ell}$ is possibly wildly ramified at $p=2$ is because the character $\epsilon\colon I_2\to \{\pm1\}$ has order $2$.}. By \eqref{eq::P-L-formula}, the short exact sequence
	\begin{equation*}
		\begin{tikzcd}
		0\ar[r] & 	\mathrm{H}_{\et}^{n k-1}(\bar{\mathcal{K}}'_{\bar{\bb{F}}_p})		\ar[r,"\iota"]	&	\mathrm{H}_{\et}^{n k-1}(\bar{\mathcal{K}}'_{\bar{\bb{Q}}})		\ar[r,"\gamma"]	&	\oplus_{x\in \Sigma'(p)}\bb{Q}_\ell(-m) \ar[r] & 0
		\end{tikzcd}
	\end{equation*}
	splits, and $\mathrm{H}_{\et}^{n k-1}(\bar{\mathcal{K}}'_{\bar{\bb{F}}_p})$ is orthogonal to $\Delta=\oplus_{x}\bb{Q}_\ell(-m)\delta_x$ in 
	$\mathrm{H}_{\et}^{n k-1}(\bar{\mathcal{K}}'_{\bq})$. By taking the $({G,\chi})$-isotypic component and by doing diagram-chasing argument in diagram \eqref{eq::ell-adic-p-l-formula}, we deduce 
		$$\im(\alpha)^{G,\chi}=\im(\beta)^{G,\chi}\oplus \Delta^{G,\chi}.$$ 
	Since $nk$ is odd, the global monodromy group of $\Kl_{n+1}$ is $\rr{SP}_{n+1}$, which implies that $\rr{H}^0(\bb{G}_{m,\bfp},\Sym^k\Kl_{n+1})=0$. By Lemma~\ref{lem:gr-im-beta} we have $\dim \gr^{W'}_{nk-2}\gr_F^0\rr{H}^{nk-1}_{\et,\rr{c}}(\mathcal{K}'_{\bfp})^{G,\chi}=0$. Hence, $\im(\beta)^{G,\chi}$ is pure of weight $nk-1$ and $\gr^{W'}_{n |\lambda|-1} \im(\beta) ^{G_\lambda\times \mu_{n+1},\chi_n}=\im(\beta) ^{G_\lambda\times \mu_{n+1},\chi_n}$. By \eqref{eq::image-beta}, we can take $\rr{H}=\im(\beta)^{G,\chi}$ and $\rr{E}=\Delta^{G,\chi}$.

\noindent
\textbf{(2)} Assume that $n+1$ is a prime number. Recall that $nk=2m+2$ and $\Delta=\sum_{x\in \Sigma'(p)}\bb{Q}_\ell(-m)\delta_x$ is the subspace of $\mathrm{H}_{\et}^{nk-1}(\bar{\mathcal{K}}'_{\bar{\bb{Q}}})$ generated by vanishing cycle classes. By the Picard--Lefschetz formula, the action of $\sigma\in I_p$ acting on a cohomology class $v\in \Het{nk-1}(\bar{K}_{\bq})$ is given by
$$\sigma(v)=v-(-1)^{m+1}t_{\ell}(\sigma)\sum_{x\in \Sigma'(p)}<v,\delta_x>\delta_x,$$
which implies that $(\sigma-1)^2=0$. It follows that  
	$$(\mathrm{M}_{n+1}^k)_\ell^{I_p} =\bigl(\Delta^\perp\bigr)^{G,\chi} =\im(\alpha)^{G,\chi}\cap \mathrm{H}_{\et}^{nk-1}(\bar{\mathcal{K}}'_{\bar{\bb{F}}_p})\supset \im(\beta)^{G,\chi},$$
and the induced map $\sigma-1\colon (\mathrm{M}^k_{n+1})_\ell\mapsto (\mathrm{M}^k_{n+1})_\ell/\Delta^{G,\chi}$ is zero. It suffices to calculate the dimension of $U=\Delta^{G,\chi}$.

Consider the diagram 
\begin{equation}\label{eq::equivariant-part-p-l-formula-kl3}
	\begin{tikzcd}
		& \im(\beta)^{G,\chi}\ar[d,hook,"i_1"] & & &\\
		0\ar[r]	&	\im(\alpha)^{G,\chi}\cap \mathrm{H}_{\et}^{nk-1}(\bar{\mathcal{K}}'_{\bar{\bb{F}}_p})\ar[r,hook] & \im(\alpha)^{G,\chi} \ar[r,"\gamma"]& C\ar[d,hook,"i_2"]\\
		&	&	& (\oplus_{x\in \Sigma'(p)}\bql(-m-1))^{G,\chi}	&
	\end{tikzcd}
\end{equation}
where $C$ is the image of the map $\gamma$ inside $(\oplus_{x\in \Sigma'(p)}\bql(-m-1))^{G,\chi}$. 

\begin{lem}\label{lem::lemma-ell-adic}
	In the diagram \eqref{eq::equivariant-part-p-l-formula-kl3}, the vertical map $i_1$ is an isomorphism. If $p=2$ and $k$ is even, the cokernel of the vertical map $i_2$ is one-dimensional. Otherwise, the map $i_2$ is an isomorphism. 
\end{lem}
\begin{proof}
By a diagram-chasing argument in \eqref{eq::ell-adic-p-l-formula}, we conclude that $\im(\beta)=\im(\alpha)\cap \mathrm{H}_{\et}^{nk-1}(\bar{\mathcal{K}}'_{\bar{\bb{F}}_p})$. So the map $i_1$ is an isomorphism. 

Consider the subsequent part of the diagram \eqref{eq::ell-adic-p-l-formula}, i.e., 
	$$\begin{tikzcd}
		 &  &  \mathrm{H}_{\et,\rr{c}}^{nk}(\mathcal{K}'_{\bfp}) \ar[d,"\beta'"] \ar[r] & \mathrm{H}_{\et,\rr{c}}^{nk}({\mathcal{K}}'_{\bq})\ar[d,"\alpha'"] & \\
		 \mathrm{H}_{\et}^{nk-1}(\bar{\mathcal{K}}'_{\bar{\bb{Q}}})\ar[r,"\gamma"] &\oplus_{x\in \Sigma'(p)}\bql(-m-1) \ar[r,"\kappa"] & \mathrm{H}_{\et}^{nk}(\bar{\mathcal{K}}'_{\bar{\bb{F}}_p})		\ar[r]&	\mathrm{H}_{\et}^{nk}(\bar{\mathcal{K}}'_{\bar{\bb{Q}}}) \ar[r]& 0
	\end{tikzcd}$$
where the two vertical maps are the surjective edge map from the abutment $\mathrm{H}^{nk}_{\et,\rr{c}}(\mathcal{K}')$ to $E^{0,nk}$. By the same argument for the cohomology of degree $nk-1$, we have $\im(\alpha')=\gr^W_{nk}\mathrm{H}_{\et,\rr{c}}^{nk}({\mathcal{K}}'_{\bar{\bb{Q}}})$, and by \eqref{eq::ell-3} an exact sequence
	\begin{equation}\label{eq::coh-degree-2k}
		 \bb{Q}_\ell(\zeta_p)(-1)^{G,\chi}\to \mathrm{H}_{\et,\rr{c}}^2(\bb{G}_m,\Sym^k\Kl_{n+1})\to \mathrm{H}^{nk}_{\et,\rr{c}}(\mathcal{K}_{\bfp})^{G,\chi}(-1)[\zeta_p]\to \bb{Q}_\ell(\zeta_p)(-2)^{G,\chi},
	\end{equation}

Assume that $\bar{\mathcal{K}}'$ has good reduction at $p'$. Consider the above diagram for $p'$, then $\im(\beta')=\im(\alpha')$. Since $\mathrm{H}_{\et,\rr{c}}^2(\bb{G}_{m,\bar{\bb{F}}_{p'}},\Sym^k\Kl_{n+1})=0$ and $\im(\beta')$ is pure of Frobenius weight $nk$, we have $\im(\beta')^{G,\chi}=0$ by \eqref{eq::coh-degree-2k}. This forces $\im(\alpha')^{G,\chi}=0$, which does not depend on the choice of $p$. 

	If $p\neq 2$ or $k$ is odd, we have $\dim \Hetc{2}(\bb{G}_{m,\bfp},\Sym^k\Kl_{n+1})=\dim \Het{0}(\bb{G}_{m,\bfp},\Sym^k\Kl_{n+1})=0$. So \eqref{eq::coh-degree-2k} implies that $\im(\beta')=0$. Hence, $\kappa=0$, $C=(\oplus_{x\in \Sigma'(p)}\bql(-m-1))^{G,\chi}$, and the two vertical maps $i_1$ and $i_2$ are isomorphisms.

If $p=2$ and $k$ is even, the monodromy group of $\Kl_{n+1}$ is either $\rr{SO}_{n+1}$ or $G_2$. So $(\Sym^k\Kl_{n+1})^{G_{\rr{geom}}}$ is one-dimensional and we have $\dim \Hetc{2}(\bb{G}_{m,\bfp},\Sym^k\Kl_{n+1})=\dim \Het{0}(\bb{G}_{m,\bfp},\Sym^k\Kl_{n+1})=1$. By the property of the spectral sequence \eqref{eq:spectral-seq} and \eqref{eq::coh-degree-2k}, we have 
	$$\gr^W_{nk+2}\mathrm{H}_{\et,\rr{c}}^2(\bb{G}_{m,\bfp},\Sym^k\Kl_{n+1})
	\simeq \gr^W_{nk+2}\mathrm{H}_{\et,\rr{c}}^{nk}(\mathcal{K}'_{\bfp})^{G,\chi}(-1)=\im(\beta')^{G,\chi}(-1).$$
So $\im(\beta')^{G,\chi}$ is one-dimensional. Since $\im(\alpha')^{G,\chi}=0$, the morphism $\kappa$ is surjective. Therefore, in \eqref{eq::equivariant-part-p-l-formula-kl3}, the cokernel of $i_2$ has dimension $1$.
\end{proof}

Notice that $\Delta^{G,\chi}$ is contained in $\im(\beta)^{G,\chi}$, because $\Delta^{G,\chi} \subset \bigl(\Delta^\perp\bigr)^{G,\chi}=\im(\alpha)^{G,\chi}\cap \mathrm{H}_{\et}^{nk-1}(\bar{\mathcal{K}}'_{\bar{\bb{F}}_p})=\im(\beta)^{G,\chi}$. Since $\Delta^{G,\chi}$ is pure of Frobenius weight $nk-2$, we deduce that $\Delta^{G,\chi} \subset W'_{nk-2}\im(\beta)$. By \cref{cor::dimension-mid-de-ell-adic}, the properties of the spectral sequence \eqref{eq:spectral-seq}, and \eqref{eq::image-alpha}, we have 
	\begin{equation*}
		\begin{split}
			\dim \im(\alpha)^{G,\chi}
			&=\dim \Hetc{nk-1}(\mathcal{K}_{\bq})-\sum_{i=1}^{nk-1}\dim (E_\infty^{i,nk-1-i})_{\bq}^{G,\chi}\\
			&=\frac{\binom{k+n}{n}-d(k,n+1)}{n+1}-\dim (\Sym^k\Kl_{n+1})_{\bar \eta_0}^{I_0} - 
				a(k,n+1).
		\end{split}
	\end{equation*} 
As for the dimension of $\im(\beta)^{G,\chi}$, by Proposition~\ref{prop::moments-p-large}, \eqref{eq::image-beta}, and Lemma~\ref{lem:gr-im-beta}, we deduce that 
	\begin{equation*}
		\begin{split}
			\dim \im(\beta)^{G,\chi}
			&=\dim \gr^{W'}_{nk-1}\im(\beta)^{G,\chi}+\dim \gr^{W'}_{nk-2}\im(\beta)^{G,\chi}\\
			&=\frac{\binom{k+n}{n}-d(k,n+1,p)}{n+1}-\dim (\Sym^k\Kl_{n+1})_{\bar \eta_0}^{I_0}-
				a(k,n+1) +\dim \mathrm{H}^2_{\et,\rr{c}}(\mathbb{G}_{m,\bfp},\Sym^k\Kl_{n+1}).
		\end{split}
	\end{equation*}
Notice that we have the identity $d(k,n+1,p)-d(k,n+1)=(n+1)(a(k,n+1,p)-a(k,n+1))$. Then, we get from \cref{lem::lemma-ell-adic} that
	\begin{equation}\label{eq:ell-7}
		\begin{split}
			\dim C
			&\stackrel{(*)}{=}\dim\im(\alpha)^{G,\chi}-\dim\im(\beta)^{G,\chi}\\
			&=a(k,n+1,p)-a(k,n+1)-\dim \Hetc{2}(\bb{G}_{m,\bfp},\Sym^k\Kl_{n+1})\\
			&=\dim \gr^{W'}_{nk-2}\im(\beta)^{G,\chi}\geq \dim \Delta^{G,\chi}\\
            &= \dim \im(\alpha)^{G,\chi}- \dim (\Delta^\perp)^{G,\chi}\\
            &\stackrel{(**)}{=} \dim \im(\alpha)^{G,\chi}- \dim \im(\beta)^{G,\chi}=\dim C,
		\end{split}
	\end{equation}
    where $(*)$ is deduced from the short exact sequence \eqref{eq::equivariant-part-p-l-formula-kl3}, and $(**)$ is because $\Delta^\perp=\im(\beta)$.
	So $\Delta^{G,\chi}=\gr^{W'}_{nk-2}\im(\beta)^{G,\chi}$, and its dimension is $a(k,n+1,p)-a(k,n+1)-\dim \Hetc{2}(\bb{G}_{m,\bfp},\Sym^k\Kl_{n+1})$.
\end{proof}
\begin{rek}\label{rek:weight-monodromy}
  We proved that when $n+1$ is a prime and $p\neq n+1$ the representation $(\mathrm{M}^k_{n+1})_\ell$ satisfies the weight--monodromy conjecture, where the monodromy filtration is given by 
    $$M_{nk}=\Delta^{G,\chi}(-1)\subset M_{nk+1}=\im(\beta)^{G,\chi}(-1)\subset M_{nk+2}=\im(\alpha)^{G,\chi}(-1).$$  In other words, the associated Weil--Deligne representation is pure of weight $nk+1$ (See \cite[p.528]{barnet2014potential} and Section~\ref{sec:W-D-repr} for the definition). 
\end{rek}

\begin{cor}\label{cor:swan-conductor}
	\begin{enumerate}
		\item If $n+1$ is a prime, the exponent of the Artin conductor of $(\mathrm{M}^k_{n+1})_\ell$ at $p$ is $a(k,n+1,p)-a(k,n+1)-1$ if $p=2$ and $k$ even, and is $a(k,n+1,p)-a(k,n+1)$ otherwise.
		\item The exponent of the Artin conductor of $\{\bigl( \mathrm{M}_3^{(2,1)}\bigr)_\ell\}_\ell$ at $p$ is $1$ if $p=2,7$, and is $0$ if $p\neq 2,3,7$.
		\item The exponent of the Artin conductor of $\{\bigl( \mathrm{M}_3^{(2,2)}\bigr)_\ell\}_\ell$ at $p$ is $1$ if $p=2$, and is $0$ if $p\neq 2,3$.
	\end{enumerate}
\end{cor}
\begin{proof}
	For the first case, the Artin conductor of $(\mathrm{M}_{n+1}^k)_\ell$ at $p$ is $\dim C=\dim \im(\alpha)-\dim \im(\beta)$. We get the exact formula by Lemma~\ref{lem:gr-im-beta} and \eqref{eq:ell-7}.
	
	For the second and the third cases, if $p\neq 3$, we can perform the same argument in the above theorem for $\{\bigl( \mathrm{M}_3^{(2,1)}\bigr)_\ell\}_\ell$ and $\{\bigl( \mathrm{M}_3^{(2,2)}\bigr)_\ell\}_\ell$, together with the local behaviors of $\Kl_3^{(2,1)}$  and $\Kl_3^{(2,2)}$ from Propositions~\ref{prop::more-example-at-0} and \ref{prop::more-example-at-infty}. If $p=3$, the representation $\{\bigl( \mathrm{M}_3^{(2,1)}\bigr)_\ell\}_\ell$ is unramified by an analog of Corollary~\ref{cor::dimension_mid_p=3} and Remark~\ref{rek::unramified-p=3}.
\end{proof}

\subsubsection{The \texorpdfstring{$p$}{p}-adic case}
We study the $p$-adic Galois representations $(\mathrm{M}^\lambda_{n+1})_{p}$ in this section. 

\begin{prop}\label{prop::p-adic-galois-representation}
	The $p$-adic representation $(\mathrm{M}^\lambda_{n+1})_{p}$ is de Rham. If $p\nmid n+1$ and $\bar{\cc{K}}'$ has good reduction at $p$. Then, the representation $(\mathrm{M}^\lambda_{n+1})_{p}$ is crystalline and there is an isomorphism of Frobenius modules
	\begin{equation*}
	\begin{split}
		\mathrm{H}^1_{\mathrm{rig,mid}}(\bb{G}_{m}/K,\Kl_{n+1}^{\lambda})\simeq  ((\mathrm{M}^\lambda_{n+1})_{p}\otimes \mathbf{B}_{crys})^{\gal(\bar{\bb{Q}}_p/\bb{Q}_p)}\otimes K.
	\end{split}
	\end{equation*}
\end{prop}

\begin{proof}
As in Section~\ref{subsec::compactification}, we let $\mathcal{K}'$ be $\mathcal{K}$ if $\gcd(n+1,k)=1$ and the blow-up of $\mathcal{K}$ along singular locus otherwise. By \cite[\S3.3(i) and \S3.4]{beilinson2013a}, since the $p$-adic representation $\Het{nk-1}(\bar{\mathcal{K}}'_{\bq},\bar{\bb{Q}}_p)$ comes from a proper smooth variety, it is de Rham. Then we conclude the first assertion by the fact that the subquotient of a de Rham representation remains de Rham.

Now assume that $\gcd(p,n+1)=1$ and $\mathcal{K}'$ has good reduction at $p$. Then by the $p$-adic comparison theorem, the representation $\Het{nk-1}\bigl(\bar{\mathcal{K}}'_{\bar{\bb{Q}}_p}\bigr)$ is crystalline. Therefore, as a subquotient of $\Het{nk-1}\bigl(\bar{\mathcal{K}}'_{\bar{\bb{Q}}_p}\bigr)$, the representation $(\rr{M}_{n+1}^\lambda)_p$ remains crystalline.

Recall that we have an isomorphism
	$$\mathrm{H}^1_{\mathrm{rig,mid}}(\bb{G}_{m}/K,\Kl_{n+1}^{\lambda})\simeq\gr^W_{n|\lambda|-1}\mathrm{H}^{n|\lambda|-1}_{\mathrm{rig,c}}(\mathcal{K}/K)(-1)^{G_\lambda\times \mu_{n+1},\chi_\lambda}[\varpi]$$
from Section~\ref{sec::p-adic-crystal}. We have results similar to those in Lemma~\ref{lem::blowup} by simply replacing étale cohomology with rigid cohomology everywhere. Consider the spectral sequence \cite[Prop.\,8.2.17 and 8.2.18(ii)]{lestum2007}
	$$E_1^{i,j}=\mathrm{H}^j_{\mathrm{rig}}\bigl(\bar{\mathcal{K}}_{\fpp}'^{(i)}/\bb{Q}_p\bigr)\Rightarrow \mathrm{H}^{i+j}_{\mathrm{rig,c}}(\mathcal{K}'_{\fpp}/\bb{Q}_p),$$
and we denote by 
	$$\alpha\colon \mathrm{H}^{n|\lambda|-1}\bigl(\bar{\mathcal{K}}'_{\bfp}/\mathbb{Q}_p\bigr)\to \mathrm{H}^{n|\lambda|-1}(\bar{\mathcal{K}}'^{(1)}_{\bfp}/\mathbb{Q}_p)$$
the differential from $E_1^{0,n|\lambda|-1}$ to $E_1^{1,n|\lambda|-1}$. Since the varieties $\bar{\cc{K}}^{'(i)}$ are smooth proper for all $i\geq 1$, the only contribution of weight $n|\lambda|-1$ to the abutment of the spectral sequence comes from the kernel of $\alpha$. So 
	\begin{equation}\label{eq::kernel-alpha-rigid}
		\mathrm{gr}^W_{n|\lambda|-1}\mathrm{H}^{n|\lambda|-1}_{\mathrm{rig,c}}(\mathcal{K}'_{\bfp}/\mathbb{Q}_p)\simeq \mathrm{gr}^W_{n|\lambda|-1}\ker\alpha.
	\end{equation}
Then use the analog of \eqref{eq::image-alpha}, \eqref{eq::kernel-alpha-rigid} and the $p$-adic comparison theorem, we get the isomorphism of Frobenius modules
	$$\mathrm{H}^1_{\mathrm{rig,mid}}(\bb{G}_{m}/K,\Kl_{n+1}^{\lambda})\simeq  ((\mathrm{M}^\lambda_{n+1})_{p}\otimes \mathbf{B}_{crys})^{\gal(\bar{\bb{Q}}_p/\bb{Q}_p)}\otimes K.$$
\end{proof}

\subsection{Generalities on Deligne--Weil representations}\label{sec:W-D-repr}
We recall the definition of Weil--Deligne (or simply, WD-)representations from \cite{taylor2004galois}. For each prime $p$, there is an exact sequence 
	$$1\to I_p\to \gal(\bq_p / \bb{Q}_p)\simeq \hat{\bb{Z}} \to \gal(\bfp/\fpp)\to 1,$$
where $I_p$ is the \textit{inertia group} at $p$. Moreover, there is a surjection $t_\ell\colon I_p\to \bb{Z}_{\ell}$. Let $ W_{\bb{Q}_p}$ be the \textit{Weil group} of $\mathbb{Q}_{p}$, i.e., the inverse image of the subgroup generated by Frobenius of $\gal(\bfp / \fpp)\simeq \hat{\bb{Z}}$ in $\gal(\bq_p/\bb{Q}_p)$ equipped with the induced topology. 

A WD\textit{-representation} on an $E$-vector space $V$ (with discrete topology) is a pair $(r,N)$, consisting of a representation
	$r\colon W_{\bb{Q}_p}\to \mathrm{GL}(V)$
with open kernel, and an endomorphism $N\in \mathrm{End}(V)$, such that 
	$$r(\phi)Nr(\phi\inv)=p\inv N$$
for every lift $\phi\in W_{\bb{Q}_p}$ of $\frob_p$. It is called \textit{unramified} if $N=0$ and $r(I_p)=1$. It is called \textit{Frobenius semisimple} if $r$ is semisimple. For a lift $\phi$ of Frobenius, we can decompose $r(\phi)=r(\phi)^{ss}r(\phi)^{u}=r(\phi)^{u}r(\phi)^{ss}$, where $r(\phi)^{ss}$ is semisimple and $r(\phi)^{u}$ is unipotent. Any WD-representation $(r,N)$ has a canonical Frobenius semisimplification $(r,N)^{ss}$, by keeping $N$ and $r|_{I_p}$ unchanged, and replacing $r(\phi)$ by $r(\phi)^{ss}$. 

If $\ell\neq p$, there is a canonical way to attach a WD--representation WD$_p(\rho)$ to an $\ell$-adic representation $\rho$ of $\gal(\bq_p,\bb{Q}_p)$ as follows. By Grothendieck's quasi-unipotency theorem, there exists an open subgroup $H$ of $I_p$ of finite index, and a unique nilpotent endomorphism $N$ satisfying $\rho(\sigma)=\exp (t_{\ell}(\sigma) N)$ for all $\sigma\in H$. Let $\phi$ be a lift of $\frob_p$ and $\sigma\in I_p$, one sets
\begin{equation}\label{eq::Deligne-Weil-repr}
	r(\phi^n\sigma):=\rho(\phi^n\sigma)\exp(-t_{\ell}(\sigma)N).
\end{equation}
Notice that $\WD_p(\rho)$ is unramified if and only if $\rho(I_p)=1$, i.e., $\rho$ is unramified.

A WD-representation $(r,N)$ on $\bql$ is called \textit{pure of weight $w$} \cite[p.\,528]{barnet2014potential} if there is an exhaustive and separated ascending monodromy filtration $M_i$ of $V$ such that 
\begin{itemize}
	\item each $F_iV$ is invariant under $r$,
 	\item for each lift $\phi$ of $\frob_p$, all eigenvalues $r(\phi^m)$ on $\gr^M_iV$ are Weil-numbers of weight $m\cdot i$,
  	\item the endomorphism $N$ sends $M_iV$ into $M_{i-2}V$, and induces isomorphisms $N^j\colon \gr^M_{w+j} V\simeq \gr_{w-j}^MV$ for each $j\geq 1$.
\end{itemize}

\subsection{Potential automorphy}\label{subsec:potential-automorphy}

A \textit{weakly compatible system} $\scr{R}=\{\rho_\ell\}$ of $n$-dimensional $\ell$-adic representations of $\gal(\bq/\bb{Q})$ over $\bb{Q}$ and unramified outside $S$  is a family of continuous semisimple representations
	$$\rho_{\ell}\colon \gal(\bq/\bb{Q})\to \mathrm{GL}(V_\ell)$$
for each prime number $\ell$, with the following properties.
\begin{itemize}
	\item[(1).] If $p\not\in S$, for all $\ell\neq p$, the representation $\rho_{\ell}$ is unramified at $p$ and the characteristic polynomial of $\rho_{\ell}(\frob_p)$ is a polynomial with coefficients in $\bb{Q}$, independent of the choice of $\ell$,
	\item[(2).] Each representation $\rho_{\ell}$ is de Rham at $\ell$, and is crystalline if $\ell\not\in S$,
	\item[(3).] The Hodge--Tate weights of $\rho_{\ell}$ are independent of $\ell$.
\end{itemize}
To a weakly compatible system of $\ell$-adic representations, we can attach a partial $L$-function
	\begin{equation*} 
		L^S(\scr{R},s)=\prod_{p\not\in S}\det(1-\rho_\ell(\frob_p)p^{-s})\inv.
	\end{equation*}
Moreover, we call $\scr{R}$ \textit{strictly compatible} if for each $p$, there exists a WD--representation $\WD_p(\scr{R})$ of $W_{\bb{Q}_p}$ over $\bq$ such that for each $\ell\neq p$ and each $\iota \colon \bq\hookrightarrow \bql$, the push forward $\iota \WD_p(\scr{R})$ is isomorphic to $\WD_p(\rho_\ell)^{ss}$. To a strictly compatible family $\scr{R}$, we can attach an $L$-function
	$$L(\scr{R},s)=\prod_{p}\det\left(1-\frob_p\cdot p^{-s}\mid \WD_p(\scr{R})^{I_p,N=0}\right)\inv,$$
which differs from $L^S(\scr{R},s)$ only by finitely many Euler factors at $p\in S$. To describe the complete $L$-function, we still need the gamma factor at $\infty$. Serre conjectured the form of the gamma factors at infinity of the complete $L$-function for a pure motive over $\bb{Q}$ in \cite[\S\, 3]{serre1970facteurslocaux}. We denote by $L_\infty(\scr{R},s)$ the gamma factor associated with $\scr{R}$.

\begin{thm}[{\cite[Thm.\,A\,\&\,Cor.\,2.2]{patrikis_taylor_2015}}]\label{cor::P-T}
Let $\mathscr{R}=\left\{\rho_{\ell}\right\}$ be a weakly compatible system of $n$-dimensional $\ell$-adic representations of $\gal(\bq/\bb{Q})$ defined over $\bb{Q}$ and unramified outside $S$. Suppose that $\{\rho_\ell\}$ satisfies the following properties.
\begin{enumerate}
	\item (Purity) There exists an integer $w$ such that, for each prime $p \notin S$, the roots of the common characteristic polynomial of $\rho_\ell (\mathrm{Frob}_{p} )$ are Weil numbers of weight $w$.
	\item (Regularity) The representation $\rho_\ell$ has $n$ distinct Hodge--Tate weights.
	\item (Odd essential self-duality) Either each $\rho_{\ell}$ factors through a map to $\mathrm{GO}_{n} (\overline{\mathbb{Q}}_{\ell} )$ with even similitude character, or each $\rho_\ell$ factors through a map to $\operatorname{GSp}_{n}  (\overline{\mathbb{Q}}_{\ell} )$ with odd similitude character. Moreover, in either case, similitude characters form a weakly compatible system.
\end{enumerate} 
 Then there exists a finite Galois totally real extension $F/\bb{Q}$, over which all the $\rho_\ell$ become automorphic. Additionally, for any distinct primes $p$ and $\ell$, the WD-representation $\WD_{p}(\mathscr{R})$ of $\gal (\bq_{p} / \mathbb{Q}_{p} )$ associated with $\rho_{\ell}$ is pure of weight $w$. Furthermore, the completed L-function
	$$\Lambda(\mathscr{R}, s)=L_{\infty}(\mathscr{R}, s) \cdot L (\scr{R}, s ) $$
satisfies the functional equation $\Lambda(\mathscr{R}, s)=\varepsilon(\mathscr{R}, s) \Lambda (\mathscr{R}^{\vee}, 1-s )$.
\end{thm}

We are now in a position to prove Theorem~\ref{thm::L-function} using the above theorem of Patrikis--Taylor.

\begin{proof}[Proof of Theorem~\ref{thm::L-function}]
Assume that $k\geq 3$ because by Proposition~\ref{cor::dimension-mid-de-ell-adic} we have $\dim \mathrm{M}_{n+1}^k=0$ when $k\leq 2$. Let $S(k,n+1)$ be the set of primes $p$ such that either $p\mid   n+1 \text{ or } \mathcal{K}'_{\bfp}$ is not smooth.  We start with verifying that the family of semisimplifications of $\ell$-adic Galois representations $\scr{R}=\{(\mathrm{M}^k_{n+1})_\ell^{ss}\}$ is weakly compatible. Indeed, it is sufficient to demonstrate that the three conditions of weakly compatible systems are satisfied for $\{(\mathrm{M}^k_{n+1})_\ell\}$. The first two conditions are readily derived from Theorem~\ref{thm::unramified} and Proposition~\ref{prop::p-adic-galois-representation}. Regarding the third condition, we fix an embedding $\bb{Q}_p\hookrightarrow \bb{C}$ and utilize the $p$-adic comparison theorem to obtain a filtered isomorphism as follows$\colon$
		\begin{equation*}
			\begin{split}
				\Bigl((\mathrm{M}_{n+1}^k)_{p}\otimes \mathbf{B}_{\mathrm{dR}}\Bigr)^{\gal(\bar{\bb{Q}}_p/\bb{Q}_p)}\otimes \bb{C}
				&=\Bigl(\mathrm{gr}^W_{nk+1}\Hetc{nk-1}(\mathcal{K}_{\bar{\bb{Q}}_p},\bb{Q}_p)^{S_k\times \mu_{n+1},\chi_n}(-1)\otimes \mathbf{B}_{\mathrm{dR}}\Bigr)^{\gal(\bar{\bb{Q}}_p/\bb{Q}_p)}\otimes \bb{C}\\
				&=\mathrm{gr}^W_{nk+1}\Hetc{nk-1}(\mathcal{K}_{\bar{\bb{Q}}_p},\bb{Q}_p)^{S_k\times \mu_{n+1},\chi_n}(-1)\otimes \bb{C}
				\simeq (\mathrm{M}_{n+1}^k)_{\dR},
			\end{split} 
		\end{equation*}
As a consequence, the Hodge--Tate weights are independent of $\ell$.

In order to apply the \cref{cor::P-T} to the weakly compatible family $\scr{R}$, it is necessary to verify the conditions stated in \cref{cor::P-T}. The purity is satisfied because the Galois representations $(\mathrm{M}_{n+1}^k)_\ell$, as well as their semisimplifications, are pure of weight $nk+1$. The regularity condition is also fulfilled for pairs $(n+1,k)$ presented in \cref{thm::L-function}, as the multiplicities of Hodge--Tate weights of $(\mathrm{M}_{n+1}^k)_\ell$ (and their semisimplifications) are either $0$ or $1$, by \cref{cor::hodge} and the comparison isomorphism above. 
	
	The odd essential self-duality for  $\scr{R}$ can be verified as follows. The perfect pairing, as described in \cref{prop::perfect-pairing}, indicates that the representations $(\mathrm{M}_{n+1}^k)_\ell$ factor through either $\mathrm{GSP} ((\mathrm{M}_{n+1}^k)_\ell)$ or $\mathrm{GSO}((\mathrm{M}_{n+1}^k)_\ell)$, with a similitude character $\chi_{cyc}^{nk+1}$. By selecting a generator of $\bql(-nk-1)$, we can regard the perfect pairing as a compatible nondegenerate bilinear form on the module $(\mathrm{M}_{n+1}^k)_\ell$ over the group ring of $\gal(\bq/\bb{Q})$, with the involution $g\mapsto \chi_{cyc}^{-nk-1}(g)g\inv$. According to \cite[Thm.\,4.2.1]{serre2018modp}, the semisimplification also factors through either $\mathrm{GSP}$ or $\mathrm{GSO}$, with the same character. This establishes the odd essential self-duality for $\scr{R}$. 

	According to \cref{cor::P-T}, the weakly compatible family $\scr{R}$ is potentially automorphic, and the partial $L$-function $L^{S}(\scr{R},s)$ extends to a meromorphic function on $\bb{C}$ satisfying a functional equation. Observe that the partial $L$-function of $\scr{R}$ agrees with $L^{S}(k,n+1;s)$, as their local factors coincide for each $p\not\in S(k,n+1)$, which can be verified by applying \cref{thm::ell-adic-galois-even}, \cref{rek:weight-monodromy}, and \cite[Lem.\,5.40]{fresan2018hodge}. As a result, the partial $L$-function $L^{S}(k,n+1;s)$ can be completed to 
		$\Lambda_k(s)=L_{\infty}(\mathscr{R}, s) \cdot L\left(\scr{R}, s\right),$
	which extends meromorphically to the whole complex plane and satisfies the claimed functional equation
	in \cref{thm::L-function}.
\end{proof}

\section{Conjectures of Evans type}\label{sec::evans} 
In this section, we prove Theorem~\ref{thm::modular form} with the help of the database LMFDB \cite{lmfdb}. Recall that a modular form will refer to a normalized holomorphic cuspidal Hecke eigenform.

\subsection{Modularity}
\subsubsection{Galois representations attached to modular forms}
One can attach two-dimensional Galois representations to modular forms $f\in S_k(\Gamma_1(N))$ of weight $k$, as constructed in \cite{deligne1971formes,deligne1971formes-b}. More precisely, let $N$ and $k$ be positive integers, $f\in S_k(\Gamma(N))$ a modular form, and $K_f=\bb{Q}(a_f(p))$ the number field generated by the Fourier coefficients of $f$. Then for any place $\lambda$ of $K_f$ over a prime $\ell\nmid N$, there exists a continuous odd irreducible Galois representation
		\begin{equation}\label{eq::galois-repr-modular-form}
				\begin{split}
						\rho_{f,\lambda}\colon \gal(\bq/\bb{Q})\to \mathrm{GL}_2(K_{f,\lambda}),
				\end{split}
		\end{equation}
		unramified if $p\nmid N$, such that for $p\nmid N\ell$, the trace of the arithmetic Frobenius $\frob_p\inv$ at $p$ is $a_p(f)$.

Notice that $\rho_{f,\lambda}$ has conductor $N$ and Hodge--Tate weight $(0,k-1)$. Moreover, it is odd, i.e., the value of $\det(\rho_{f,\lambda})$ at the complex conjugation is $-1$.

Given such a $\rho_{f,\lambda}$, we denote by $\bar \rho_{f,\lambda}\colon \gal(\bq/\bb{Q}) \to \GL_2(\bar{\bb{F}}_\ell)$ its mod $\ell$ reduction. It is obtained by choosing a Galois stable $\mathcal{O}_\lambda$-lattice in $K_{f,\lambda}^2$ and reducing modulo the maximal ideal of $\mathcal{O}_\lambda$, where $\mathcal{O}_\lambda$ is the ring of integers of $K_{\lambda}$. Although $\bar\rho_{f,\lambda}$ depends on the choice of the lattice, its semisimplification does not.

\subsubsection{A special case of modularity}
We recall a weaker version of a theorem by Kisin \cite[Thm.\,1.4.3]{kisin2007modularity}, which says that the $\ell$-adic Galois representations associated with certain two-dimensional motives are modular, i.e., isomorphic to one $\rho_{f,\ell}$ in \eqref{eq::galois-repr-modular-form}. The argument is originally due to Serre \cite[\S4.8]{serre1987representations}, with similar arguments also appearing in \cite[Thm.\,4.6.1]{yun2015galois}.
\begin{thm}\label{thm::modularity}
		Let $\mathrm{M}$ be a pure motive of dimension $2$ over $\bb{Q}$ with coefficients in $\bb{Q}$. Assume that the nonzero Hodge numbers of the de Rham realization of $\mathrm{M}$ are $h^{r,s}=h^{s,r}=1$ for some $0\leq r <s$, and the $\ell$-adic Galois representations $\mathrm{M}_\ell$ are odd and absolutely irreducible. Then for some $N\geq 1$ and some Dirichlet character $\varepsilon\colon \bb{Z}/N\bb{Z}\cros \to \bb{C}\cros$, there exists a modular form $f\in S_{s-r+1}(\Gamma_0(N), \varepsilon)$ such that $\rho_{f,\ell}\simeq \mathrm{M}_{\ell}(r)$.
\end{thm}
\begin{rek}\label{rek:estimation-conductor}
    By (4.8.8) and the last paragraph in p. 216 of \cite{serre1987representations}, the $2$-adic and $3$-adic valuation of $N$ are at most $8$ and $5$ respectively.
\end{rek}

\subsection{Conjectures of Evans type}
In this subsection, we prove Theorem~\ref{thm::modular form} by considering each case individually.
\subsubsection{\texorpdfstring{$\varepsilon$}{epsilon}-factors}
In order to apply Theorem~\ref{thm::modularity} for motives attached to Kloosterman sheaves, it is necessary to check that the associated Galois representations are odd. This is ensured by the following Proposition and Chebotarev's density theorem.
\begin{prop}\label{prop::epsilon_factor}
	If $n|\lambda|$ is even, then 
		$$\varepsilon(\bb{P}^1_{\fpp},j_*\Kl_{n+1}^\lambda)=p^{\frac{n|\lambda|+1}{2}\cdot\dim \rr{H}^1_{\et,\rr{mid}}\bigl(\bb{G}_{m,\bfp},\Kl_{n+1}^\lambda\bigr)}.$$
\end{prop}
\begin{proof}
	By applying Corollary~\ref{cor::perfect-pariring-ell-adic}, we find that the middle $\ell$-adic cohomology $\mathrm{H}^1_{\et,\rr{mid}}\big(\bb{G}_{m,\bfp},\Kl_{n+1}^\lambda\big)$ is a symplectic representation of $\gal(\bfp/\fpp)$. Consequently, the determinant of $\frob_p$ is a power of $p$. Taking into consideration both the dimension and the weight of $\mathrm{H}^1_{\et,\rr{mid}}\big(\bb{G}_{m,\bfp},\Kl_{n+1}^\lambda\big)$, we deduce that
		\begin{equation*}
			\begin{split}
				\varepsilon\big(\bb{P}^1_{\bb{F}_p},j_{*}\Kl_{n+1}^\lambda\big)
				&= \det\bigl(-\frob_p, \Het{1}\bigl(\bb{P}^1_{\bb{F}_p},j_{*}\Kl_{n+1}^\lambda\bigr)\bigr)\\
				&= \det\bigl(\frob_p, \mathrm{H}^1_{\et,\rr{mid}}\bigl(\bb{G}_{m,\bfp},\Kl_{n+1}^\lambda\bigr)\bigr)=p^{\frac{n|\lambda|+1}{2}\dim\mathrm{H}^1_{\et,\rr{mid}}\bigl(\bb{G}_{m,\bfp},\Kl_{n+1}^\lambda\bigr)}. 
			\end{split}
		\end{equation*}
\end{proof}

\subsubsection{\texorpdfstring{$\Sym^4\Kl_3$}{Sym4Kl3}}
The motive $\mathrm{M}_3^4$ is defined over $\bb{Q}$, pure of weight $9$ and equipped with a skew-symmetric perfect pairing, as described in  Proposition~\ref{prop::perfect-pairing}. It has dimension $2$, and the Hodge numbers $h^{p,9-p}$ of its de Rham realization are $1$ if $p=3$ or $6$, and $0$ otherwise by \cite[Thm.\,1.2]{Qin23Hodge}. Our goal is to show that the compatible system of Galois representations $\{ (\mathrm{M}_3^4)_{\ell}(6)\}$ is modular.

\begin{prop}\label{prop::modularity_sym4Kl3}
	There exists a (unique) modular form $f$ in $S_{4}(\Gamma_0(14))$, such that for each prime $p\not \in\{2,7\}$, the Fourier coefficient $a_p(f)$ satisfies
		$$a_f(p)=-\frac{1}{p^3}(m_3^4(p)+1+p^2+p^4),$$
	where $m_3^4(p)$ is the symmetric power moment of $\Sym^4\Kl_3$. In particular, the label of this modular form in the database LMFDB is $14.4.a.b$.
\end{prop}

\begin{proof}
    By \eqref{eq:number-singular-points}, we find that the hypersurface $\mathcal{K}_{\bfp}$ in \eqref{eq:hypersurface} is smooth if the number $d(4,3,p)$ in \cref{nota::multi-indices} is $0$. 
	According to Theorem~\ref{thm::unramified}, we find that the $\ell$-adic representation $(\mathrm{M}_3^4)_{\ell}$ is unramified at $p\neq 2,3,7$. As noted in \cref{rek::unramified-p=3}, the $\ell$-adic representation $(\mathrm{M}_3^4)_{\ell}$ is also unramified at $p=3$, because the middle $\ell$-adic cohomology $\rr{H}^1_{\et,\rr{mid}}\bigl(\bb{G}_{m,\bar{\bb{F}}_3},\Sym^4\Kl_3\bigr)$ has dimension $2$ by \cref{cor::dimension_mid_p=3}. Additionally, \cref{cor:swan-conductor} tells us that the conductor of the compatible family $\{(\mathrm{M}_3^4)_{\ell}\}_\ell$ is $14$.

Since the motive $\mathrm{M}_3^4$ is pure of weight $9$ and its nonzero Hodge numbers are given by $h^{3,6}=h^{6,3}=1$, the Hodge--Tate weight of $ (\mathrm{M}_3^4)_{\ell}(6)$ is $(0,3)$ with multiplicity $1$ by the $p$-adic comparison theorem.
	
According to \cref{prop::epsilon_factor} and Chebotarev density theorem, we find that the determinant $\det\bigl((\mathrm{M}_3^4)_{\ell}(6)\bigr)$ is equal to $\chi_{cyc,\ell}^{3}$. As $\chi_{cyc,\ell}(c)=-1$, the representation $(\mathrm{M}_3^4)_{\ell}(6)$ is odd. Thus, $\{ (\mathrm{M}_3^4)_{\ell}(6)\}$ is modular according to \cref{thm::modularity}. 

By the exact sequence \eqref{eq::long-exact-sequence}, \cref{thm::local-symk-at-0}, and~\cref{thm::local-inv-infty}, we deduce that  
	$$\tr(\frob_p\mid(\mathrm{M}_3^4)_{\ell})=-(m_3^4(p)+1+p+p^2).$$ 
It follows that for any $p\not\in \{2,7,\ell\}$,
	\begin{equation*}
		\begin{split}
			a_f(p)&=\tr(\frob_p\inv\mid(\mathrm{M}_3^4)_{\ell}(6))
			=-\frac{1}{p^3}(m_3^4(p)+1+p+p^2).
		\end{split}
	\end{equation*}

	Now, the remaining task is to identify the modular form. The weight and the level of the corresponding modular form are $k=4$ and $N_f=14$. By computing the Fourier coefficient $a_f(3)$, as detailed in \cref{app::m_3^4(p)}, we find that this modular form $f$ is labeled $14.4.a.b$ in the database LMFDB.
\end{proof}

\subsubsection{\texorpdfstring{$\Sym^3\Kl_4$}{Sym3Kl4}}
The motive $\mathrm{M}_4^3$ is defined over $\bb{Q}$, pure of weight $10$, and equipped with a symmetric perfect pairing. It has dimension $2$, and the nonzero Hodge numbers $h^{p,10-p}$ of its de Rham realization are $1$ if $p=4$ or $6$ by \cite[Thm.\,1.2]{Qin23Hodge}. We aim to demonstrate that the compatible family of Galois representations $\{(\mathrm{M}_4^3)_{\ell}(6)\}$ is modular.

\begin{prop}\label{prop::modularity-Sym3Kl4}
	There exists a (unique) modular form $f$ in $S_{3}\bigl(\Gamma_0(15),\left(\frac{\cdot}{15}\right)\bigr)$ with complex multiplication, such that for each prime $p\not\in\{2,5\}$, the Fourier coefficient $a_f(p)$ satisfies
	\begin{equation*} 
		a_f(p)=-\left(\frac{p}{15}\right)\frac{1}{p^4}(m_4^3(p)+1+p^2+p^3).
   \end{equation*}
Here $m_4^3(p)$ is the symmetric power moment of $\Sym^3\Kl_4$. In particular, the label of the corresponding modular form is $15.3.d.a $ in the database LMFDB.
\end{prop}
\begin{proof}
	Based on \eqref{eq:number-singular-points}, \cref{thm::unramified}, and \cref{thm::ell-adic-galois-even}, we know that $(\mathrm{M}_4^3)_{\ell}$ is unramified if $p\neq 2,3,5$, and tamely ramified if $p=3,5$. Moreover, applying Proposition~\ref{prop::moments-p-large}, we obtain the dimension of the middle $\ell$-adic cohomologies of $\Sym^3\Kl_4$ at $p\neq 2$. Hence, $(\mathrm{M}_4^3)_\ell^{I_p}\simeq\mathrm{H}^1_{\et,\mathrm{mid}}\bigl(\bb{G}_{m,\bfp},\Sym^3\Kl_4\bigr)$ has dimension $1$ when $p=3$ or $5$. This implies that the conductor $N$ of $\{(\mathrm{M}_4^3)_{\ell}\}$ is of the form $2^s\cdot 15$ for some $s\in \bb{Z}_{\geq 0}$.

	\begin{lem}
		For each $\ell\neq 2$, the representation $(\mathrm{M}_4^3)_{\ell}$ is unramified at $p=2$. In particular, the conductor $N$ of $\{(\mathrm{M}_4^3)_{\ell}\}$ is $15$.
	\end{lem}
	\begin{proof}
		At $p=2$, the Swan conductor of $\Sym^3\Kl_4$ is at most $5$.  Since the monodromy group of $\Kl_4$ is $\mathrm{Sp_4}$ and the symmetric power of standard representation of $\mathrm{SP}_4$ remains irreducible, the $0$-th cohomology $\mathrm{H}^0_{\et}(\bb{G}_{m,\bfp},\Sym^3\Kl_4)$ vanishes. By the exact sequence \eqref{eq::long-exact-sequence} and Grothendieck--Ogg--Shafarevich formula, we deduce that 
			$$\dim \mathrm{H}^1_{\et,\mathrm{mid}}\bigl(\bb{G}_{m,\bfp},\Sym^3\Kl_4\bigr)=\sw(\Sym^3\Kl_4)-3-\dim (\Sym^3\Kl_4)^{I_{\bar\infty}}.$$
		As a result, we find that $3\leq \sw(\Sym^3\Kl_4)\leq 5$. 
		
		By \cref{app::m_4^3(2)}, the trace of Frobenius at $p=2$ on $\mathrm{H}^1_{\et,\mathrm{mid}}\bigl(\bb{G}_{m,\bfp},\Sym^3\Kl_4\bigr)$ is
			$$-(m_4^3(p)+1+p^2+p^3+\tr(\frob_p\mid (\Sym^3\Kl_4)^{I_{\bar \infty}}))=-16-\tr(\frob_p\mid (\Sym^3\Kl_4)^{I_{\bar \infty}}).$$
		We proceed by examining each possible value of $\sw(\Sym^3\Kl_4)$ as follows.
	\begin{itemize}
		\item If $\sw(\Sym^3\Kl_4)=5$, the sheaf $\Sym^3\Kl_4$ has only one slope (equal to $1/4$) at $\infty$, which implies that $ (\Sym^3\Kl_4)^{I_{\bar\infty}}=0$. So the dimension of the middle $\ell$-adic cohomology is $2$. As a result, the representation $(\mathrm{M}_4^3)_\ell$ is unramified at $p=2$, thanks to \cref{rek::unramified-p=3}.
		 \item If $\sw(\Sym^3\Kl_4)=4$, then $\dim (\Sym^3\Kl_4)^{I_{\bar\infty}}\leq 1$.  
			 \begin{itemize}
				 \item If $\dim (\Sym^3\Kl_4)^{I_{\bar\infty}}= 1$, the middle $\ell$-adic cohomology of $\Sym^3\Kl_4$ is $0$. The trace of $\frob_2$ on $\mathrm{H}^1_{\et,\mathrm{mid}}\bigl(\bb{G}_{m,\bfp},\Sym^3\Kl_4\bigr)$ is $0$. So we obtain 
					 $$0=-16-\tr(\frob_p\mid (\Sym^3\Kl_4)^{I_{\bar \infty}}).$$
				 This is impossible because $(\Sym^3\Kl_4)^{I_{\bar \infty}}$ is pure of weight $9$ and one-dimensional.
				 \item If $\dim (\Sym^3\Kl_4)^{I_{\bar\infty}}= 0$, the middle $\ell$-adic cohomology is one-dimensional. The trace of $\frob_2$ on $\mathrm{H}^1_{\et,\mathrm{mid}}\bigl(\bb{G}_{m,\bfp},\Sym^3\Kl_4\bigr)$ is $-16$. However, since $\mathrm{H}^1_{\et,\mathrm{mid}}\bigl(\bb{G}_{m,\bfp},\Sym^3\Kl_4\bigr)$ is pure of weight $10$ and one-dimensional, this situation is not possible.
		  \end{itemize}
		  \item If $\sw(\Sym^3\Kl_4)=3$, then $\dim (\Sym^3\Kl_4)^{I_{\bar\infty}}=0$. So the dimension of the middle $\ell$-adic cohomology is $0$. However, the trace of $\frob_2$ on $\mathrm{H}^1_{\et,\mathrm{mid}}\bigl(\bb{G}_{m,\bfp},\Sym^3\Kl_4\bigr)$ is at the same time $0$ and $-16$, which is absurd. 
	\end{itemize}
	In conclusion, we deduce that $\sw(\Sym^3\Kl_4)=5$ and the representation $(\mathrm{M}_4^3)_\ell$ is unramified at $2$. As a consequence, the conductor $N$ is $2^0\cdot 15=15$.
	\end{proof}

By the $p$-adic comparison theorem and our computation of the Hodge numbers for the motive $\mathrm{M}_4^3$, we determine that the Hodge--Tate weight of $\{(\mathrm{M}_4^3)_{\ell}(6)\}$ is $(0,2)$. Observe that these Galois representations $(\mathrm{M}_4^3)_\ell$ are orthogonal, as we have a symmetric perfect paring on the motive $\mathrm{M}_4^3$ given in Proposition~\ref{prop::perfect-pairing}. According to  \cite[1.4(2)]{livne1995motivic}, the associated Galois representation $\{(\mathrm{M}_4^3)_{\ell}(6)\}$ corresponds to a modular form $f=q+\sum_{n=2}^\infty a_nq^n\in S_{3}(15,\varepsilon_f)$ of complex multiplication for some characters $\varepsilon_f\colon \bb{Z}/15\bb{Z}\to \bb{C}\cros$. Moreover, for any $p\not\in \{3,5\}\cup\{\ell\}$ we deduce that
\begin{equation*}
	\begin{split}
		a_f(p)&=\tr(\frob_p\inv\mid(\mathrm{M}_4^3)_{\ell}(6))
		=\det\bigl((\mathrm{M}_3^4)_{\ell}(6)\bigr)\inv\cdot  \tr(\frob_p\mid(\mathrm{M}_4^3)_{\ell}(6))\\
		&=-\varepsilon_f\inv\cdot \frac{1}{p^4}(m_4^3(p)+1+p^2+p^3).
	\end{split}
\end{equation*}

At this point, the remaining task is to identify the modular form. We already know that this modular form has level $15$ and weight $3$.

\begin{lem}
	The character $\varepsilon_f$ is the Legendre symbol $\left(\frac{\bullet}{15}\right)$.
\end{lem}
\begin{proof}
	Using LMFDB, we find that there are only two modular forms with level $15$ and weight $3$. Their characters are both given by the Legendre symbol $\left(\frac{\bullet}{15}\right)$.
\end{proof}

To summarize, we have determined that the desired modular form has weight $k=3$, level $15$, and nebentypus $\varepsilon_f=(\frac{\cdot}{15})$. However, there are still two possibilities in LMFDB. To determine the correct one, we use the Frobenius trace $a_f(2)=-1$ of $(\mathrm{M}_4^3)_{\ell}(6)$ in \cref{app::m_4(p)}. Our search in the LMFDB database yields a unique match: the modular form labeled $15.3.d.b$.
\end{proof}

\subsubsection{\texorpdfstring{$\Sym^4\Kl_4$}{Sym4Kl4}}
The two-dimensional motive $\mathrm{M}_4^4$ is defined over $\bb{Q}$, pure of weight $13$, and equipped with an anti-symmetric perfect self-pairing.

\begin{prop}\label{prop::modularity-sym4kl4}
	There exists a (unique) modular form $f$ in $S_{6}(\Gamma_0(10))$, such that for each prime $p\not\in\{2,5\}$, the Fourier coefficient $a_f(p)$ satisfies
		$$a_f(p)=-\frac{1}{p^4}(m_4^4(p)+1+p^2+p^3+p^4+2p^6),$$
	where $m_4^4(p)$ is the symmetric power moment of $\Sym^4\Kl_4$. In particular, the label of the corresponding modular form is $10.6.a.a.$ in the database LMFDB.
\end{prop}

\begin{proof}
By Theorems~\ref{thm::unramified} the representation $(\mathrm{M}_4^4)_\ell$ is unramified at $p\neq 2,5$, as $\mathcal{K}'_{\bfp}$ in \cref{subsec::compactification} is smooth in this case, i.e., $d(4,4,p)-d(4,4)=0$ in \eqref{eq:number-singular-points}. Moreover, we deduce from Theorem~\ref{thm::ell-adic-galois-even} that the representation $(\mathrm{M}_4^4)_\ell$ is possibly wildly ramified at $p=2$, and is tamely ramified at $p=5$. According to \cref{cor:swan-conductor} and \cref{rek:estimation-conductor}, the conductor of the compatible family $\{(\mathrm{M}_4^4)_\ell\}_\ell$ is of the form $N=2^s\cdot 5$ for some $0\leq s\leq 8$.

	By the Hodge symmetry, there exists an integer $h\in \{0,1,\ldots,6\}$ such that the Hodge numbers $h^{p,13-p}$ of $\mathrm{M}_4^4$ are $1$ if $p=h$ or $13-h$, and $0$ otherwise. Hence, the Hodge--Tate weights of $(\mathrm{M}_4^{4})_\ell(13-h)$ are $(0,13-2h)$. 
	
	The determinant of the Galois representations $(\mathrm{M}_4^4)_{\ell}(13-h)$ is an odd character $\chi_{cyc}^{13-2h}$, according to \cref{prop::epsilon_factor} and the Chebotarev density theorem. Then, the existence of the modular form is provided by \cref{thm::modularity}. It follows that for any $p\not\in \{2,5,\ell\}$,
	\begin{equation*}
		\begin{split}
			a_f(p)&=\tr(\frob_p\inv\mid(\mathrm{M}_4^4)_{\ell}(13-h))
			=-\frac{1}{p^{h}}(m_4^4(p)+1+p^2+p^3+p^4+2p^6).
		\end{split}
	\end{equation*}

	At last, we can compute the Fourier coefficients $a_f(3)=-26\cdot 3^{4-h}$ and $a_f(7)=-22\cdot 7^{4-h}$ by numerical results in \cref{app::m_4^4(p)}. Notice that LMFDB  contains the complete list of modular forms when $k^2\cdot N\leq 40000$. We try $0\leq h\leq 6$ and $0\leq s\leq 8$ one by one. If $(s,h)= (8,0), (8,1), (8,2)$ $(8,3)$, $(8,4)$, $(7,0)$, $(7,1)$, $(7,2)$, $(7,3)$, $(6,0)$ or $(6,1)$, we have $k^2\cdot N> 40000$. In this case, the database LMFDB is insufficient for our needs. So we follow the appendix in \cite{yun2015galois} to compute the space of cuspidal new modular symbols over the finite field $\bb{F}_{p}$. We find that for some primes $p$, the numbers $a_f(p)$ are not roots of the characteristic polynomials of the Hecke operators $T_p$, as shown in the table in \cref{app::m_4^4(p)}. In the remaining possible cases, we find two remaining modular forms in the database of weight $6$ with the prescribed Fourier coefficients. By considering the level, there is only one left with the label $10.6.a.a$. in LMFDB because the other one is of level $400$. 	
	\end{proof}
\begin{rek}\label{rek::hodge-m44}
	We deduced from the proof above that the nonzero Hodge numbers of the de Rham realization of $\mathrm{M}_4^4$ are $h^{4,9} = h^{9,4} = 1$. Although the Hodge numbers weren't computed directly in \cite{Qin23Hodge}, they can still be calculated by following an argument similar to that of $\mathrm{M}_3^{3k}$.
\end{rek}

\subsubsection{\texorpdfstring{$\Sym^3\Kl_5$}{Sym3Kl5}}
The motive $\mathrm{M}_5^3$ is defined over $\mathbb{Q}$, pure of weight $13$, and equipped with an anti-symmetric perfect pairing. It has dimension $2$. According to \cite[Thm.\,1.2]{Qin23Hodge}, the Hodge numbers $h^{p,13-p}$ of its de Rham realization are $1$ if $p=5$ or $8$, and $0$ in other cases. We aim to show that the compatible family of Galois representations $\{(\mathrm{M}_5^3)_\ell (8)\}$ is modular.

\begin{prop}\label{prop::modularity_sym3Kl5}
	There exists a (unique) modular form $f$ in $S_{4}(\Gamma_0(33))$, such that for each prime $p\not\in\{3,11\}$, the Fourier coefficient $a_f(p)$ satisfies
		\begin{equation}\label{eq::fourier-coeff-sym3kl5}
		    a_f(p)=-\frac{1}{p^5}(m_5^3(p)+1+p^2+p^3+p^4+p^6),
		\end{equation}
	where $m_5^3(p)$ is the symmetric power moment of $\Sym^5\Kl_3$. In particular, the label of the corresponding modular form is $33.4.a.b$ in the database LMFDB.
\end{prop}
\begin{proof}
	The representation $(\rr{M}_5^2)_\ell$ is unramified at $p$ if $p\not\in \{3,5,11,\ell\}$ by \cref{thm::unramified} and \eqref{eq:number-singular-points}. According to \cref{cor:swan-conductor}, the conductor of $\{(\mathrm{M}_5^3)_{\ell}(8)\}$ is of the form $3^s\cdot 5^t\cdot 11^e$ for some $0\leq s,e\leq 2$ and $0\leq t$.
	
	\begin{lem}
		If $5\neq \ell$, the representation  $(\mathrm{M}_5^3)_{\ell}$ is unramified at $ 5$. 
	\end{lem}
	\begin{proof}
		At $p=5$, the Swan conductor of $\Sym^3\Kl_5$ is at most $7$. Given that the monodromy group of $\Kl_5$ is $\mathrm{SL_5}$ and the symmetric power of standard representation of $\mathrm{SL}_5$ remains irreducible, the $0$-th cohomology $\mathrm{H}^0_{\et}(\bb{G}_{m,\bfp},\Sym^3\Kl_5)$ vanishes. By the exact sequence \eqref{eq::long-exact-sequence} and the Grothendieck--Ogg--Shafarevich formula, we obtain that 
			$$\dim \mathrm{H}^1_{\et,\mathrm{mid}}\bigl(\bb{G}_{m,\bfp},\Sym^3\Kl_5\bigr)=\sw(\Sym^3\Kl_5)-5-\dim (\Sym^3\Kl_5)^{I_\infty}.$$
		Consequently, we have $5\leq \sw(\Sym^3\Kl_5)\leq 7$. According to the numerical results in \cref{app::m_5^3(p)}, the trace of $\mathrm{H}^1_{\et,\mathrm{mid}}\bigl(\bb{G}_{m,\bfp},\Sym^3\Kl_5\bigr)$ at   $p=5$ is given by
			$$-(m_5^3(p)+1+p^2+p^3+p^4+p^6+\tr(\frob_p\mid (\Sym^3\Kl_5)^{I_{\bar \infty}}))=-4\cdot 5^5-\tr(\frob_p\mid (\Sym^3\Kl_5)^{I_{\bar \infty}}).$$
	Now we proceed by examining each possible value of $\sw(\Sym^3\Kl_5)$ as follows.
	\begin{itemize}
		\item If $\sw(\Sym^3\Kl_5)=7$, the sheaf $\Sym^3\Kl_5$ only has one slope (=$1/5$) at $\infty$. We deduce that the dimension of $(\Sym^3\Kl_5)^{I_{\bar \infty}}$ is $0$. Thus, the dimension of the middle $\ell$-adic cohomology of $\Sym^3\Kl_5$ is $2$. By Remark~\ref{rek::unramified-p=3}, the representation $(\mathrm{M}_5^3)_\ell$ is unramified at $5$.
  		\item If $\sw(\Sym^3\Kl_5)=6$, then $\dim (\Sym^3\Kl_5)^{I_{\bar \infty}}\leq 1$. We consider two cases.  
  		\begin{enumerate}
	  		\item Assume that $\dim (\Sym^3\Kl_5)^{I_{\bar \infty}}= 1$, then the middle $\ell$-adic cohomology vanishes. The trace of Frobenius on $\mathrm{H}^1_{\et,\mathrm{mid}}\bigl(\bb{G}_{m,\bfp},\Sym^3\Kl_5\bigr)$ at  $p=5$ is $0$. So we have 
		  		$$0=-4\cdot 5^5-\tr(\frob_p\mid (\Sym^3\Kl_5)^{I_{\bar \infty}}).$$
	  		Since $(\Sym^3\Kl_5)^{I_{\bar \infty}}$ is pure of weight $12$ and of dimension $1$, this is impossible. 
	  		\item Assume that $\dim (\Sym^3\Kl_5)^{I_{\bar \infty}}= 0$, the middle $\ell$-adic cohomology is one-dimensional. The trace of $\mathrm{H}^1_{\et,\mathrm{mid}}\bigl(\bb{G}_{m,\bfp},\Sym^3\Kl_5\bigr)$ at prime $p=5$ is $-4\cdot 5^5.$ Since $\mathrm{H}^1_{\et,\mathrm{mid}}\bigl(\bb{G}_{m,\bfp},\Sym^3\Kl_5\bigr)$ is pure of weight $13$ and of dimension $1$, it leads to a contradiction.
  		\end{enumerate}
		\item If $\sw(\Sym^3\Kl_5)=5$, then $\dim (\Sym^3\Kl_5)^{I_{\bar \infty}}=0$. So the dimension of the middle $\ell$-adic cohomology of $\Sym^3\Kl_5$ is $0$. The trace of $\mathrm{H}^1_{\et,\mathrm{mid}}\bigl(\bb{G}_{m,\bfp},\Sym^3\Kl_5\bigr)$ at prime $p=5$ is at the same time $0$ and $-4\cdot 5^5$, which is absurd. 
	\end{itemize}
	In conclusion, we have $\sw(\Sym^3\Kl_5)=7$ and the representation $(\mathrm{M}_5^3)_\ell$ is unramified at $5$. 
	\end{proof}
	Consider the Galois representations $(\mathrm{M}_5^3)_{\ell}(8)$. The Hodge--Tate weights of $(\mathrm{M}_5^{3})_\ell(8)$ are $(0,3)$. Their determinants are the odd characters $\chi_{cyc}^{3}$ by \cref{prop::epsilon_factor} and the Chebotarev density theorem. The existence of the modular form is guaranteed by \cref{thm::modularity}. Consequently, we deduce \eqref{eq::fourier-coeff-sym3kl5} for any $p\not\in \{3,11,\ell\}$.

	Thus, the modular form we are seeking has weight $4$, and its level is of the form $N_f = 3^s \cdot 11^e \leq 1089$, with $0 \leq s, e \leq 2$. Furthermore, we compute the Fourier coefficients $a(2) = -1$ and $a(5) = -4$ in \cref{app::m_5^3(p)}. Given this information, there is only one remaining modular form, with weight $4$ and level $N = 33$, which is labeled as $33.4.a.b$ in the LMFDB database.
\end{proof}

\subsubsection{\texorpdfstring{$\Kl_3^{(2,1)}$}{Kl3V21}}
The motive $\mathrm{M}_3^{(2,1)}$ is defined over $\bb{Q}$, pure of weight $9$ and equipped with an anti-symmetric perfect pairing. It has dimension $2$, and the Hodge numbers $h^{p,9-p}$ of its de Rham realization is $1$ if $p=4$ or $5$ and is $0$ otherwise by \cite[Prop.\,5.20]{Qin23Hodge}. We want to show that the compatible family of Galois representations $\{\bigl(\mathrm{M}_2^{(2,1)}\bigr)_{\ell}(5)\}$ is modular.
\begin{prop}\label{prop::modularity_Kl3V21}
	There exists a (unique) modular form $f\in S_{2}(\Gamma_0(14))$, such that for each prime $p\not\in\{2,3,7\}$, the Fourier coefficient $a_p$ satisfies
	\begin{equation}\label{eq:fourier-coeff-Kl3v21}
	    a_f(p)=-\frac{1}{p^4}\Bigl(m_3^{(2,1)}(p)+p+p^2+p^3\Bigr),
	\end{equation}
	where $m_3^{(2,1)}(p)$ is the moment of the sheaf $\Kl_3^{(2,1)}$. In particular, this modular form is labeled $14.2.a.a$ in the database LMFDB.
\end{prop}
\begin{proof}
The sheaf $\Kl_{3}^{{(2,1)}}$ is tamely ramified at $0$ and wildly ramified at $\infty$. By Grothendieck--Ogg--Shafarevich formula \eqref{eq:GOS-formula}, the dimension of the $\ell$-adic cohomology is equal to the Swan conductor at $\infty$. Similar to \cref{cor::dimention_p=3}, since $\Kl_3^{(2,1)}\subset \Kl_3^{\otimes 4}$ and $\zeta_3$ acts on $(\Kl_3^{\otimes 4})_{\eta_\infty}$ freely, we can compute that the Swan conductor of $\Kl_3^{(2,1)}$ at $\infty$ is $5$ when $p=3$. By the exact sequence \eqref{eq::long-exact-sequence}, \cref{prop::more-example-at-0}, and~\cref{prop::more-example-at-infty}, we have
	$$\dim \mathrm{H}^1_{\et,\mathrm{mid}}\bigl(\bb{G}_{m,\bfp},\Kl_3^{(2,1)}\bigr) =	
	\begin{cases}
		2 & p\neq 2,7\\
		1 & p=2,7
	\end{cases}$$
and
	$$\tr\bigl(\frob_p,\bigl(\mathrm{M}^{(2,1)}_3\bigr)_\ell(4)\bigr)=-p^{-4}(m_3^{(2,1)}(p)+p+p^2+p^3).$$
By \cref{rek::unramified-p=3} and \cref{cor:swan-conductor}, the representation $\bigl( \mathrm{M}_3^{(2,1)}\bigr)_\ell$ is unramified at  $p\not\in \{2, 7,\ell\}$ and the conductor of the compatible family $\{\bigl( \mathrm{M}_3^{(2,1)}\bigr)_\ell\}_\ell$ is $14$.

Using \cref{prop::epsilon_factor} and the Chebotarev density theorem, the determinant of $\bigl(\mathrm{M}_3^{(2,1)}\bigr)_\ell(5)$ is $\chi_{cyc}\inv$, which is odd. Then Theorem~\ref{thm::modularity} shows the existence of the modular form and we deduce \eqref{eq:fourier-coeff-Kl3v21} for any $p\neq 2,7,\ell$.

At last, by computations of Fourier coefficients $a_f(p)$ in \cref{app::m3v21m3v22} for $p\leq 23$, we can determine the modular form in the database LMFDB.
\end{proof}

\subsubsection{\texorpdfstring{$\Kl_3^{(2,2)}$}{Kl3V22}}
The  motive $\mathrm{M}_3^{(2,2)}$ is defined over $\bb{Q}$, pure of weight $13$ and equipped with an anti-symmetric perfect pairing in Proposition~\ref{prop::perfect-pairing}.  
\begin{prop}\label{prop::modularity_Kl3V22}
	There exists a (unique) modular form $f=q+\sum_{n\geq 2}^\infty a_nq^n \in S_{4}(\Gamma_0(6))$, such that for each prime $p\not\in\{2,3\}$, the Fourier coefficient $a_f(p)$ satisfies
		\begin{equation*}
		    a_f(p)=-\frac{1}{p^5}\Bigl(m_3^{(2,2)}(p)+p^2 + p^3 + 2 p^4 + 2 p^6\Bigr),
		\end{equation*}
	where $m_3^{(2,2)}(p)$ is the moment of the sheaf $\Kl_3^{(2,2)}$. In particular, this modular form is labeled $6.4.a.a$ in the database LMFDB, the same as the modular form corresponding to $\Sym^6\Kl_2$.
\end{prop}
\begin{proof}
The sheaf $\Kl_{3}^{{(2,2)}}$ is tamely ramified at $0$ and wildly ramified at $\infty$. By \cref{prop::more-example-at-0}, \cref{prop::more-example-at-infty}, and the long exact sequence \eqref{eq::long-exact-sequence}, we obtain that 
	$$\dim \mathrm{H}^1_{\et,\mathrm{mid}}\bigl(\bb{G}_{m,\bfp},\Kl_3^{(2,2)}\bigr) =\begin{cases}
		2 & p\neq 2,3\\
		1 & p=2
	\end{cases}$$
and
	$$\tr\bigl(\frob_p,\bigl(\mathrm{M}^{(2,2)}_3\bigr)_\ell(5)\bigr)=-p^{-5}(m_3^{(2,2)}(p)+p^2 + p^3 + 2 p^4 + 2 p^6)$$ 
if $p\neq 2,3$.
By Remark~\ref{rek::unramified-p=3}, the set of bad primes $S$ is a subset of $\{2,3\}$, and $\dim\bigl( \mathrm{M}_3^{(2,2)}\bigr)_\ell=2$. According to \cref{thm::unramified} and \cref{cor:swan-conductor}, the Galois representation $\bigl( \mathrm{M}_3^{(2,2)}\bigr)_\ell$ is tamely ramified at $p=2$ and its Artin conductor at $p=2$ is $1$. As consequence of \cref{rek:estimation-conductor}, the conductor of $\{\bigl( \mathrm{M}_3^{(2,2)}\bigr)_\ell\}_\ell$ is of the form $N=2\cdot 3^s$ for some $0\leq s\leq 5$.  

By the Hodge symmetry, there exists an integer $h\in \{0,1,\ldots,6\}$ such that the Hodge numbers $h^{p,13-p}$ are $1$ if  $p=h$ or $13-h$, and are $0$ otherwise. Hence, the Hodge--Tate weights of $\bigl(\mathrm{M}_3^{(2,2)}\bigr)_\ell(13-h)$ are $(0,13-2h)$.

By Proposition~\ref{prop::epsilon_factor} and Chebotarev density theorem, we have $\det\bigl(\mathrm{M}_3^{(2,2)}\bigr)_\ell=\chi_{cyc}^{-13}$. Thus, the determinant of $\bigl(\mathrm{M}_3^{(2,2)}\bigr)_\ell(13-h)$ is $\chi_{cyc}^{13-2h}$, which is an odd character. Therefore, Theorem~\ref{thm::modularity} guarantees the existence of a modular form of weight $14-2h$ and of level $2\cdot 3^s$ such that $\bigl(\mathrm{M}_3^{(2,2)}\bigr)_\ell(13-h)\simeq \rho_{f,\ell}$. It follows that for any $p\not\in S\cup\{\ell\}$,
\begin{equation*}
	\begin{split}
		a_f(p)&=\tr(\frob_p\inv\mid((\mathrm{M}_3^{(2,2)})_\ell(13-h))
		=-\frac{1}{p^h}\bigl(m_3^{(2,2)}(p)+p^2 + p^3 + 2 p^4 + 2 p^6\bigr).
	\end{split}
\end{equation*}

To determine the modular form, we use a similar argument to that in \cref{prop::modularity-sym4kl4}. 
We test the combinations $0\leq h\leq 6$ and $0\leq s\leq 5$ one by one. If $(s,h)= (5,0), (5,1)$ or $(5,2)$, we compute the space of cuspidal new modular symbols over the finite field $\bb{F}_{p}$. We find that for some primes $p$, the numbers $a_f(p)$ are not roots of the characteristic polynomials of the Hecke operators $T_p$, as shown in the table in \cref{app::m3v21m3v22}. Therefore, $(s,h)\neq (5,0), (5,1)$ or $(5,2)$, and we proceed to search the modular form within LMFDB. The remaining modular form has weight $4$ and level $6$, corresponding to $(s,h)=(1,5)$ in this case. 
\end{proof}

\begin{rek}\label{rek::hodge-m322}
	The nonzero Hodge numbers of the de Rham realization of $ \mathrm{M}_3^{(2,2)}$ are $h^{5,8}=h^{8,5}=1$. We cannot calculate these using the methods for \cite[Thm.\,1.2]{Qin23Hodge}, as the nilpotent part of the local monodromy of the connection $\Kl_3^{(2,2)}$ at $0$ is not a direct sum of Jordan blocks of different sizes (there are two blocks of size $4$).
\end{rek}

\subsubsection{A conjecture}\label{sec::a-conjecture}

One interesting corollary of \cref{prop::modularity_Kl3V22} is that for $p\nmid 6$, the moments of the sheaves $\Sym^6\Kl_2$ and $\Kl_3^{(2,2)}$, as they both correspond to the modular form with label $6.4.a.a.$ As a direct consequence, we have the identity
\begin{equation}\label{eq:identical-moments}
    m_3^{(2,2)}(p)-p^3m_2^6(p)=-2p^6-2p^4-p^2.
\end{equation}
Moreover, we have isomorphisms of $\ell$-adic Galois representations $(\mathrm{M}_2^6)_\ell(-3)\simeq \bigl(\mathrm{M}_3^{(2,2)}\bigr)_\ell$, which leads us the following conjecture. 
\begin{conj}\label{conj:identical-motives}
	The two motives $\mathrm{M}_2^6(-3)$ and $\mathrm{M}_3^{(2,2)}$ are isomorphic.
\end{conj}

\begin{appendices}
	
\section{Computation of moments}\label{appx:computation}

In this article, we used several numerical results computed using the software Sagemath \cite{sagemath}. This appendix explains the algorithms and all codes can be found on \href{https://yichenqin.net}{my web page}. We fix an embedding $\iota\colon \bql \hookrightarrow \bb{C}$ and identify $\ell$-adic numbers with their images in $\bb{C}$ via $\iota$.

\subsection{Computations of \texorpdfstring{$m_3^k(p)$}{m3k(p)}, \texorpdfstring{$m_{3}^{(2,1)}(p)$}{m321(p)} and  \texorpdfstring{$m_3^{{(2,2)}}(p)$}{m322(p)} }

\subsubsection{\texorpdfstring{$m_3^k(p)$}{m3k(p)}}\label{app::m_3^4(p)}
    
    For a prime number $p$, after Deligne \cite[Somme.\,Trig.]{SGA41/2}, we know that for each $a\in \bb{F}_p\cros$, there exist $3$ algebraic numbers $\alpha_a,\beta_a$ and $\gamma_a$, of absolute value $p$, such that
            $s_1(a)=\alpha_a+\beta_a+\gamma_a=\Kl_3(a;p)$
    and
            $s_3(a)=\alpha_a\cdot \beta_a\cdot \gamma_a=p^3.$
    Then the degree two elementary symmetric polynomials are
            $$\begin{aligned}
                s_2(a):=\alpha_a\beta_a+\beta_a\gamma_a+\gamma_a\alpha_a
                &=p^3(\alpha_a\inv+\beta_a\inv+\gamma_a\inv)
                =p(\bar{\alpha}_a+\bar{\beta}_a+\bar{\gamma}_a)=p\cdot\overline{\Kl_3(a;p)}.
            \end{aligned}$$
    
    The $k$-th symmetric power moments of $\Kl_3$ are integers of the form
        $$m_3^k(p):=\sum_{a\in \bb{F}_p\cros}\sum_{i+j+k=k}\alpha_a^i\beta_a^j\gamma_a^k,$$
    which can be computed using the value of elementary symmetric polynomials. 
    For example, the $3$-rd, $4$-th and $6$-th symmetric power moments can be computed by
        $$m_3^3(p)=\sum_a (s_1(a)^3 - 2 s_1(a) s_2(a) + p^3),$$
        $$m_3^4(p)=\sum_a (s_1(a)^4 - 3 s_1(a)^2 s_2(a) + s_2(a)^2 + 2 p^3 s_1(a) ),$$
        $$m_3^6(p)=\sum_a (s_1(a)^6 - 5 s_1(a)^4 s_2 + 6 s_1(a)^2 s_2(a)^2 - s_2(a)^3 + 4 p^3s_1(a)^3  - 6 p^3 s_1(a) s_2(a) + p^6,$$
    respectively. 
  	Hence, we obtain from \eqref{eq::long-exact-sequence} that 
        $$a_3^4(p)=-\frac{1}{p^3}(m_3^4(p)+1+p^2+p^4)$$
    is the trace of the middle cohomology $\rr{H}^1_{\et,\rr{mid}}(\bb{G}_{m,\bfp},\Sym^4\Kl_3)$. We list some numerical results as follows.
        \begin{center}
            \begin{tabular}{@{}l cccccccc} \toprule
                &\multicolumn{2}{c}{Primes} \\ \cmidrule(r){2-3}
                & 3&5\\ \midrule
                $m_3^3(p)$&-10&~\\
                $a_3^4(p)$&-2&-12\\
                $m_3^6(p)$&-820&~\\ \bottomrule
            \end{tabular}
        \end{center}

\subsubsection{\texorpdfstring{$m_{3}^{(2,1)}(p)$}{m321(p)} and  \texorpdfstring{$m_3^{{(2,2)}}(p)$}{m322(p)}       }\label{app::m3v21m3v22}
   By \eqref{eq::sheaf-repr-sl3}, the moment of $\Kl_{\mathrm{SL}_3}^{V_{2,1}}$ is the difference of the moment of $\Sym^2\Kl_3\otimes \wedge^2\Kl_3$ and that of $\Kl_3(-3)$. Hence, we obtain
            $$m_3^{(2,1)}(p)=\sum_a (s_1(a)^2s_2(a) - s_2^2(a) -p^3 s_1(a)).$$
    Let $a_3^{(2,1)}(p)$ be the traces of the middle cohomology $\rr{H}^1_{\et,\rr{mid}}(\bb{G}_{m,\bfp},\Kl_3^{(2,1)})$, which can be computed by 
        $$a_3^{(2,1)}(p)=-p^{-4}(m_3^{{(2,1)}}(p)+p+p^2+p^3).$$

    As for $\Kl_3^{(2,2)}$, we conclude similarly from \eqref{eq::sheaf-repr-sl3} that the moment of $\Kl_{3}^{(2,2)}$ is
            $$m_3^{(2,2)}(p)=\sum_a\q{ (s_1(a)^2-s_2(a))(s_2(a)^2-p^3s_1(a)) - p^3s_1(a)\cdot s_2(a)}.$$
    Let $a_3^{(2,2)}(p)$ be the traces of the middle cohomology $\rr{H}^1_{\et,\rr{mid}}(\bb{G}_{m,\bfp},\Kl_3^{(2,2)})$. Then we obtain
            $$a_3^{(2,2)}(p)=-p^{-5}(m_3^{({2,2})}(p)+p^2 + p^3 + 2 p^4 + 2 p^6).$$
    Some numerical results are as follows.

        \begin{center}
            \begin{tabular}{@{}l cccccccc} \toprule
                &\multicolumn{7}{c}{Primes} \\ \cmidrule(r){2-8}
                & 5&7 & 11 & 13 & 17 & 19 & 23\\ \midrule
                $a_3^{(2,1)}$ &0& &0&-4&6&2&0\\
                $a_3^{(2,2)}$ &6&-16&12&38&-126&20&168\\ \bottomrule
            \end{tabular}
        \end{center}
	Moreover, we compute the space of cuspidal new modular symbols over some finite fields $\bb{F}_\ell$ and verify whether the prescribed traces are roots of the characteristic polynomials of $T_p$. Below are some numerical results.
\begin{center}
    \begin{tabular}{@{}l cccccccc} \toprule
        \multicolumn{1}{c}{Level $N$}
        &\multicolumn{1}{c}{weight $k$}
        &\multicolumn{1}{c}{Prime $p$} 
        &\multicolumn{1}{c}{Finite field $\bb{F}_\ell$}
        &\multicolumn{1}{c}{$T_p(a_f(p))$}\\ \midrule
        $2\cdot 3^5$ & $14$ & $5$ & $\bb{F}_{23}$ &1\\
        $2\cdot 3^5$ & $12$ & $5$ & $\bb{F}_{23}$ &-1\\
        $2\cdot 3^5$ & $10$ & $5$ & $\bb{F}_{13}$ &5\\ 
       \bottomrule
    \end{tabular}
\end{center}

 \subsection{Computation of \texorpdfstring{$m_4^3(p)$}{m43(p)} and \texorpdfstring{$m_4^4(p)$}{m44(p)}}\label{app::m_4(p)}
    \subsubsection{\texorpdfstring{$m_4^3(2)$}{m34(2)}}\label{app::m_4^3(2)}
    Here we compute the third symmetric power moment at $p=2$. Using Sagemath \cite{sagemath}, we know that $\Kl_4(1;2)=1$ and $\Kl_4(1;4)=11$. Let $\alpha_1,\ldots,\alpha_5$ be the eigenvalues of $\frob_2$ acting on $(\Kl_4)_{\bar {1}}$ and let $s_1,\ldots,s_4$ be the elementary symmetric polynomials on $\alpha_i$. By the definition of $\Kl_4$, we have
        \begin{equation*}
        \begin{split}
            s_1=\sum \alpha_i=-\Kl_4(1;2)=-1,
        \end{split}
        \end{equation*} 
        \begin{equation*}
        \begin{split}
            s_1^2-2s_2=\sum \alpha_i^2=-\Kl_4(1;4)=-11.
        \end{split}
        \end{equation*}
    Therefore, $s_1=-1$ and $s_2=6$. Moreover, since $\det\Kl_4=E(-6)$, we have $s_4=\prod \alpha_i=p^{6}.$ Noticing that $\alpha_i\cdot \bar{\alpha}=p^3$, we have $s_3=p^3\bar{s}_1-8.$
    Then, the moments can be computed by
        \begin{equation*}
        \begin{split}
            m_4^3(2)=\sum_{i,j,k}\alpha_i\alpha_j\alpha_k=s_1^3-2s_1s_2+s_3=3.
        \end{split}
        \end{equation*}
    It follows that 
        $$a_4^3(2)=-\frac{1}{p^4}(m_4^3(p)+1+p^2+p^3)=-1.$$
    
    \subsubsection{\texorpdfstring{$m_4^3(p)$ and $m_4^4(p)$}{m43(p) and m44(p)}}\label{app::m_4^4(p)}
     
    Let $\alpha_1(a),\ldots,\alpha_4(a)$ be the eigenvalues of $\frob_p$ acting on $(\Kl_4)_{\bar {a}}$ for $a\in \bb{F}_p\cros$ and by $s_1(a),\ldots,s_4(a)$ the elementary symmetric polynomials on $\alpha_i(a)$. By the definition of $\Kl_4$, we have
    \begin{equation*}
    \begin{split}
        s_1(a)=\sum \alpha_i(a)=-\Kl_4(a;p) \text{ and } s_1(a)^2-2s_2(a)=\sum \alpha_i(a)^2=-\Kl_4(a;p^2).
    \end{split}
    \end{equation*} 
Furthermore, since $\det\Kl_4=E(-6)$, we have $s_4(a)=\prod \alpha_i=p^{6}.$
Noticing that $\alpha_i(a)\cdot \bar{\alpha}_i(a)=p^3$, we have $ s_3(a)=p^3\overline{s_2(a)} .$
Then, the moments can be computed as
    \begin{equation*}
    \begin{split}
        m_4^4(p)
        &=\sum_a (s_1(a)^4 - 3 s_1(a)^2 s_2(a) + s_2(a)^2 + 2p^3 s_1(a) \overline{s_1(a)} - p^6).
    \end{split}
    \end{equation*}
At last, the traces of the middle cohomology $\rr{H}^1_{\et,\rr{mid}}(\bb{G}_{m,\bfp},\Sym^4\Kl_4)$ are
    $$a_4^4(p)=-\frac{1}{p^4}(m_4^4(p)+1+p^2+p^3+p^4+2p^6).$$
Some numerical results are listed below.

            \begin{center}
                \begin{tabular}{@{}l cccc} \toprule
                    &\multicolumn{3}{c}{Primes} \\ \cmidrule(r){2-4}
                    &2 & 3 &  7 \\ \midrule
                    $a_4^3(p)$ &-1 & & \\
                    $a_4^4(p)$ & &-26& -22\\ \bottomrule
                \end{tabular}
            \end{center}

Similar to the end of \cref{app::m3v21m3v22}, we list some numerical results when $N\cdot k^2\geq 40000$.
\begin{center}
    \begin{tabular}{@{}l cccccccc} \toprule
        \multicolumn{1}{c}{Level $N$}
        &\multicolumn{1}{c}{weight $k$}
        &\multicolumn{1}{c}{Prime $p$} 
        &\multicolumn{1}{c}{Finite field $\bb{F}_\ell$}
        &\multicolumn{1}{c}{$T_p(a_f(p))$}\\ \midrule
        $2^8\cdot 5$ & $14$ & $7$ & $\bb{F}_{11}$ &3\\
        $2^8\cdot 5$ & $12$ & $3$ & $\bb{F}_{13}$ &10\\
        $2^8\cdot 5$ & $10$ & $7$ & $\bb{F}_{11}$ &3\\ 
        $2^8\cdot 5$ & $8$ & $7$ & $\bb{F}_{11}$ &5\\
        $2^8\cdot 5$ & $6$ & $7$ & $\bb{F}_{11}$ &4\\
        $2^7\cdot 5$ & $14$ & $3$ & $\bb{F}_{17}$ &8\\
        $2^7\cdot 5$ & $12$ & $7$ & $\bb{F}_{17}$ & 8\\
        $2^7\cdot 5$ & $10$ & $3$ & $\bb{F}_{11}$ &5\\
        $2^7\cdot 5$ & $8$ & $7$ & $\bb{F}_{11}$ &3\\
        $2^6\cdot 5$ & $14$ & $3$ & $\bb{F}_{17}$ &3\\
        $2^6\cdot 5$ & $12$ & $3$ & $\bb{F}_{29}$ &2\\
        \bottomrule
    \end{tabular}
\end{center}

\subsection{Computation of \texorpdfstring{$m_5^3(p)$}{m53(p)}}\label{app::m_5^3(p)}
    Let $\alpha_1(a),\ldots,\alpha_5(a)$ be the eigenvalues of $\frob_p$ acting on $(\Kl_5)_{\bar {a}}$ for $a\in \bb{F}_p\cros$ and by $s_1(a),\ldots,s_5(a)$ the elementary symmetric polynomials on $\alpha_i(a)$. By the definition of $\Kl_5$, we have
            \begin{equation*}
            \begin{split}
                s_1(a)=\sum \alpha_i(a)=\Kl_5(a;p) \text{ and } s_1(a)^2-2s_2(a)=\sum \alpha_i(a)^2=\Kl_5(a;p^2).
            \end{split}
            \end{equation*} 
    Furthermore, since $\det\Kl_5=E(-10)$, we have $ s_5(a)=\prod \alpha_i=p^{10}.$ Because $\alpha_i(a)\cdot \bar{\alpha}_i(a)=p^4$, we have $s_3(a)=p^2\overline{s_2(a)}$ and $s_4=p^6\overline{s_1(a)}.$
    Then, the moments can be calculated as
            \begin{equation*}
            \begin{split}
                m_5^3(p)=\sum_{a\in \bb{F}_p\cros}\sum_{i\leq j\leq k}\alpha_i(a)\alpha_j(a)\alpha_k(a)=\sum_a s_1(a)^3-2s_1(a)s_2(a)+3s_3(a).
            \end{split}
            \end{equation*}
    At last, the traces of middle cohomology  $\rr{H}^1_{\et,\rr{mid}}(\bb{G}_{m,\bfp},\Sym^3\Kl_5)$ are 
            $$a_5^3(p)=-\frac{1}{p^5}(m_5^3(p)+1+p^2+p^3+p^4+p^6).$$
    The values of moments and Frobenius traces at $p=2,5$ are listed below.
            \begin{center}
                \begin{tabular}{@{}l ccc} \toprule
                    &\multicolumn{2}{c}{Primes} \\ \cmidrule(r){2-3}
                    & 2 & 5\\ \midrule
                    $m_5^3(p)$  &   -61 &   3901\\
                    $a_5^3(p)$  &   -1  &   -4\\ \bottomrule
                \end{tabular}
            \end{center}

\end{appendices}

\bibliography{l-function}
\bibliographystyle{abbrv}

\end{document}